\documentclass{amsart}
\usepackage{amssymb,amsmath,amsthm}
\usepackage[inline,shortlabels]{enumitem}
\usepackage{eucal}
\usepackage{mathabx}
\usepackage{xcolor}
\usepackage{hyperref}

\hypersetup{
    colorlinks,
    linkcolor={red!75!black},
    citecolor={green!65!black},
    urlcolor={magenta!75!black}
}

\usepackage[backend=biber,sorting=nyt,style=alphabetic,minalphanames=3,maxbibnames=99]{biblatex}

\DeclareNameAlias{author}{family-given}
\let\oldsubs=\tocsubsection
\renewcommand{\tocsubsection}[2]{\hspace{2em}\oldsubs{#1}{#2}}

\theoremstyle{definition}
\newtheorem{Theorem}{Theorem}[section]
\newtheorem{Proposition}[Theorem]{Proposition}
\newtheorem{Lemma}[Theorem]{Lemma}
\newtheorem{Corollary}[Theorem]{Corollary}
 
\theoremstyle{definition}
\newtheorem{Definition}[Theorem]{Definition}
\newtheorem{Fact}[Theorem]{Fact}
\newtheorem{Conjecture}{Conjecture}
\newtheorem{Example}[Theorem]{Example}
\newtheorem{Remark}[Theorem]{Remark}
\newtheorem{Question}[Theorem]{Question}
\newtheorem{Notation}[Theorem]{Notation}
\newtheorem{Theorem1}{Theorem}

\def\strok{\restriction}

\def\Ind#1#2{#1\setbox0=\hbox{$#1x$}\kern\wd0\hbox to 0pt{\hss$#1\mid$\hss}
\lower.9\ht0\hbox to 0pt{\hss$#1\smile$\hss}\kern\wd0}

\def\ind{\mathop{\mathpalette\Ind{}}}

\def\notind#1#2{#1\setbox0=\hbox{$#1x$}\kern\wd0
\hbox to 0pt{\mathchardef\nn=12854\hss$#1\nn$\kern1.4\wd0\hss}
\hbox to 0pt{\hss$#1\mid$\hss}\lower.9\ht0 \hbox to 0pt{\hss$#1\smile$\hss}\kern\wd0}

\def\dep{\mathop{\mathpalette\notind{}}}

\newcommand{\wor}{\perp\hspace*{-0.4em}^{\mathit w}}
\newcommand{\nwor}{\not\perp\hspace*{-0.4em}^{\mathit w}}

\newcommand{\Mon}{\mathfrak C}
\newcommand{\Moneq}{\mathfrak C^{eq}}
\newcommand{\tp}{\mathrm{tp}}

\newcommand{\dcleq}{\mathrm{dcl^{eq}}}

\newcommand{\conv}{\mathrm{conv}}

\newcommand{\Aut}{\mathrm{Aut}}
\def\pp{\mathbf{p}}

\def\lif{\rightarrow}

\DeclareMathOperator{\supin}{in}
\DeclareMathOperator{\fin}{fin}
\DeclareMathOperator{\lgh}{lg}

\usepackage{etoolbox}
\usepackage{orcidlink}
\usepackage[margin=30mm]{geometry}

\title{Countable models of weakly quasi-o-minimal theories II}  
\author[S.\ Moconja]{Slavko Moconja\ \orcidlink{0000-0003-4095-8830}}
\address[S.\ Moconja]{University of Belgrade, Faculty of mathematics, Belgrade, Serbia}
\email{slavko@matf.bg.ac.rs}

\author[P.\ Tanovi\'c]{Predrag Tanovi\'c\ \orcidlink{0000-0003-0307-7508}}
\address[P.\ Tanovi\'c]{Mathematical Institute of the Serbian Academy of Sciences and Arts, Belgrade, Serbia}
\email{tane@mi.sanu.ac.rs}

\begin{document} 

\begin{abstract} 
 We confirm Martin's conjecture for a broad subclass of weakly quasi-o-minimal theories. 
\end{abstract}

\maketitle

\section{Introduction and preliminaries}

This paper continues \cite{MT} and \cite{MTwqom1}, where we started developing the classification theory for countable models of weakly quasi-o-minimal theories.
In \cite{MT}, the case of binary such theories was resolved, and in \cite{MTwqom1} we started the "Non-structure" part in the general case, that is, searching for conditions that imply the existence of the maximum number of countable models. There, we found one such condition: the existence of a definable family of semiintervals of $(\Mon,<)$ that contains a shift. In this paper, we study relatively definable families of semiintervals defined on the locus of a weakly o-minimal type and, motivated by the work of Baizhanov and Kulpeshov \cite{Baizhanov2006}, introduce the notion of simplicity of these semiintervals. The main result is the following theorem.

\begin{Theorem1}\label{Theorem main}
Let $T$ be a countable, weakly quasi-o-minimal theory. Suppose that at least one of the following two conditions holds.
\begin{enumerate}[\hspace{10pt}(1)]
    \item $T$ does not have simple semiintervals;
    \item There exists a non-convex type $p\in S_1(T)$. 
\end{enumerate}
Then $I(T,\aleph_0)=2^{\aleph_0}$.
\end{Theorem1}

The simplicity of semiintervals is equivalent to the nonexistence of shifts in small weakly quasi-o-minimal theories. However, it is a strong technical property that guaranties the exceptional simplicity of the relatively definable binary subsets of the locus of a complete 1-type. Together with the convexity of types, it will suffice to develop the "Structure" part of the classification theory for almost $\aleph_0$-categorical weakly quasi-o-minimal theories, for which we verify Martin's conjecture, a strengthening of Vaught's conjecture. 

\begin{Theorem1}\label{Theorem_AAC}
Martin's conjecture holds for almost $\aleph_0$-categorical weakly quasi-o-minimal theories.    
\end{Theorem1}

In Section \ref{Section_(R)}, we prove that every quasi-o-minimal theory satisfies the following condition:
\begin{enumerate}
    \item[(R)] For all $A$ and all $p\in S_1(A)$, every relatively $A$-definable equivalence relation on the locus of $p(\Mon)$ is relatively $0$-definable.
\end{enumerate}
We prove that (R) is also satisfied if a weakly quasi-o-minimal theory has few countable models and, in addition, is rosy or has a finite convexity rank. Then we prove that any weakly quasi-o-minimal theory that satisfies (R) and has simple semiintervals must be binary. Since, by \cite[Lemma 8.1]{MT}, any binary (weakly quasi-o-minimal) theory with few countable models is almost $\aleph_0$-categorical, as a corollary of Theorem \ref{Theorem_AAC} we obtain the following theorem. 

\begin{Theorem1}\label{Theorem_bin_ros_quasio} 
Martin's conjecture holds for weakly quasi-o-minimal theories with an additional feature: binary, rosy, quasi-o-minimal, or has finite convexity rank.
\end{Theorem1}

An example of a weakly o-minimal $\aleph_0$-categorical theory that is not binary (even not $n$-ary for all $n\in\mathbb N$) was found by Herwig et al.\ in \cite{Herwig}. Therefore, not all almost $\aleph_0$-categorical theories with few countable models are binary, so
Theorem \ref{Theorem_AAC} is stronger than Theorem \ref{Theorem_bin_ros_quasio}. However, there are weakly o-minimal theories with 3 countable models that are not almost $\aleph_0$-categorical, see Examples 7.9 and 7.10 in \cite{MTwqom1}, so Vaught's conjecture for weakly (quasi-) o-minimal theories remains widely open; we speculate on the further work in Section \ref{Section 7}.

\smallskip 
The paper is organized as follows. 
In Section \ref{Section 2}, we introduce and study $\mathbf p$-semiintervals, where $\mathbf p=(p,<)$ is a weakly o-minimal pair over $A$. 
In Proposition \ref{Prop_simple_semiintervals_basic} we prove that $\mathbf p$ has simple semiintervals if and only if every relatively $A$-definable subset of $p(\Mon)^2$ is relatively defined by a Boolean combination of formulae $x\leqslant y, y\leqslant x$, and $\{E(x,y)\mid E\in \mathcal E_p\}$,  where $\mathcal E_p$ is the set of all relatively $A$-definable equivalence relations on $p(\Mon)$. Then we deduce that the simplicity of semiintervals is a property of the type $p\in S(A)$: if some weakly o-minimal pair $(p,<)$ has simple semiintervals, then all such pairs do. We also show that $a\ind b\,({\dcleq(aA)\cap \dcleq(bA)})$ holds whenever $\tp(a/A)$ and $\tp(b/A)$ are weakly o-minimal, and at least one of them has simple semiintervals. This is used to explain that, assuming the simplicity of semiintervals, the binary forking dependence is measured by elements of a certain meet-tree.

The main result of Section \ref{Section 3} is Theorem \ref{Thm_noshifts_iff_simplesemiintervals}, in which we prove that a small weakly quasi-o-minimal theory $T$ has simple semiintervals if and only if there is no definable family of semiintervals of $(\Mon,<)$ that contains a shift. This together with \cite[Theorem 1]{MTwqom1} immediately proves one part of Theorem \ref{Theorem main}: if $T$ does not have simple semiintervals, then $I(T,\aleph_0)=2^{\aleph_0}$.

In Section \ref{Section_(R)}, we study condition (R). Using the techniques developed in Section 2, we prove that for all weakly quasi-o-minimal theories with simple semiintervals, condition (R) is equivalent to the binarity of the theory. This together with \cite[Theorem 1]{MT} suffices to confirm Vaught's conjecture for the classes of theories mentioned in Theorem \ref{Theorem_bin_ros_quasio}. To confirm Martin's conjecture for them, we need the results of
Section \ref{Section 5}. There, using the simplicity of semiintervals, we prove the second part of Theorem \ref{Theorem main}: all types $p\in S_1(T)$ in a weakly quasi-o-minimal theory with few countable models are convex. Then, in Proposition \ref{Proposition mixed kind is trivial} we prove a general result which relates the definabilty of a convex weakly o-minimal type to its triviality (in any $T$). In particular, Proposition \ref{Proposition mixed kind is trivial} implies that every definable convex weakly o-minimal 1-type over a model (in any $T$) is trivial; this generalizes \cite[Proposition 5.18]{MTwqom1}, which covers the case of complete 1-types over an o-minimal model $M$. These results are used essentially in Section \ref{Section Martin}, where we prove Theorems \ref{Theorem_AAC} and \ref{Theorem_bin_ros_quasio}. 

Section \ref{Section 7} is devoted to the questions that can be relevant for further work.

\subsection{Preliminaries} 

Throughout the paper $T$ is an arbitrary, possibly multi-sorted, complete first-order theory, and $\Mon$ is its monster model. 
We will rely on the notation and results of \cite{MTwqom1}, occasionally recalling the definitions and properties of the concepts introduced in that paper.

When working with a definable, dense linear order $(D,<)$ it is convenient to use $T^\mathrm{eq}$-sorts of the Dedekind completion, $\overline D$ (see \cite{Mac}). But in this paper, we deal with semiintervals in arbitrary linear orders, so the following complete extension of $(D,<)$ turns out to be more convenient:
\begin{itemize}
    \item $\overline D^{\,in}$ denotes the set of all initial parts of $(D,<)$.
\end{itemize}
Note that $(\overline D^{\,in}, \subset)$ is a complete linear order in which $(D,<)$, and hence also $(\overline D,<)$, can be embedded. In general,  $\overline D^{\,in}$ and  $\overline D$ are not isomorphic orders; for example,  when $D$ is dense, the intervals $(-\infty, d)$ and $(-\infty, d]$ are distinct members of $\overline D^{\,in}$. Usually, we will identify an element $d\in D$ with $(-\infty, d]\in \overline D^{\,in}$; in that case, since $I<d$ and $I\subset (-\infty, d]$ are equivalent for all $I\in \overline D^{\,in}$ and all $d\in D$, $x\longmapsto (-\infty, x]$ defines an embedding of $(D,<)$ in $(\overline D^{\,in},\subset)$. Recall that for $X\subseteq D$, $\supin_D(X)=\{d\in D\mid (\exists x\in X)\,D\leqslant x\}$ is the smallest initial part that contains $X$; if the meaning of $D$ is clear from the context, then we write $\supin (X)$. Also, if $P\subseteq D$, then $\supin_P(X)$ is the corresponding initial part of $(P,<)$. Similarly, $\fin(X)$ is the smallest final part that contains $X$. Two convex sets, $X_1$ and $X_2$, are equal if and only if $\supin(X_1)=\supin(X_2)$ and $\fin(X_1)=\fin(X_2)$.  

The sorts in $\overline D^{\,in}$ are defined routinely. For each formula $\phi(x,y)$ such that for all $b\in \Mon^{|y|}$ the set $\phi(\Mon,b)$ is an initial part of $D$, let $E_{\phi}$ be the equivalence relation on $\Mon^{|x|+|y|}$ defined by \ $(x,y)\sim (x',y')\Leftrightarrow \phi(\Mon,y)=\phi(\Mon,y')$. Elements of the sort $\Mon^{|x|+|y|}/E_{\phi}$ are $E_{\phi}$-classes and for each class $e$ the set $e_\phi=\phi(\Mon,y)$ is constant for all $(x,y)\in e$; note that $e_\phi$ is an initial part of $(D,<)$. Usually, we will not distinguish between $e$ and $e_\phi$; this is because $e$ and $e_{\phi}$ are interdefinable, that is, $e\in \Moneq$ is a canonical parameter of $e_{\phi}$. The uniformly definable family
$\{e_{\phi}\mid e\in\Mon^{|x|+|y|}/E_{\phi}\}$ is a sort in $\Moneq$; sorts of this kind will be called $\overline D^{\,in}$-sorts and will be identified with $\{\phi(\Mon, y)\mid y\in \Mon^{|y|}\}$. By a (relatively) definable function into $\overline D^{\,in}$ we always mean a function into one of these sorts;
if $f:C\to \overline D^{\,in}$ is such, then $f(x)<d$, $f(x)\subset (-\infty, d]$, and $d\notin f(x)$ have the same meaning.

\smallskip 

Recall the following convention for weakly o-minimal pairs $(p,<)$ over $A$: We will always assume that the order $<$ is relatively defined by the formula $x<y$, and when we say that the order $(D,<)$ is a definable extension of $(p(\Mon),<)$, we mean that $D=\theta(\Mon)$ for some $\theta(x)\in p$ such that $x<y$ defines a linear order on $D$. A similar convention applies when $E$ is a relatively definable equivalence relation on $p(\Mon)$.

\section{$\mathbf p$-semiintervals}\label{Section 2}

Let $(D,<)$ be a linear order. 
Recall from \cite[Section 6]{MTwqom1} that $S\subseteq D$ is a semiinterval of $(D,<)$ if $S$ is convex and $a=\min S$ exists; in that case, we denote it by $S_a$ to emphasize $a=\min S_a$. Also, recall that a family $\mathcal S=(S_x\mid x\in D)$ of semiintervals of $D$ is $A$-definable if $(D,<)$ is $A$-definable and there exists some $L_A$-formula $S(x,y)$ such that $S_a=S(a,\Mon)$ for all $a\in D$.
The family $\mathcal S$ is {\it monotone} if $\supin (S_x)\subseteq\supin (S_y)$ for all elements $x\leqslant y$ of $D$. In this section, we study relatively definable semiintervals mainly of the loci of weakly o-minimal types.

\begin{Definition}\label{Definition_P_semiinterval} 
Let $\mathbf p=(p,<)$ where $p\in S(A)$ and $<$ is a relatively definable order on $p(\Mon)$.
\begin{enumerate}[\hspace{10pt} (a)]
\item The set $S\subset p(\Mon)$ is a {\em $\mathbf p$-semiinterval} if it is convex in $(p(\Mon),<)$, $a=\min S$ exists, and $S$ is relatively $Aa$-definable within $p(\Mon)$; in that case, we write $S^{\mathbf p}_a$ instead of $S$; $S^{\mathbf p}_a$ is a {\em bounded $\mathbf p$-semiinterval} if $S^{\mathbf p}_a<b$ for some $b\in p(\Mon)$.
\item A {\em family of $\mathbf p$-semiintervals} is a collection $(S^{\mathbf p}_x\mid x\in p(\Mon))$ such that $S^{\mathbf p}_x=S(x,p(\Mon))$ holds for some $L_A$-formula $S(x,y)$ and all $x\in p(\Mon)$. 
\item The family of $\mathbf p$-semiintervals $(S^{\mathbf p}_x\mid x\in p(\Mon))$ is {\em monotone} if $\supin_{p(\Mon)} (S^{\mathbf p}_x)\subseteq \supin_{p(\Mon)} (S^{\mathbf p}_y)$ for all $x<y$ realizing $p$. 
\end{enumerate}
\end{Definition}

\begin{Remark}
If $p$ is a convex weakly o-minimal type, then every $\mathbf p$-semiinterval is definable.     
\end{Remark}

Let $\mathcal S=(S(x,D)\mid x\in D)$ be an $A$-definable family of semiintervals of $(D,<)$. 
\begin{itemize}
    \item For $D'\subseteq D$, denote $\mathcal S^{D'}=(S(x,D')\mid x\in D')$; we say that $\mathcal S^{D'}$ is the family of semiintervals of $D'$ that is {\em induced by} $\mathcal S$.
    \item If $p\in S_x(A)$ and if $p(x)$ implies $x\in D$, then by $\mathcal S^{\mathbf p}$ we denote the induced family of $\mathbf p$-semiintervals: $\mathcal S^{\mathbf p}=(S(x,p(\Mon))\mid x\in p(\Mon))$, where $\mathbf p=(p,<)$.
\end{itemize}

\begin{Lemma}\label{Lemma_new_defextof_seminiterval}
Suppose that $\mathcal S_0=(S(x,p(\Mon))\mid x\in p(\Mon))$ is a family of $\mathbf p$-semiintervals, where the formula $S(x,y)$ and the pair $\mathbf p=(p,<)$ are defined over $A$. Then there exists an $A$-definable extension, $(D,<)$, of $(p(\Mon),<)$ such that $\mathcal S= (S(x,D)\mid x\in D)$ is a family of semiintervals of $(D,<)$; in particular $\mathcal S$ induces $\mathcal S_0$. Moreover, if $\mathcal S_0$ is monotone, then the order $(D,<)$ can be found such that the family $\mathcal S$ is also monotone. 
\end{Lemma}
\begin{proof} 
The fact that $S(x,p(\Mon))$, for $x\models p$, is a convex subset of $p(\Mon)$ whose minimum is $x$, is expressed by the conjunction of the following $\tp$-universal sentences\footnote{For the definition of tp-univarsal properties and sentences see \cite[Subsection 2.1]{MTwqom1}.}:
\begin{enumerate}[\hspace{10pt}(1)]
    \item \ $(\forall x\models p)\ S(x,x)$,
    \item \ $(\forall x,y\models p)(S(x,y)\lif x\leqslant y)$, and
    \item \ $(\forall x,y,z\models p)(x\leqslant z<y\land S(x,y)\lif S(x,z))$.
\end{enumerate}
Also, ``$x<y$ relatively defines a linear order on $p(\Mon)$" is a $\tp$-universal property, so by compactness applied simultaneously to both this and to (1)--(3), we find $\theta(x)\in p(x)$ such that $(\theta(\Mon),<)$ is an $A$-definable extension of $(p(\Mon),<)$, and $\mathcal S=(S(x,\theta(\Mon))\mid x\in\theta(\Mon))$ is an $A$-definable family of semiintervals of $(\theta(\Mon),<)$. Clearly, $\mathcal S^\pp=\mathcal S_0$.
For the ``moreover" part, note that the monotonicity of $\mathcal S$ is expressed by:
\begin{enumerate}[\hspace{10pt}(1)]\setcounter{enumi}{3}
    \item \ $(\forall x,y,z\models p)(x<y\land S(x,z)\lif z<y\lor S(y,z))$.
\end{enumerate}
By compactness, the formula $\theta(x)\in p(x)$ can be chosen such that the family $\mathcal S$ is also monotone.
\end{proof}
 
\noindent {\bf Convention.} \ From now on, when we denote some family of $\mathbf p$-semiintervals by $\mathcal S^{\mathbf p}$, where $\mathbf p=(p,<)$, we implicitly assume that $\mathcal S$ is an $A$-definable family of some (unspecified) order $(D,<)$, which is an $A$-definable extension of $(p(\Mon),<)$. Unless otherwise stated, $S(x,y)$ denotes a formula that defines $\mathcal S$. 
 
\begin{Remark}\label{Remark definable extension of a definable extension}
Let $\mathcal S=(S_x\mid x\in D)$ be an $A$-definable family of semiintervals of $(D,<)$ and let $p\in S_x(A)$ be consistent with $x\in D$. Denote $\mathbf p=(p,<)$.
If $D'\subseteq D$ is an $A$-definable set such that $(x\in D')\in p$, then the induced family $\mathcal S^{D'}=(S_x\cap D'\mid x\in D')$ of semiintervals of $(D',<)$ also induces the family $\mathcal S^{\mathbf p}$. Moreover, if $\mathcal S$ is monotone, then so are $\mathcal S^{D'}$ and $\mathcal S^{\mathbf p}$. 
\end{Remark}

Recall that powers of a family of semiintervals $\mathcal S=(S_x\mid x\in D)$
are defined recursively by: $S_x^1=S_x$ and $S^{n+1}_x=\bigcup_{t\in S_x^n}S_t$; $\mathcal S^n=(S_x^n\mid x\in D)$; if $\mathcal S$ is $A$-definable, then so is each $\mathcal S^n$. Also, recall that $S_a$ is an $\mathcal S$-shift, where $a\in D$, if the sequence of powers $(S_a^n\mid n\in\mathbb N)$ strictly $\subset$-increases or, equivalently, if it is not eventually constant. By induction, one easily sees that $S^{n+m}_x=\bigcup_{t\in S_x^n}S_t^m$; we will freely use this fact further on. 

\begin{Lemma}\label{Lemma extending family of semiintervals}
Suppose that $(D',<)$ is an $A$-definable order, $D\subseteq D'$ is $A$-definable, 
and $\mathcal S=(S_x\mid x\in D)$ is an $A$-definable family of semiintervals of $(D,<)$. Then:
\begin{enumerate}[\hspace{10pt}(a)]
\item  There exists an $A$-definable family $\mathcal R=(R_x\mid x\in D')$ of
semiintervals of $(D',<)$ such that $(\mathcal R^n)^D=\mathcal S^n$ for all $n\in\mathbb N$.
\item If $\mathcal S$ is monotone, then $\mathcal R$ as in (a) can also be found to be monotone. 
\end{enumerate} 
Moreover, the following holds in (a)-(b): for all $x\in D'$, $R_x$ is an $\mathcal R$-shift if and only if $S_x$ is an $\mathcal S$-shift.
\end{Lemma}
\begin{proof}
(a) For $x\in D'$ define: $R_x=\{x\}$ if $x\notin D$, and $R_x=\conv(S_x)$ if $x\in D$, where the convex hull is taken in $(D',<)$. We claim that for all $n\in\mathbb N$: $R_x^n=\{x\}$ if $x\notin D$, and $R_x^n=\conv(S_x^n)$ if $x\in D$. The first part of this claim is immediate and the second is proved by induction on $n$. The case $n=1$ is clear: $R_x\cap D=\conv (S_x)\cap D=S_x$. The induction step is justified by: 
$$R_x^{n+1}=\bigcup_{t\in R_x}R^n_t=\bigcup_{t\in R_x\cap D}R_t^n\cup \bigcup_{t\in R_x\smallsetminus D}R_t^n=\bigcup_{t\in S_x}\conv(S_t^n)\cup \bigcup_{t\in R_x\smallsetminus D}\{t\}=$$
$$=\bigcup_{t\in S_x}\conv(S_t^n)\cup (\conv(S_x)\smallsetminus D)=\bigcup_{t\in S_x}\conv(S_t^n)=\conv (\bigcup_{t\in S_x}S_t^n)=\conv \left(S_x^{n+1}\right).$$
Here, $\bigcup_{t\in S_x}\conv(S_t^n)= \conv (\bigcup_{t\in S_x}S_t^n)$ is justified as follows: For all $t\in S_x$, we have $S_t^n\cap  S_x\neq 0$. This implies $\conv(S_t^n)\cup \conv(S_x)=\conv(S_t^n\cup S_x)$, which together with $\conv(S_x)\subseteq \bigcup_{t\in S_x}\conv(S_t^n)$ justifies the first two equalities in
$$\bigcup_{t\in S_x}\conv(S_t^n)=\bigcup_{t\in S_x}(\conv(S_t^n)\cup \conv (S_x))=
\bigcup_{t\in S_x}\conv(S_t^n\cup S_x)=\conv(\bigcup_{t\in S_x}(S_t^n\cup S_x))=\conv(\bigcup_{t\in S_x}S_t^n).$$
The third equality holds because the family $(S_t^n\cup S_x\mid t\in S_x)$ has the finite intersection property, while the fourth is a consequence of $S_x\subseteq \bigcup_{t\in S_x}S_t^n$. This proves the claim. 
Now, $(\mathcal R^n)^D=\mathcal S^n$ easily follows: for $x\in D$ we have $D\cap R_x^n=D\cap\conv(S_x^n)=S_x^n$, completing the proof of (a). 

\smallskip
(b) Suppose that $\mathcal S$ is monotone. For $x\in D'$ define: $$\text{$R_x=\{x\}$ \ if \ $x<D$ \ and \ $R_x=\{x\}\cup\conv(\bigcup\limits_{t\in D,\ t\leqslant x}S_t\cap[x,+\infty)_{D'})$ \ if \ $x\nless D$}.$$
Clearly, $\mathcal R$ is monotone and for all $x\in D$ we have $R_x=\conv(S_x)$ and $R_x\cap D=S_x$. By induction on $n$, we prove $R_x^n\cap D=S^n_x$ for $x\in D$: 
$$R^{n+1}_x\cap D=\bigcup_{t\in R_x}R^n_t\cap D=\bigcup_{t\in \conv(S_x)}(R_t^n\cap D)=\bigcup_{t\in S_x}(R_t^n\cap D)=\bigcup_{t\in S_x} S_t^n=S^{n+1}_x.$$
Here, the third equality holds because $\mathcal R^n$ is monotone and $S_x$ is cofinal in $\conv(S_x)$. 
\end{proof}

\begin{Lemma}\label{Lemma monotone family not Ep determined is a shift}
Suppose that $\mathcal S=(S_x\mid x\in D)$ is a monotone $A$-definable family of semiintervals of $(D,<)$, $p(x)\in S(A)$ is consistent with $x\in D$, $\mathbf p=(p,<)$, $k\geqslant 2$, $a\models p$, and $(S_a^k)^{\mathbf p}\neq S^{\mathbf p}_a$. Then $S_a$ is an $\mathcal S$-shift.
\end{Lemma}
\begin{proof}Clearly, $(S_a^k)^{\mathbf p}\neq S^{\mathbf p}_a$  implies that $S^{\mathbf p}_a$ is a proper initial part of $(S_a^k)^{\mathbf p}$, so $S^{\mathbf p}_a<x$ holds for some $x\in (S_a^k)^{\mathbf p}$.
Hence, we can recursively define a sequence $(a_n\mid n\in \omega)$ of realizations of $p$ such that $a_0=a$, $a_{n+1}\in (S_{a_n}^k)^{\mathbf p}$ \ and \ $S^{{\mathbf p}}_{a_n}<a_{n+1}$ \ for all $n\in \omega$; this is possible since $a_{n+1}\in (S_a^k)^{\mathbf p}$. Note that $a_{n+1}\models p$ and $S^{{\mathbf p}}_{a_n}<a_{n+1}$ imply $S_{a_n}<a_{n+1}$. Next, for each $n$ choose $b_n\in S_{a_n}$ such that $a_{n+1}\in S_{b_n}^{k-1}$; this is possible as $a_{n+1}\in (S_a^k)^{\mathbf p}$. Then $(a_0,b_0,a_1,b_1,\ldots,a_n)$ is an $\mathcal S^{k-1}$-path, so we conclude $a_n\in S^{2n(k-1)}_{a}$. 
By induction, we prove $S_a^n<a_{n}$ for all $n\geqslant 1$. The case $n=1$ is clear, so assume that $S_a^{n}<a_{n}$ holds for some $n$. Then we have
\[\supin_D (S_a^{n+1})=\supin_D (\bigcup_{t\in S_a^n}S_t)= \bigcup_{t\in S_a^n}\supin_D (S_t)\subseteq \supin_D (S_{a_{n}})<a_{n+1};\] 
Here, the first equality follows by the definition of $S_a^{n+1}$, the second is obvious, and the inclusion is a consequence of the monotonicity of $\mathcal S$: by the induction hypothesis, for each $t\in S_a^n$ we have $t<a_{n}$, so $\supin_D (S_t)\subseteq \supin_D (S_{a_{n}})$ follows by the monotonicity; the inequality follows from $S_{a_{n}}<a_{n+1}$. This proves $S_a^{n+1}<a_{n+1}$. 

Now,  $a_n\in S^{2n(k-1)}_{a} \smallsetminus S_a^n$ implies $S_a^n\subset S_a^{2n(k-1)}$, so the sequence $(S_a^n\mid n\in \mathbb N)$ is not eventually constant. Therefore, it is strictly increasing and $S_a$ is an $\mathcal S$-shift. 
\end{proof} 

\begin{Remark} The definition of shifts is motivated by the notion of quasi-succesors introduced by Alibek, Baizhanov, and Zambarnaya in  \cite{Alibek}. 
Let $\mathcal S=(S(D,x)\mid x\in D)$ be an $A$-definable family of semiintervals of $(D,<)$. The formula $S(x,y)$ is called a {\it quasi-successor on $D'\subseteq D$}
if $\mathcal S^{D'}=(S(D',x)\mid x\in D')$ is a monotone family of semiintervals of $(D',<)$ such that $(S^{D'}_x)^2\neq S^{D'}_x$ for all $x\in D'$; $S(x,y)$ is a {\it quasi-successor on $p\in S(A)$} if it is so on $p(\Mon)$. 
Note that if $S(x,y)$ is a quasi-successor on $p$, then Lemma \ref{Lemma monotone family not Ep determined is a shift} implies that $S_a$ is an $\mathcal S$-shift for all $a\models p$; if $S(x,y)$ is a quasi-successor on $D$, then the proof of Lemma \ref{Lemma monotone family not Ep determined is a shift} can be slightly modified to show that $S_x$ is an $\mathcal S$-shift for all $x\in D$. 
\end{Remark} 

\subsection{Simplicity of $\mathbf p$-semiintervals}

From now on, we will be dealing with $\mathbf p$-semiintervals where $\mathbf p$ is a weakly o-minimal pair. (For preliminaries on weakly o-minimal types see \cite[Subsection 2.2]{MTwqom1}.)
Recall that $\mathcal E_p$ denotes the set of all relatively $A$-definable equivalence relations on the locus of $p\in S(A)$; $\mathbf 1_p$ denotes the full equivalence relation $p(\Mon)^2$. If $(p,<)$ is a weakly o-minimal pair, then every $E\in \mathcal E_p$ is convex on $(p(\Mon),<)$ and $(\mathcal E_p, \subseteq)$ is a total order (see \cite[Lemma 3.6 and Proposition 3.10]{MTwom}).

\begin{Definition}\label{Definition_Ep_determined}
Let $\mathbf p=(p,<)$ be a weakly o-minimal pair over $A$.
\begin{enumerate}[(a)]
\item The semiinterval $S^{\mathbf p}_a$ is $\mathcal E_p$-determined if there exists some $E\in\mathcal E_p$ such that $S^{\mathbf p}_a=\{x\in [a]_E\mid a\leqslant x\}$; in that case, we say that $E$ determines $S^{\mathbf p}_a$.
\item The family $\mathcal S^{\mathbf p}$ is $\mathcal E_p$-determined if some (equivalently, any) of its members is such.
\end{enumerate}
\end{Definition}

This definition is taken from the work of Baizhanov and Kulpeshov in \cite{Baizhanov2006}. They introduce the notion of a {\it $p$-stable\footnote{Also called $p$-preserving.} convex to the right formula} that is {\it equivalence-generated}, where $(M,<)$ is a linearly ordered structure and $p\in S_1(A)$. In our terminology, the first notion corresponds to formulas $S(x,y)$ that relatively define a family of bounded non-trivial $(p,<)$-semiintervals such that the definable set $S(a,M)$ has an upper and a lower bound in $p(M)$ for all $a\models p$, while the second corresponds to the simplicity of these semiintervals. This proved useful in studying weakly o-minimal theories with few countable models, for example, in  \cite{Kulpeshov2007}, \cite{Altayeva2020}, \cite{Kulpeshov2020}, and \cite{Kulpeshov2021}.

\begin{Remark}\label{Remark 1.6}
Let $\mathcal S=(S(x,D)\mid x\in D)$ be an $A$-definable family of semiintervals of $(D,<)$ and let $\mathcal S^{\mathbf p}$ be the induced family of $\mathbf p$-semiintervals. 

(a) The family $\mathcal S^{\mathbf p}$ is $\mathcal E_p$-determined if and only if the formula $S(x,y)\lor S(y,x)$ relatively defines an equivalence relation on $p(\Mon)$, in which case this relation determines $\mathcal S^{\mathbf p}$. 

(b) $\mathcal S^\pp$ is $\mathcal E_p$-determined if and only if $\supin_{p(\Mon)} (S_x^\pp)=\supin_{p(\Mon)} (S_y^\pp)$ is a relatively definable relation on $p(\Mon)$, in which case it is relatively defined by $S(x,y)\vee S(y,x)$. One more equivalent condition is: 
\begin{center}for all $x<y$ realizing $p$: \ $y\in S^{\mathbf p}_x$ \ \ \ if and only if \ \ \ $\supin_{p(\Mon)} (S^{\mathbf p}_x)=\supin_{p(\Mon)} (S^{\mathbf p}_y)$.
\end{center}
\end{Remark}

\begin{Definition} Let $\mathbf p=(p,<)$ be a weakly o-minimal pair over $A$.
\begin{enumerate}[(a)]
\item  The pair $\mathbf p$ {\em has simple semiintervals} if every $\mathbf p$-semiinterval is $\mathcal E_p$-determined. 
 
\item The type $p$ {\em has simple semiintervals} if every weakly o-minimal pair $(p,<_p)$ over $A$ has simple semiintervals. 

\item The type $p$ {\em has hereditarily simple semiintervals} if every complete type that extends $p$ has simple semiintervals. 

\item A weakly quasi-o-minimal theory $T$ {\it has simple semiintervals} if all types $p\in S_1(T)$ have hereditarily simple semiintervals.  
\end{enumerate}
\end{Definition}

All $\aleph_0$-categorical and all almost $\aleph_0$-categorical\footnote{Almost $\aleph_0$-categoricity is defined in Definition \ref{Definition AAC}.} weakly o-minimal theories have simple semiintervals, see \cite{Baizhanov2006} and \cite{Altayeva2020}.  

\begin{Example}\label{Example_ssi_not hereditary}
It is possible that a weakly o-minimal type has simple semiintervals, but not hereditarily so. Consider the reduct of the additive ordered group of reals $(\mathbb R,m,<)$, where $m(x,y,z):=x-y+z$. This structure is o-minimal with a unique type $p\in S_1(\emptyset)$. There are exactly three complete 2-types, determined by $x<y$, $y<x$, and $x=y$, respectively; this implies that $p$ has simple semiintervals. However, if we name $0$ and put $q=\tp(1/0)$, then $S_a=\{t\mid  a\leqslant t\leqslant m(a,0,a)\}$ is a $\mathbf q$-semiiniterval that is not $\mathcal E_q$-determined. 
\end{Example}

In Proposition \ref{Prop_simple_semiintervals_basic} below, we find a syntactic condition equivalent to the simplicity of $\mathbf p$-semiintervals that offers a simple description of the relatively definable binary structure induced on the locus of $p$. As a corollary, we obtain that the simplicity of semiintervals of $p\in S(A)$ follows from the simplicity of semiintervals of {\em some} weakly o-minimal pair $(p,<)$ over $A$.

\begin{Remark}\label{Remark SEMI rel def set is BC of semiint}
Let $\mathbf p=(p,<)$ be a weakly o-minimal pair over $A$ and let $a\models p$.

(a) Every convex, relatively $Aa$-definable subset of the interval $(a,+\infty)$ (of $(p(\Mon),<)$) is a difference of two $\mathbf p$-semiintervals (with minimum $a$). Therefore, every relatively $Aa$-definable subset of $[a,+\infty)$ is a finite Boolean combination of $\mathbf p$-semiintervals with minimum $a$. Similarly, every relatively $Aa$-definable subset of $(-\infty,a]$ is a finite Boolean combination of $\mathbf p^*$-semiintervals with minimum $a$; here, $\mathbf p^*$ is the opposite order $\mathbf p^*=(p,>)$.

(b)  Every relatively $Aa$-definable subset of $p(\Mon)$ is a finite Boolean combination of $\mathbf p$-semiintervals and $\mathbf p^*$-semiintervals, where $\mathbf p^*$ is the reverse of $\mathbf p$.

(c) The set $\mathcal D_p(a)=\{b\models p\mid b\dep_A a\}$\ is the union of all bounded $\mathbf p$-semiintervals and $\mathbf p^*$-semiintervals (see \cite[Remark 2.4]{MTwqom1}). 
\end{Remark}

\begin{Proposition}\label{Prop_simple_semiintervals_basic}
A weakly o-minimal pair over $A$, $\mathbf p=(p,<)$, has simple semiintervals if and only if every relatively $A$-definable subset of $p(\Mon)^2$ is relatively defined by some finite Boolean combination of $x\leqslant y$, $y\leqslant x$ and formulae of the form $E(x,y)$ for $E\in\mathcal E_p$.
\end{Proposition}
\begin{proof} 
Let $\mathcal B$ denote the set of all Boolean combinations of $x\leqslant y$, $y\leqslant x$ and $E(x,y)$ for $E\in\mathcal E_p$.
To prove $(\Rightarrow)$, suppose that every $\mathbf p$-semiinterval is $\mathcal E_p$-determined, and let $D$ be a relatively $A$-definable subset of $p(\Mon)^2$. 
It suffices to prove that each of the sets $D_1=\{(x,y)\in D\mid x\leqslant y\}$ and $D_2=\{(x,y)\in D\mid y\leqslant x\}$ is relatively defined by a formula from $\mathcal B$. Fix $a\models p$, and consider the fiber $D_1(a,p(\Mon))$. Note that $D_1(a,p(\Mon))\subseteq [a,+\infty)_{\mathbf p}$, so by Remark \ref{Remark SEMI rel def set is BC of semiint}, $D_1(a,p(\Mon))$ is a finite Boolean combination of $\mathbf p$-semiintervals with minimum $a$. Since $\mathbf p$ has simple semiintervals, each of these semiintervals is relatively defined by $a\leqslant x\land E(a,x)$ for some $E\in\mathcal E_p$; note that this formula is the $a$-instance of a $\mathcal B$-formula. 
Therefore, $D_1(a,p(\Mon))$ is relatively defined by an $a$-instance of a $\mathcal B$-formula; let $\phi(x,y)\in\mathcal B$ be such that $\phi(a,y)$ relatively defines $D_1(a,p(\Mon))$ within $p(\Mon)$. Then $\phi(x,y)\land x\leqslant y$ relatively defines $D_1$ within $p(\Mon)^2$. 
Similarly, considering the fiber $D_2(p(\Mon),a)$, we see that $D_2$ is relatively defined by some $\mathcal B$-formula. This proves $(\Rightarrow)$.

To prove $(\Leftarrow)$, assume that every relatively $A$-definable subset of $p(\Mon)^2$ is relatively defined by some $\mathcal B$-formula. Let $S_a$ be a bounded $\mathbf p$-semiinterval. We need to show that $S_a$ is $\mathcal E_p$-determined; without loss, assume $S_a\neq \{a\}$. 
Choose $\phi(x,y)\in \mathcal B$ which relatively defines $S_a$. Let $E_1,\dots,E_n$ be the sequence of all equivalence relations from $\mathcal E_p$ that appear in $\phi$. Since $\mathcal E_p$ is totally ordered by the inclusion, we may assume that the sequence is strictly increasing. In addition, we may assume that $E_1$ is trivial and $E_n$ is the full relation. Then
\[ \{E_1(a,p(\Mon))\}\cup\{(E_i(a,p(\Mon)))\smallsetminus E_{i-1}(a,p(\Mon))\cap [a,+\infty)_{\mathbf p}\mid i=2,\dots,n\}\] 
is a convex partition of $[a,+\infty)_\mathbf p$, with each piece being relatively defined by an $a$-instance of a $\mathcal B$-formula. It is easy to see that every relatively definable subset of $[a,+\infty)_{\mathbf p}$, defined by an $a$-instance of a $\mathcal{B}$-formula in which only equivalences $E_1,\dots,E_n$ appear, is equal to the union of some members of the partition. In particular, this holds for $S_a$, so if $i$ is maximal such that $(E_i(a,p(\Mon)))\smallsetminus E_{i-1}(a,p(\Mon))\cap [a,+\infty)_{\mathbf p}\subseteq S_a$, the convexity of $S_a$ and $a=\min S_a$ together imply that $S_a$ is relatively defined by $a\leqslant x\land E_i(a,x)$, and we are done.
\end{proof}

Let $(p,<)$ be a weakly o-minimal pair over $A$. In the next proof, we will use the characterization of relatively $A$-definable orders on $p(\Mon)$, proved in Theorem 3.2 of \cite{MTwom}. For a strictly increasing sequence $\vec E=(E_1,\dots,E_n)\in \mathcal E_p^n$, the order $<_{\vec E}$ is defined recursively by: 
\begin{itemize}
    \item $x<_{E_1}y $ \ if and only if \ $[x]_{E_1}<[y]_{E_1}$ or $[x]_{E_1}=[y]_{E_1}\land y<x$;
    \item $<_{(E_1,\dots,E_{k+1})}=\left(<_{(E_1,\dots,E_k)}\right )_{E_{k+1}}$.
\end{itemize}
Clearly, $<_{\vec E}$ is a relatively $A$-definable order on $p(\Mon)$,  defined by a Boolean combination of formulae $x\leqslant y$, $y\leqslant x$ and $E_i(x,y)$ ($1\leqslant i\leqslant n)$; here, $E_i(x,y)$ is any formula that relatively defines $E_i$.   
By \cite[Theorem 3.2]{MTwom}, every relatively $A$-definable order on $p(\Mon)$ equals $<_{\vec E}$ for some $\vec E$.

\begin{Corollary}\label{COr_simple_semiint_ind_of order}
A weakly o-minimal type $p\in S(A)$ has simple semiintervals if and only if    {\it some} weakly o-minimal pair $\mathbf p=(p,<)$ over $A$ has simple semiintervals. 
\end{Corollary}
\begin{proof}Suppose that $\mathbf p=(p,<)$ has simple semiintervals. Let $\mathcal B_{<}$ denote the set of all $L_A$-formulae that relatively define within $p(\Mon)^2$ the same set as some Boolean combination of formulae $x\leqslant y$, $y\leqslant x$ and $E(x,y)$ ($E\in\mathcal E_p$). 
By Proposition \ref{Prop_simple_semiintervals_basic}, every relatively $A$-definable subset of $p(\Mon)^2$ can be defined by a formula from $\mathcal B_{<}$, so $\mathcal B_{<}$ is the set of all $L_A$-formulae in two variables. 
Now, let $\mathfrak p=(p,<_p)$ be another weakly o-minimal pair over $A$, and let $\mathcal B_{<_p}$ be the corresponding set of formulae.
By \cite[Theorem 4.2]{MTwom} we have $<=(<_p)_{\vec E}$ for some increasing sequence of equivalences from $\mathcal E_p$. From the definition of $(<_p)_{\vec E}$, we easily derive
$(x\leqslant y), (y\leqslant x)\in \mathcal B_{<_p}$, which implies $\mathcal B_<\subseteq \mathcal B_{<_p}$. As $\mathcal B_<$ is the set of all binary $L_A$-formulae, $\mathcal B_<=\mathcal B_{<_p}$ follows, so by Proposition \ref{Prop_simple_semiintervals_basic} again, the pair $\mathfrak p$ has simple semiintervals. Therefore, every weakly o-minimal pair $(p,<_p)$ has simple semiintervals and $p$ has simple semiintervals. 
\end{proof}

\begin{Remark}\label{Remark ultrametric on p(Mon)}
For a weakly o-minimal type $p\in S_x(A)$, fix the following notation.
\begin{itemize}
    \item $S_{p,p}(A):=\{P(x,y)\in S_{x,y}(A)\mid P(x,y)\supseteq p(x)\cup p(y)\}$;
    \item $\mathcal E(P)=\{E\in\mathcal E_p\mid E(x,y)\in P\}$ \ where $P\in S_{p,p}(A)$;
    \item $\mathcal F_p$ is the set of all the final parts of $(\mathcal E_p,\subseteq)$.
\end{itemize}

(a) Suppose that $p$ has simple semiintervals. By Proposition \ref{Prop_simple_semiintervals_basic}, for each final part $F\in \mathcal F_p$ the following partial type:
\[p(x)\cup p(y)\cup \{ E(x,y)\mid  E\in F\}\cup \{\lnot E(x,y)\mid E\in\mathcal E_p\smallsetminus F\}\]
determines a unique complete type $P\in S_{p,p}(A)$; note that $\mathcal E(P)=F$. Therefore, $P\mapsto \mathcal E(P)$ is a bijection between $S_{p,p}(A)$ and $\mathcal F_p$. Moreover, if we endow $\mathcal F_p$ with the order topology of $(\mathcal F_p,\supseteq)$, then this bijection is a homeomorphism. 

(b) Suppose that $\mathcal E_p$ is countable. Then there is a strictly increasing continuous function $\delta:\mathcal F_p\to [0,1]$ such that $\delta(\mathcal E_p)=0$ and $\delta(\emptyset)=1$.  
For $a,b\models p$ define $d(a,b)=\delta(\tp(a,b/A))$. 
Then $(p(\Mon),d)$ is an ultrametric space: The ultrametric inequality $d(x,y)\leqslant \max(d(x,z),d(y,z))$ translates to:
\[\left(\mathcal E(\tp(x,z))\cap \mathcal E(\tp(y,z))\right)\subseteq \mathcal E(\tp(x,y)).\]
This is easy to verify: if $E\in (\mathcal E(\tp(x,z))\cap \mathcal E(\tp(y,z)))$, then $\models E(x,z)\land E(y,z)$, so $\models E(x,z)$ and therefore $E\in\mathcal E(\tp(x,y))$. 

(c) In general, the ultrametric $d$ can be trivial, for example,  in $o$-minimal theories. However, if $p\in S(\emptyset)$ has simple semiintervals, then the distance $d(a,b)$ measures the amount of forking dependence of $a$ and $b$, as we prove in Proposition  \ref{Prop_simple_semiintervals_implies_interdef}.
The importance of this ultrametric space was first observed by
Herwig et al. in \cite[Theorems 2.2-2.3]{Herwig} for the case of 1-types in $\aleph_0$-categorical weakly o-minimal theories. In that case, $d(x,y)=r$ is $A$-definable for all $r\in[0,1]$, and they have shown that every complete 2-type over $A$ extending $p(x)\cup p(y)\cup \{x<y\}$ is isolated by some formula $d(x,y)=r$. In other words, the binary structure on $p(\Mon)$ is completely determined by the ultrametric and the order $<$. Proposition \ref{Prop_simple_semiintervals_basic} can be regarded as saying the same for any weakly o-minimal type with simple semiintervals. 
\end{Remark}

\begin{Proposition}\label{Prop_ssi_consequences}
Let $\mathbf p=(p,<)$ be a weakly o-minimal pair over $A$ and let $a\models p$. Suppose that $p$ is nonalgebraic and has simple semiintervals.   
\begin{enumerate}[(a)]
\item The image of $p$ under a relatively $A$-definable function is a weakly o-minimal type with simple semiintervals.
\item The only relatively $A$-definable map $f:p(\Mon)\to p(\Mon)$ is the identity map. 
\item $(p(\Mon)/E,<)$ is a dense linear order without endpoints for every $E\in\mathcal E_p\smallsetminus \{\mathbf 1_p\}$. 
\item If the class $[b]_F$ is $A\cup\{[a]_E\}$-invariant, where $b\in p(\Mon)$ and $E,F\in \mathcal E_p$, then $[a]_E\subseteq [b]_F$.
\end{enumerate}
\end{Proposition}
\begin{proof} (a) Let $f:p(\Mon)\to \Mon$ be a relatively $A$-definable function, and let $q=f(p)$. By \cite[Proposition 3.13(i)]{MTwom},  $f^{-1}(\{y\})<f^{-1}(\{y'\})$ defines a relatively $A$-definable linear order, denoted by $<_f$, on $q(\Mon)$ such that $\mathbf q=(q,<_f)$ is a weakly o-minimal pair over $A$. To verify that $q$ has simple semiintervals, let $b=f(a)$, let $S_b$ be a $\mathbf q$-seminterval which is relatively defined by $S(b,y)$, and let $X_a\subseteq p(\Mon)$ be relatively defined by $(a\leqslant x \land S(f(a),f(x)))$; clearly $X_a$ is a $\mathbf p$-semiinterval. Choose $E\in \mathcal E_p$ which determines $X_a$. Then by \cite[Proposition 2.13]{MTwom}(b),
$f(E)\in \mathcal E_q$ and it is easy to see that $S_a$ is $f(E)$-determined. 
Therefore, $q$ has simple semiintervals. 

(b) Suppose that $f:p(\Mon)\to p(\Mon)$ is relatively $A$-definable. Reversing the order $<$ if necessary, assume $a\leqslant f(a)$.
Consider the set $S_a=\{x\in p(\Mon)\mid a\leqslant x\leqslant f(a)\}$; clearly, $S_a$ is a $\mathbf p$-semiinterval and $f(a)=\max S_a$. Let $E\in\mathcal E_p$ determine $S_a$. Since $[a]_E=[f(a)]_E$, the tuple $f(a)$ satisfies: $x\in p(\Mon)$ is the maximal element of its $E$-class; so does $a$, and $a=f(a)$ follows.  

(c) Let $\hat E$ be an $A$-definable extension of $E$. Then $(p(\Mon)/E,<)$ and $(p/\hat E(\Moneq),<)$ are naturally identified. By part (a), the type $p/\hat E$ is weakly o-minimal and has simple semiintervals. Observe that $(p/\hat E(\Moneq),<)$ is not a discrete order, because otherwise the successor function would contradict part (b). Therefore, $(p/\hat E(\Moneq),<)$ is a dense order.

(d) Without loss assume that $A=\emptyset$. Suppose that $E,F\in \mathcal E_p$ and that the class $[b]_F$ is $\{[a]_E\}$-invariant. Choose $\hat E$ and $\hat F$, definable extensions of $E$ and $F$, respectively. Then $[b]_{\hat F}\in \dcleq([a]_{\hat E})$, say $[b]_{\hat F}=h([a]_{\hat E})$. Without loss of generality, assume $a<b$. Since $E$ and $F$ are $\subseteq$-comparable, the classes $[a]_E$ and $[b]_F$ are $\subseteq$-comparable or disjoint, so to establish $[a]_E\subseteq [b]_F$, it is enough to eliminate the cases where $[a]_E \supset [b]_F$ and where $[a]_E\cap [b]_F=\emptyset$. 
If $[a]_E\supset [b]_F$ were true, then there would be some $b'\in [a]_E$ with $[b]_F\neq [b']_F$, so any $f\in \Aut(\Mon)$ moving $b$ to $b'$ would fix $[a]_E=[b]_E=[b']_E$; this contradicts the $\{[a]_E\}$-invariance of $[b]_F$. Therefore, $[a]_E\nsupset [b]_F$.  
Next, assume that $[a]_E\cap [b]_F=\emptyset$. Then $a<b$ implies $[a]_E<[b]_F$, so $a<[b]_F$ implies $[a]_F<[b]_F$. 

Consider the set $X_a=\{x\in p(\Mon)\mid a\leqslant x<[b]_F\}$ and note that it is relatively defined by the formula $a\leqslant x<h([x]_{\hat E})$. Hence $X_a$ is
a bounded $\mathbf p$-semiinterval, so by the simplicity of semiintervals of $p$ there is some $G\in\mathcal E_p$ such that $X_a$ is a final part of the class $[a]_G$. In particular, the set $Y=[a]_G\cup [b]_F$ is convex in $(p(\Mon),<)$. By (c), $(p(\Mon)/F,<)$ is a dense endless order, so $[a]_F<[b]_F$ implies that there are infinitely many $F$-classes lying strictly between $[a]_F$ and $[b]_F$; since $Y$ is convex, at least one of these classes is properly contained in $[a]_G$; $F\subset G$ follows. 
Now, as $[a]_G<b$ implies $[a]_G<[b]_G$, $[b]_F\subset[b]_G$, and $Y$ is convex in $p(\Mon)$, $[a]_G\cup[b]_G$ is convex in $p(\Mon)$ too, so $[a]_G$ and $[b]_G$ are two consecutive $G$-classes. This is a contradiction as $(p(\Mon)/G,<)$ is a dense order by (c).
\end{proof}

\begin{Lemma}\label{Lemma O EKVIVALENCIJAMA}
Let $A\subseteq B$, let $p\in S(A)$ be a weakly o-minimal type, and let $q\in S(B)$ be its extension. Suppose that $E\in\mathcal E_q\smallsetminus\{\mathbf 1_q\}$. The following conditions are equivalent:
\begin{enumerate}[\hspace{10pt}1)]
\item There exists $E_0\in\mathcal E_p$ such that $(E_{0})_{\strok q(\Mon)}= E$;
\item $E$ is relatively $A$-definable;
\item For some (all) $a\models q$, the class $[a]_E$ is $Aa$-invariant.
\end{enumerate}
\end{Lemma}
\begin{proof} Without loss of generality, assume $A=\emptyset$ and fix a weakly o-minimal pair $\mathbf p=(p,<)$ over $\emptyset$. 
1)$\Rightarrow$2) is obvious. 

2)$\Rightarrow$3) Assume 2). Let $E(x,y)$ be an $L$-formula that relatively defines $E$ on $q(\Mon)$, and let $a\models q$. By weak o-minimality of $q$ and since $E\neq\mathbf 1_q$, the class $[a]_E= E(a,q(\Mon))$ is a convex bounded subset of $q(\Mon)$, so, since $q(\Mon)$ is a convex subset of $p(\Mon)$ (see \cite[Lemma 3.6(ii)]{MTwom}), $[a]_E$ is one of the convex components of the set $E(a,p(\Mon))$ within $p(\Mon)$. Since each component is relatively $a$-definable within $p(\Mon)$ (see \cite[Lemma 3.6(i)]{MTwom}), so is $[a]_E$. In particular, $[a]_E$ is $a$-invariant, proving 3). 

3)$\Rightarrow$1)
Let $E(x,y;\bar b)$ be a formula that relatively defines $E$ on $q(\Mon)$, where $\bar b$ are parameters from $B$, and let $a\models q$. Denote:  $s(\bar z)=\tp(\bar b)$ and $r(x,\bar z)=\tp(a,\bar b)$. Clearly, $r(x,\bar z)\supset p(x)$ and $q(x)\supseteq r(x,\bar b)$.  

From the proof of 2)$\Rightarrow$3), we know that the class $[a]_E=E(a,q(\Mon);\bar b)$ is one of the convex components of $E(a,p(\Mon);\bar b)$ within $p(\Mon)$ and that $[a]_E$ is convex and bounded within $q(\Mon)$. We also know that each of these components is relatively $a\bar b$-definable within $p(\Mon)$, so we can replace $E(x,y;\bar z)$ by a formula whose instance $E(a,y;\bar b)$ relatively defines the class $[a]_E$ within $p(\Mon)$, that is, assume $E(a,p(\Mon);\bar b)=[a]_E$. Since $r(x,\bar b)=\tp_x(a/\bar b)\subseteq q(x)$, after this replacement, $E(x,y;\bar b)$ still relatively defines $E$ on $q(\Mon)$. Hence, we have
\begin{equation}\tag{1}
    s(\bar z)\vdash"\text{$E(x,y;\bar z)$ relatively defines an equivalence relation on $r(\Mon,\bar z)$."}
\end{equation}
Further, $E(a,q(\Mon);\bar b)= E(a,p(\Mon);\bar b)$ implies $p(y)\cup\{E(a,y;\bar b)\}\vdash q(y)$; as $q(x)=r(x,\bar b)$, we have:
\begin{equation}\tag{2}
    r(x,\bar z)\cup p(y)\cup\{E(x,y;\bar z)\}\vdash r(y,\bar z).
\end{equation}
Now, note that $a\bar{b'}\models r$ implies $E(a,p(\Mon);\bar b)=E(a,p(\Mon);\bar{b'})$ as $[a]_E=E(a,p(\Mon);\bar b)$ is $a$-invariant by 3). This together with (2) gives:
\begin{equation}\tag{3}
    r(x,\bar z)\cup r(x,\bar{z'})\cup p(y)\vdash E(x,y;\bar z)\leftrightarrow E(x,y;\bar{z'}).
\end{equation}
Consider the type: $\Pi(x,y)=(\exists\bar z')(\bigwedge r(x,\bar z')\land \bigwedge r(y,\bar z')\land E(x,y;\bar z')).$ From (2) and (3) we derive 
\begin{equation}\tag{4}
    r(x,\bar z)\cup \Pi(x,y)\vdash \bigwedge r(y,\bar{z})\land E(x,y;\bar{z}).  \ 
\end{equation}
This, together with (1) gives:
\begin{equation}\tag{5}
\Pi(x,y)\cup \Pi(y,u)\cup r(y,\bar z)\vdash \bigwedge r(x,\bar{z})\land \bigwedge r(u,\bar{z})\land E(x,u;\bar{z}). 
\end{equation}
Note that applying $(\exists \bar z)$ to the right-hand side in (5) gives the definition of $\Pi(x,u)$, so, in particular, $\Pi(x,y)\land \Pi(y,u)\vdash \Pi(x,u)$ holds. It follows that the locus of $\Pi(x,y)$, denoted by $E_0$, is an equivalence relation on $p(\Mon)$; here, reflexivity and symmetry are clear from the definition of $\Pi$.
Clearly, $E\subseteq E_0$ holds as the existential quantifier in $\Pi(x,y)$ is witnessed by $\bar b$. On the other hand, if $a_1,a_2\models q$ and $\models \Pi(a_1,a_2)$, then $\models E(a_1,a_2;\bar b)$ follows by (4); this proves $(E_0)_{\strok q(\Mon)}\subseteq E$. Therefore, $(E_0)_{\strok q(\Mon)}=E$.

By now, we know that $E$ is the restriction of $E_0$, which is type-definable over $\emptyset$. To prove 1), it remains to show that $E_0$ is relatively 0-definable. 
By compactness, applied to $r(x,\bar z)\cup\Pi(x,y)\vdash E(x,y;\bar z)$ from (4),  there is a finite conjunction $\pi(x,y)$ of formulae from $\Pi(x,y)$ such that: 
\begin{equation}\tag{6}
    r(x,\bar z)\cup \{\pi(x,y)\}\vdash E(x,y;\bar z).
\end{equation}
We {\em claim} that $E_0$ is relatively defined by $\pi(x,y)$ on $p(\Mon)$. To prove this, suppose that $a_1,a_2\models p$ and $\models\pi(a_1,a_2)$. Pick $\bar{b'}\models s$ such that $a_1\bar{b'}\models r$. Then (6) yields $\models E(a_1,a_2;\bar{b'})$, which together with $a_1\bar{b'}\models r$ implies $a_2\bar{b'}\models r$ by (2), so $\models \Pi(a_1,a_2)$ holds by the definition (where $\bar{b'}$ witnesses the existential quantifier). Therefore, $E_0$ is relatively 0-definable on $p(\Mon)$, completing the proof of the lemma.
\end{proof}

\begin{Corollary}\label{Fact_imaginaries_Epqdetermined_semiintervals}
Suppose that $p\in S(A)$ is weakly o-minimal, $a\models p$, $A\subseteq B\subseteq \dcleq(Aa)$, and $q=\tp(a/B)$. 
\begin{enumerate}[\hspace{10pt}(a)]
\item All equivalences $E\in\mathcal E_q$ are relatively $A$-definable. 
\item Every bounded $(q,<)$-semiinterval $S_a$ is also a $(p,<)$-semiinterval. Moreover, $S_a$ is $\mathcal E_p$-determined if and only if it is $\mathcal E_{q}$-determined. 
\item If $p$ has simple semiintervals, then $q$ also has simple semiintervals. 
\end{enumerate}
\end{Corollary}
\begin{proof}(a) If $E=\mathbf 1_q$, the conclusion holds trivially, so suppose $E\neq\mathbf 1_q$.
By Lemma \ref{Lemma O EKVIVALENCIJAMA} it suffices to show that the class $[a]_E$ is $Aa$-invariant. However, this is clear as $B\subseteq\dcleq(Aa)$. 

(b) easily follows from (a), while (c) follows from (b) and Corollary \ref{COr_simple_semiint_ind_of order}. 
\end{proof}

\subsection{Simplicity of semiintervals and binary forking dependence}
\label{Subsection 4}
 
In this subsection, in Proposition \ref{Prop_simple_semiintervals_implies_interdef}, we prove that the set $\dcleq(a)\cap \dcleq(b)$ measures quite accurately the amount of forking dependence of $a$ on $b$ whenever $\tp(a)$ and $\tp(b)$ are weakly o-minimal and at least one of them has simple semiintervals. This looks like a very weak form of 1-basedness, but the intuition behind this is that the simplicity of semiintervals implies that the binary forking dependence is measured by elements of a certain meet-tree, as evidenced in Corollary \ref{Corollary_meet_tree}.

\begin{Notation} For any $D\subseteq \Moneq$ we will use the following notation:
\begin{itemize}
\item $\mathcal W_A(D)= \{a\in D\mid \mbox{$\tp(a/A)$ is weakly o-minimal}\}$;
\item $\mathcal W_A^{ss}(D)= \{a\in \mathcal W_A(D)\mid \mbox{$\tp(a/A)$ has simple semiintervals}\}$;
\item $\mathcal T_A(D)=\{X\subset \Moneq\mid \mbox{$X=\dcleq(AX)\subseteq \dcleq(Aa)$ for some $a\in \mathcal W_A(D)$}\}$;
\item $m_A(a,b)=\dcleq(Aa)\cap\dcleq(Ab)$. 
\end{itemize}
If $A=\emptyset$, we omit stressing it in the index.
\end{Notation}

In Corollary \ref{Corollary_meet_tree} we will prove that $(\mathcal T_A(D),\subset)$ is a meet-tree with $m_A$ its meet operation, which will be used in the proof of Proposition \ref{Prop_simple_semiintervals_implies_interdef}.

\begin{Remark}\label{Remark increasing A in W_A and W_A^ss}
(a) Note that $A\subseteq B$ implies $\mathcal W_A(D)\subseteq\mathcal W_B(D)$: if $a\in\mathcal W_A(D)$, then $a\in\mathcal W_B(D)$ as any extension of a weakly o-minimal type is weakly o-minimal. 

(b) In general,  $A\subset B$ does not imply $\mathcal W_A^{ss}(D)\subseteq\mathcal W_B^{ss}(D)$: It is possible that $\tp(a/A)$ has simple semiintervals but $\tp(a/B)$ does not, as we know from Example \ref{Example_ssi_not hereditary}.

(c) $a\mapsto \dcleq(Aa)$ is a natural identification of elements of $\mathcal W_A(\Moneq)$ with elements of $\mathcal T_A(\Moneq)$, so we can view $\mathcal W_A(\Moneq)$ as a subset of $\mathcal T_A(\Moneq)$.
\end{Remark}

\begin{Lemma}\label{Lemma_wom_dcl_is_linear}
Suppose that the type $p=\tp(a/A)$ is weakly o-minimal. 
\begin{enumerate}[\hspace{10pt}(a)] 
\item For every element $e\in\dcleq (Aa)$ there exists $E\in\mathcal E_p$ such that $e$ and (the $A$-hyperimaginary element) $[a]_E$ are interdefinable over $A$; that is, $\Aut_{Ae}(\Moneq)=\Aut_{A\{[a]_E\}}(\Moneq)$. 
\item The set $\dcleq(Aa)$ is totally preordered by $x\in\dcleq(Ay)$.
\end{enumerate}
\end{Lemma}
\begin{proof} Assume, without loss of generality, that $A=\emptyset$ and $\Mon=\Moneq$.

(a) Let $e\in\dcleq(a)$. Choose a $0$-definable function $f$ such that $f(a)=e$ and let $E\in\mathcal E_p$ be the kernel relation of $f_{\strok p(\Mon)}$. It is easy to see that the hyperimaginary element $[a]_E$ is as claimed.  

(b) We only need to check that for $b_1,b_2\in\dcleq(a)$, $b_1\in\dcleq(b_2)$ or $b_2\in\dcleq(b_1)$ is true. By part (a), for $i=1,2$ there is $E_i\in \mathcal E_p$ such that $b_i$ and $[a]_{E_i}$ are interdefinable. By weak o-minimality, $E_1,E_2\in \mathcal E_p$ are comparable, so without loss suppose $E_1\subseteq E_2$. Then we have 
$$\Aut_{b_1}(\Moneq)= \Aut_{\{[a]_{E_1}\}}(\Moneq)\subseteq \Aut_{\{[a]_{E_2}\}}(\Moneq) =\Aut_{b_2}(\Moneq).$$
Now, $\Aut_{b_1}(\Moneq) \subseteq \Aut_{b_2}(\Moneq)$ 
implies $b_2\in \dcleq (b_1)$. 
\end{proof}

\begin{Corollary}\label{Corollary_meet_tree}
$(\mathcal T_A(D),\subset)$ is a meet-tree with root $\dcleq(A)$ for all $D\subseteq \Moneq$.  \end{Corollary}
\begin{proof} 
For simplicity, assume $A=\emptyset$. 
Let $X\in \mathcal T(D)$. Choose $a\in \mathcal W(D)$ such that $X=\dcleq(X)\subseteq \dcleq(a)$. By Lemma \ref{Lemma_wom_dcl_is_linear}, the set $\dcleq(a)$ is totally preordered by $x\in\dcleq(y)$. Note that every $\dcleq$-closed subset of $\dcleq(a)$ is an initial part of that preorder. Since the set of all initial parts of a total preorder is totally ordered by inclusion, we conclude that the set $\{Y\in \mathcal T(D)\mid Y\subset X\}$ is a chain in $(\mathcal T(D),\subset)$. This shows that $(\mathcal T(D),\subset)$ is a tree. 
It is easy to see that $X\cap Y\in \mathcal T(D)$ for all $X,Y\in\mathcal T(D)$, so $X\cap Y\in \mathcal T(D)$ is the meet of $X,Y$; $(\mathcal T(D),\subset)$ is a meet-tree. Obviously, $\dcleq(\emptyset)$ is the root of $\mathcal T(D)$.  
\end{proof}

\begin{Lemma}\label{Lemma_ssi_implies_a ind b_iff_m(a b)=0}   
For all $a\in \mathcal W_A^{ss}(\Moneq)$ and $b\in \mathcal W_A(\Moneq)$: \  $a\dep_Ab$ \ if and only if \ $m_A(a,b)\neq\dcleq(A)$. 
\end{Lemma}
\begin{proof} Without loss of generality, assume $A=\emptyset$. Denote $p=\tp(a)$ and $q=\tp(b)$. Let $\mathbf q=(q,<_q)$ be a weakly o-minimal pair and let $(D_q,<_q)$ be a $0$-definable extension of $(q(\Mon),<_q)$.

To prove the right-to-left implication, assume that $m(a,b)\neq \dcleq(\emptyset)$ and let $c\in m(a,b)\smallsetminus\dcleq(\emptyset)$. By Lemma \ref{Lemma_wom_dcl_is_linear}(a), there is some $E\in\mathcal E_p$ for which $c$ is interdefinable with the class $[a]_E$; similarly, $c$ is interdefinable with $[b]_F$  for some $F\in\mathcal E_q$. Note that the classes $[a]_E$ and $[b]_F$ are also interdefinable. 
Choose $b_1,b_2\models q$ such that $b_1\triangleleft^\mathbf q b\triangleleft^\mathbf q b_2$. Then, since $c\notin\dcleq(\emptyset)$ implies $F\neq \mathbf 1_q$, we conclude $[b_1]_F<_q[b]_F<_q[b_2]_F$. This, together with the fact that $[b]_F$ is $\{[a]_E\}$-invariant, yields $b_1\nequiv b\ (a)$ and $b\nequiv b_2\ (a)$. We conclude that the locus of $\tp(b/a)$ is a subset of the interval $[b_1,b_2]_{q(\Mon)}$, whence  a bounded subset of $q(\Mon)$; $a\dep b$ follows.  

To prove the other implication, assume that $a\dep b$. 
Witness the forking by a formula $\phi(x,a)\in \tp(b/a)$ such that $\phi(q(\Mon),a)$ is a bounded subset of $(q(\Mon),<_q)$:   $e_1<_q\phi(q(\Mon),a)<_qe_2$ where $e_1,e_2\in q(\Mon)$. By compactness, we can modify $\phi$ so that $e_1<_q\phi(\Mon,a)<_qe_2$; in particular, $\phi(\Mon,a)$ is a bounded subset of $D_q$. Define:
\[F(a)=\{y\in D_q\mid y<_q\phi(\Mon,a)\} \mbox{ \ and  \ } G(a)=\supin(\phi(\Mon,a)) . \]
Note that both $F(a)$ and $G(a)$ are initial parts of $D_q$, $F(a)<e_2$, and $G(a)<e_2$. Hence, we have just defined the functions $F,G:p(\Mon)\to \overline{D}_q^{\,in}$; clearly, each of them is relatively definable within $p(\Mon)\times D_q$. Also note that $F(x)\subset G(x)$ holds for all $x\in p(\Mon)$ (as $b\in G(a)\smallsetminus F(a)$). 
Let $<_p$ be a relatively definable linear order on $p(\Mon)$ such that the function $F$ is $<_p$-increasing; it exists by \cite[Theorem 1(i)]{MTwom}. Denote $\mathbf p=(p,<_p)$ and define:
\[S_a=\{x\in p(\Mon)\mid a\leqslant_p x\land F(x)\subset G(a) \}.\]
We show that $S_a$ is a bounded $\mathbf p$-semiinterval: 
Here, $\min S_a=a$ is obvious, while the convexity of $S_a$ in $(p(\Mon),<_p)$ follows by the monotonicity of $F$. Thus, $S_a$ is a $\mathbf p$-semiinterval. 
To see that it is bounded, choose $a'\in p(\Mon)$ with $G(a)\subset F(a')$; for example, let $a'$ be such that $e_2\in F(a')$.
Then $a'\notin S_a$. By the monotonicity of $F$, $F(a)\subset G(a)\subset F(a')$ implies $a<_pa'$, which combined with $a'\notin S_a$ gives $S_a<_pa'$. Hence $S_a$ is a bounded $\mathbf p$-semiinterval.  
Since $a\in \mathcal W^{ss}(\Moneq)$, $p$ has simple semiintervals, so
$S_a$ is $\mathcal E_p$-determined, say by $E\in\mathcal E_p$: $S_a=\{x\in p(\Mon)\mid a\leqslant_p x\land x\in[a]_{E}\}$; since $S_a$ is bounded, $E$ is proper. We claim that for all $a_1,a_2\in p(\Mon)$: 
\begin{itemize} 
\item[(1)] \  $[a_1]_{E}<_p[a_2]_{E} \ \text{ implies } \ F(a_1)\subset G(a_1)\subseteq F(a_2)\subset G(a_2).$
\end{itemize}
Here, $F(a_1)\subset G(a_1)$ and $F(a_2)\subset G(a_2)$ are consequences of $\tp(a)=\tp(a_1)=\tp(a_2)$. To prove $G(a_1)\subseteq F(a_2)$, note that $[a_1]_{E}<_p [a_2]_{E}$ and $S_{a_1}\subseteq [a_1]_{E}$ together imply $S_{a_1}<_pa_2$, which, by the definition of $S_{a_1}$, implies $F(a_2)\nsubset G(a_1)$; hence  $G(a_1)\subseteq F(a_2)$, proving (1). 

From (1) we find that the sets $G(a_1)\smallsetminus F(a_1)$ and $G(a_2)\smallsetminus F(a_2)$ are disjoint for all $a_1,a_2\in p(\Mon)$ belonging to distinct $E$-classes.
Therefore, the class $[a]_{E}$ is the unique $E$-class containing an element $t$ with $b\in G(t)\smallsetminus F(t)$; In particular, $[a]_{E}$ is fixed by $\Aut_b(\Mon)$. 
Let $(D_p,<_p)$ be a definable extension of $(p(\Mon),<_p)$ and let $\hat E$ be a $0$-definable convex equivalence on $D_p$ such that $E=\hat E_{\strok p(\Mon)}$. Then $[a]_{\hat E}\in \Moneq$ is fixed by $\Aut_b(\Moneq)$, so $[a]_{\hat E}\in \dcleq(b)$.
Since $[a]_{\hat E}\in\dcleq(a)$ holds trivially, we conclude $[a]_{\hat E}\in m(a,b)$. Since $E$ is proper, we have $[a]_{\hat E}\notin\dcleq(\emptyset)$, so $m(a,b)\supset\dcleq(\emptyset)$. This completes the proof of the lemma.
\end{proof}

\begin{Proposition}\label{Prop_simple_semiintervals_implies_interdef}  $a\ind_{m_A(a,b)}b$ \ holds for 
all $a\in \mathcal W_A^{ss}(\Moneq)$ and $b\in \mathcal W_A(\Moneq)$. 
\end{Proposition}
\begin{proof} Without loss of generality, assume $A=\emptyset$. Denote $B=m(a,b)$.  Then $b\in\mathcal W(\Moneq)$ implies $b\in\mathcal W_{B}(\Moneq)$ and, by Corollary \ref{Fact_imaginaries_Epqdetermined_semiintervals}(c) we know $a\in\mathcal W_{B}^{ss}(\Moneq)$. Hence, we can apply Lemma \ref{Lemma_ssi_implies_a ind b_iff_m(a b)=0}: 
\begin{itemize}
\item[(1)] \
$a\ind_{B}b$  \ \ if and only if \ $m_B(a,b)=\dcleq(B)$.
\end{itemize}
Note that $B\subseteq\dcleq(a)$ implies $\dcleq(Ba)=\dcleq(a)$ and, similarly, $\dcleq(Bb)=\dcleq(b)$. Hence, 
\begin{itemize} 
\item[(2)] \ $m_{B}(a,b)= \dcleq(Ba)\cap \dcleq(Bb)=\dcleq(a)\cap \dcleq(b)=B=\dcleq(B)$. 
\end{itemize}
From (1) and (2) we derive \ $a\ind_Bb$.
\end{proof}

\begin{Proposition}\label{Lemma_triple_noshiftsconsequence}
 Suppose that $a,b,c\in\mathcal W_A(\Moneq)$ and that at least two of them belong to $\mathcal W_A^{ss}(\Moneq)$. 
\begin{enumerate}[\hspace{10pt} (a)]
    \item  $a\ind_B b$ holds for all sets $B$ that satisfy $m_A(a,b)\subseteq B \subseteq \dcleq(Aa)$. 
    \item  The set $\{a,b,c\}$ is pairwise independent over $  m_A(a,b)\cup   m_A(b,c)\cup   m_A(a,c)$ for all $A\subset \Moneq$. 
\end{enumerate}
\end{Proposition}
\begin{proof}
Without loss of generality, assume $A=\emptyset$.
 
(a) At least one of $a$ and $b$ belongs to $
\mathcal W^{ss}(\Moneq)$, so we can apply Proposition \ref{Prop_simple_semiintervals_implies_interdef}: $a\ind_{m(a,b)}b$. Consider the type $q=\tp(b/m(a,b))$ and its extension $q'=\tp(b/B)$; let $\mathbf q=(q,<)$ be a weakly o-minimal pair over $m(a,b)$, and let $\mathbf q'=(q',<)$. By way of contradiction, suppose $a\dep_Bb$. This means that there exists a relatively $Ba$-definable bounded subset $D$ of $q'(\Mon)$ that contains $b$; we may assume that $D$ is convex. Since $B\subseteq\dcleq(a)$, $D$ is $a$-definable; let $\phi(x,a)$ relatively define $D$. Since $\phi(q'(\Mon),a)$ is bounded within $q'(\Mon)$ and $q'(\Mon)$ is a convex subset of $q(\Mon)$, $\phi(q'(\Mon),a)$ is convex and bounded within $q(\Mon)$. This implies that $\phi(q'(\Mon),a)$ is one of the (finitely many) convex components of $\phi(q(\Mon),a)$ within $(q(\Mon),<)$. Modifying $\phi(x,a)$, if necessary, we may assume $\phi(q(\Mon),a)=\phi(q'(\Mon),a)$. Thus, $\phi(x,a)$ relatively defines a bounded subset of $q(\Mon)$ that contains $b$, so $a\dep_{m(a,b)}b$. Contradiction.

(b) Denote $B= m(a,b)\cup  m(b,c)\cup  m(a,c)$. Since $(\mathcal T(\Moneq),\subset)$ is a meet-tree, the meets $m(a,b), m(b,c)$, and $ m(a,c)$ are pairwise $\subseteq$-comparable, so $B$ is equal to one of them, say $B= m(b,c)$. Then $ m(a,b)\subseteq B \subseteq \dcleq(b)$, so 
$a\ind_B b$ holds by part (a) of the lemma; similarly,  $a\ind_B c$. Finally,   $b\ind_B c$ holds  by Proposition \ref{Prop_simple_semiintervals_implies_interdef}.
This completes the proof.
\end{proof}

\section{Existence of shifts vs simplicity of semiintervals}\label{Section 3}  

Unlike in the definable case studied in \cite[Section 6]{MTwqom1}, the notion of a shift is inappropriate for families of $\mathbf p$-semiintervals: if $\mathcal S^\pp$ is such a family, then there is no reason why the power $(S_a^\pp)^2=\{x\in p(\Mon)\mid (\exists t\models p)(t\in S_a\land x\in S_t)\}$ should be relatively definable at all. 
In this section, we compare the following two conditions for a weakly quasi-o-minimal $T$:
\begin{itemize}
    \item[(NS)] $T$ does not have simple semiintervals;
    \item[(SH)] $T$ admits a definable family of semiintervals of $(\Mon,<)$ that contains a shift. 
\end{itemize}
The main technical result of this section is Proposition \ref{Prop_p_notssi_exists_shift} in which we prove that every $\mathbf p$-semiinterval that is not $\mathcal E_p$-determined can be slightly enlarged, so that the new $\mathbf p$-semiinterval is induced by a shift of some monotone definable family of semiintervals. As an immediate corollary we see that (NS) implies (SH).  
Example \ref{Example Z < many colors} shows that the converse is not always true, but in Theorem \ref{Thm_noshifts_iff_simplesemiintervals}(b), we prove the converse assuming, in addition, that $T$ is small or weakly o-minimal.
As corollaries, we obtain that every weakly quasi-o-minimal theory with few countable models has simple semiintervals, and that every convex weakly o-minimal type in a theory with few countable models has hereditarily simple semiintervals.

\begin{Fact}\label{Fact_lemma 3_1 wom}
Let $\mathbf p=(p,<)$ be a weakly o-minimal pair over $A$ and let
$\mathcal S^{\mathbf p}=(S^{\mathbf p}_x\mid x\in p(\Mon))$ be a relatively $A$-definable family of $\mathbf p$-semiintervals. Then there do {\em not} exist $a,b,c\in p(\Mon)$ such that $S^{\mathbf p}_{a}<b$ and $S^{\mathbf p}_{a}\cup S^{\mathbf p}_{b}\subseteq S_{c}^\pp$.  
\end{Fact}
\begin{proof}
    This is a rephrased version of \cite[Lemma 4.1]{MTwom}.
\end{proof}

\begin{Lemma}\label{Lemma two conditions for S_a E_p determined}
Let $\mathcal S=(S_x\mid x\in D)$ be an $A$-definable family of semiintervals of $(D,<)$ and let $a\in D$. Assume that $p=\tp(a/A)$ is weakly o-minimal and let $\mathbf p=(p,<)$. Then the following are equivalent:

(1) \ $\mathcal S^{\mathbf p}$ is $\mathcal E_p$-determined; \ \
(2) \ $(\mathcal S^{\mathbf p})^2=\mathcal S^{\mathbf p}$; \ \ 
(3) \ $(\mathcal S^{\theta})^2=\mathcal S^\theta$ holds for some formula $\theta(x)\in p$; 
\end{Lemma}
\begin{proof} 
(1)$\Rightarrow$(2) is easy. To prove (2)$\Rightarrow$(3), assume that $(S^{\mathbf p}_x)^2= S^{\mathbf p}_x$ holds for all $x\models p$. Then
\[\models (\forall x,y,z\models p)(y\in S_x\land z\in S_y\rightarrow z\in S_x).\]
By compactness, there is a formula $\theta(x)\in p$ that implies $x\in D$ such that
\[\models (\forall x,y,z\in \theta(D))(y\in S_x\land z\in S_y\rightarrow z\in S_x).\]
It is easy to see that this implies $(\mathcal S^\theta)^2=\mathcal S^\theta$. 

\smallskip(3)$\Rightarrow$(2)  Assume that $(\mathcal S^{\theta})^2=\mathcal S^\theta$ holds for some formula $\theta(x)\in p$. Then for all $a\models p$ we have \[(S^\pp_a)^2\subseteq (S_a^\theta)^2\cap p(\Mon)= S_a^\theta\cap p(\Mon)= S^\pp_a\subseteq (S^\pp_a)^2,\] 
from which $(S^\pp_a)^2=S^\pp_a$ follows.

\smallskip
(2)$\Rightarrow$(1) Assume $(\mathcal S^{\mathbf p})^2=\mathcal S^{\mathbf p}$, that is, $(S^{\mathbf p}_x)^2= S^{\mathbf p}_x$ holds for all $x\models p$. To verify (1), by Remark \ref{Remark 1.6} it suffices to prove the following
\[
\mbox{for all $x<y$ realizing $p$: \ \ $y\in S^{\mathbf p}_x$ \ { if and only if } \  $\supin (S^{\mathbf p}_x)=\supin (S^{\mathbf p}_y)$,}
\]
where the initial parts are taken in $(p(\Mon),<)$. 
The right-to-left implication is clear. To prove the other, let $a\models p$ and let $b\in S^{\mathbf p}_a$. First, we {\it claim} that for all $z\models p$, $z\in S^{\mathbf p}_a$ implies\  $\supin (S^{\mathbf p}_z)\subseteq \supin (S^{\mathbf p}_a)$. Otherwise, there would be some $t\in S^{\mathbf p}_z$ with $S^{\mathbf p}_a<t$, so $t\in S^{\mathbf p}_z\land z\in S^{\mathbf p}_a$ would imply
$t\in (S^{\mathbf p}_a)^2\smallsetminus S^{\mathbf p}_a$, which contradicts (2). Therefore, $\supin (S^{\mathbf p}_z)\subseteq\supin (S^{\mathbf p}_a)$, proving the claim. 

Since $b\in S^{\mathbf p}_a$, by the claim we have $\supin (S^{\mathbf p}_b)\subseteq  \supin (S^{\mathbf p}_a)$, and hence $S^{\mathbf p}_b\subseteq S^{\mathbf p}_a$. 
It remains to rule out $\supin (S^{\mathbf p}_b)\subset\supin (S^{\mathbf p}_a)$. 
If $\supin (S^{\mathbf p}_b)\subset\supin (S^{\mathbf p}_a)$ were true, there would be some $c\in S^{\mathbf p}_a\smallsetminus S^{\mathbf p}_b$ such that $S^{\mathbf p}_b<c$. By the claim, $c\in S^{\mathbf p}_a$ implies $\supin (S^{\mathbf p}_c)\subseteq \supin (S^{\mathbf p}_a)$, so $S^{\mathbf p}_c\subseteq \supin (S^{\mathbf p}_a)$.
Thus, $S^{\mathbf p}_b\cup S^{\mathbf p}_c\subseteq S^{\mathbf p}_a$ and $S^{\mathbf p}_b<c$; this contradicts Fact \ref{Fact_lemma 3_1 wom} and proves (1).  
\end{proof}
In the context of weakly o-minimal theories, the equivalence of conditions (1) and (2) from Lemma \ref{Lemma two conditions for S_a E_p determined} was first proved by Baizhanov and Kulpeshov in \cite{Baizhanov2006}.

\begin{Example}[A monotone family $\mathcal S$ such that $\mathcal S^\pp$ is $\mathcal E_p$-determined, but $S_a$ is an $\mathcal S$-shift for all $a\models p$]
Consider the structure $(\mathbb Z, <,P)$, where $P$ is the set of all integers divisible by 3. This structure is a colored linear order, so its theory is weakly quasi-o-minimal. Let $p\in S_1(T)$ be the type of elements divisible by 3. Define $S_x:=\{x,x+1\}$. Clearly, $\mathcal S= (S_x\mid x\in\mathbb Z)$ is a monotone family of semiintervals and $S_x^n=\{x,x+1,\ldots,x+n\}$; in particular, $S_0$ is a $\mathcal S$-shift. On the other hand, the $\mathbf p$-semiinterval $S_0^{\mathbf p}=S_0\cap p(\Mon)=\{0\}$ is obviously $\mathcal E_p$-determined; so is $\mathcal S^\pp$.  
\end{Example}

Our next aim is to prove that every $\mathbf p$-semiinterval that is not $\mathcal E_p$-determined can be slightly enlarged so that the new $\mathbf p$-semiinterval is induced by a shift of a monotone definable family of semiintervals. We do this in Proposition \ref{Prop_p_notssi_exists_shift} below, but before that, we clarify the meaning of two $\mathbf p$-semiintervals having approximately the same size.
For $\mathcal S^\mathbf p$, a family of $\mathbf p$-semiintervals, and $S^{\mathbf p}_a\in \mathcal S^{\mathbf p}$, define
\begin{itemize}
    \item $\mathcal E(\mathcal S^{\mathbf p})=\mathcal E(S^{\mathbf p}_a):=\{E\in\mathcal E_p\mid S^{\mathbf p}_a\subseteq [a]_E\}$.
\end{itemize}
Clearly, $\mathcal E(\mathcal S^{\mathbf p})$ is the final part of $(\mathcal E_p,\subseteq)$. Two $\mathbf p$-semiintervals, say $S_a^{\mathbf p}$ and $R_a^{\mathbf p}$, have approximately the same size if $\mathcal E(S_a^{\mathbf p})=\mathcal E(R_a^{\mathbf p})$; this is illustrated in the following remark.  

\begin{Remark} Let $\mathbf p=(p,<)$ be a weakly o-minimal pair over $\emptyset$ and let $S_a^\pp$ and $R_a^\pp$ be $\mathbf p$-semiintervals.

(a) Assume that $(\mathcal E_p,\subseteq)$ is countable and let $d$ be an ultrametric on $p(\Mon)$ as defined in Remark \ref{Remark ultrametric on p(Mon)}. We can define
the length of $\mathbf p$-semiintervals by: \[\lgh(S_a^\pp)=\max \{d(a,x)\mid x\in S_a^\pp\}.\] 
Here, observe that the set of all extensions of $p$ which are consistent with $x\in S^{\mathbf p}_a$ is a closed subset of $S_x(Aa)$ that is totally ordered by $<$, so by compactness the maximum exists. Note that if $p$ has simple semiintervals, then $S_a^\pp=R_a^\pp$ if and only if $\lgh(S_a^\pp)=\lgh(R_a^\pp)$.  

(b) If $\mathcal E(S_a^\pp)\subset \mathcal E(R_a^\pp)$, then $R_a^{\mathbf p}$ is considerably shorter than $S_a^{\mathbf p}$ in the following sense: $\bigcup_{n\in\mathbb N}(R^\pp_a)^n \subset S^\pp_a$. To verify this, let $E\in \mathcal E(R_a^\pp)\smallsetminus \mathcal E(S_a^\pp)$. By induction, it is straightforward to prove $(R^\pp_a)^n\subset [a]_E$ for all $n\in\mathbb N$, which implies $\bigcup_{n\in\mathbb N}(R^\pp_a)^n \subset S^\pp_a$. 
\end{Remark}

The following lemma is the main part of the proof of Proposition \ref{Prop_p_notssi_exists_shift}.

\begin{Lemma}\label{Lemma_notEdetermined_implies_exists_p_monotone}
Let $\mathbf p=(p,<)$ be a weakly o-minimal pair over $A$ and let $\mathcal S^{\mathbf p}=(S^\pp_x\mid x\in p(\Mon))$ be a family of $\mathbf p$-semiintervals that is not $\mathcal E_p$-determined. There exist an $A$-definable extension $(D,<)$ of $(p(\Mon),<)$ and a monotone $A$-definable family of semiintervals $\mathcal R=(R_x\mid x\in D)$ such that $S_x^{\mathbf p}\subseteq R^{\mathbf p}_x$ (for all $x\in p(\Mon)$), $\mathcal R^{\mathbf p}$ is not $\mathcal E_p$-determined, and $\mathcal E(\mathcal S^{\mathbf p})=\mathcal E(\mathcal R^{\mathbf p})$. 
\end{Lemma}
\begin{proof}
 Without loss of generality, assume $A=\emptyset$ and we work in $\Moneq$. Suppose that the family $\mathcal S^{\mathbf p}$  is not $\mathcal E_p$-determined; in particular, it consists of bounded $\mathbf p$-semiintervals. Fix $a\models 
 p$ and the following objects:
\begin{itemize}
    \item $(D,<)$, a $0$-definable extension of $(p(\Mon),<)$;
    \item $\mathcal S=(S(x,D)\mid x\in D)$, a $0$-definable family of semiintervals of $(D,<)$ with $S^{\mathbf p}_a=S(a,p(\Mon))$; 
    \item for $E\in \mathcal E_p$, let $\hat E$ be a $0$-definable convex equivalence relations on $(D,<)$ with $\hat E_{\strok p(\Mon)}=E$.
\end{itemize}
Note that the set $D$ and the family $\mathcal S$ exist by Lemma \ref{Lemma_new_defextof_seminiterval}. For each $E\in\mathcal E_p$, by compactness there is a definable subset $D'\subseteq D$ containing $p(\Mon)$ and a definable convex equivalence relation $\hat E'$ on $D'$ extending $E$; we may now define $\hat E(x,y)$ by $x=y\lor (\exists u,v\in D')(\hat E'(u,v)\land u\leqslant x,y\leqslant v)$.
Next, we introduce the following notation:
\begin{itemize}
    \item $\Pi(x,y):=p(x)\cup p(y)\cup\{\hat E(x,y)\mid \hat E\in \mathcal E(\mathcal S^{\mathbf p})\}$;
    \item $B:=\{[a]_{\hat E}\mid E\in\mathcal E(\mathcal S^{\mathbf p})\}$;   \item $q(y):= \tp(a/B)$ and $\mathbf q=(q,<)$.
\end{itemize}
Clearly, $\Pi(x,y)$ is a partial type over $\emptyset$ and $B\subset \dcleq(a)$. In addition, it is easy to see that $\Pi(x,y)$ defines an equivalence relation on $p(\Mon)$. First, we {\em claim}:
\begin{enumerate}[(1)]
\item\label{Item1} \ $\Pi(a,\Mon)=\bigcap\{[a]_E\mid E\in \mathcal E (\mathcal S^{\mathbf p})\}=q(\Mon)$; in particular, $\Pi(a,\Mon)$ is a convex subset of $p(\Mon)$. 
\end{enumerate}
Here, the first equality and $q(\Mon)\subseteq \bigcap([a]_E\mid E\in \mathcal E (\mathcal S^{\mathbf p}))$ are immediate. To justify the reverse inclusion, assume $a'\in \bigcap\{[a]_E\mid E\in \mathcal E (\mathcal S^{\mathbf p})\}$, and we need to prove $a'\in q(\Mon)$. Note that $a,a'\models p$ implies that there is some $f\in \Aut(\Moneq)$ with $f(a)=a'$. Since $[a]_{\hat E}=[a']_{\hat E}$ for all $E\in\mathcal E (\mathcal S^{\mathbf p})$, $f$ fixes pointwise $B$. Then $f(a,B)=(a',B)$ implies $a'\models q$, as desired. This completes the proof of \ref{Item1}. Next, we {\em claim}: 
\begin{enumerate}[(1)]
\setcounter{enumi}{1}
    \item\label{Item2} $S^{\mathbf p}_a$ is a bounded $\mathbf q$-semiinterval that is not $\mathcal E_q$-determined. 
\end{enumerate}
Let $E\in \mathcal E (\mathcal S^{\mathbf p})$, that is, $S_a^{\mathbf p}\subseteq [a]_E$. Since $S^{\mathbf p}_a$ is not $\mathcal E_p$-determined, it is a bounded subset of $[a]_E$; this holds for all $E\in \mathcal E (\mathcal S^{\mathbf p})$ so, by compactness and saturation of $\Mon$, $S^{\mathbf p}_a$ is a bounded subset of $\bigcap\{[a]_E\mid E\in \mathcal E (\mathcal S^{\mathbf p})\}$, which, by \ref{Item1}, is equal to $q(\Mon)$. Hence, $S^{\mathbf p}_a$ is a bounded $\mathbf q$-semiinterval. Since $B\subset \dcleq(a)$, we can apply Corollary \ref{Fact_imaginaries_Epqdetermined_semiintervals}(b) to conclude $S^{\mathbf p}_a$ is not $\mathcal E_q$-determined. This proves \ref{Item2} whose immediate consequence is $S^{\mathbf p}_a\subseteq \mathcal D_q(a)$, which clearly implies the following:
\begin{enumerate}[(1)]
\setcounter{enumi}{2} 
    \item\label{Item3}  $a\triangleleft^\mathbf q a'$ implies  $S^{\mathbf p}_a<a'$. 
\end{enumerate}
An important place in this proof is to establish the following two properties; in fact, the type $q$ was introduced to satisfy them. 
\begin{enumerate}[(1)]
\setcounter{enumi}{3} 
    \item\label{Item4}  If $E\in\mathcal E_q$ and $S^{\mathbf p}_a\subseteq [a]_E$, then \ $E=q(\Mon)^2$. 
    \item\label{Item5} If $X^{\mathbf p}_a$ is a bounded $\mathbf q$-seminiterval with $S^{\mathbf p}_a\subseteq X^{\mathbf p}_a$, then $X^{\mathbf p}_a$ is not $\mathcal E_q$-determined.  
\end{enumerate}
To prove \ref{Item4}, assume $E\in \mathcal E_q$ and $S^{\mathbf p}_a\subseteq [a]_E$. By Corollary \ref{Fact_imaginaries_Epqdetermined_semiintervals}(a) we know that $E$ is equal to the restriction of some relation, say $E'$, from $\mathcal E_p$. Then $S^{\mathbf p}_a\subseteq [a]_{E'}$ implies $E'\in \mathcal E (\mathcal S^{\mathbf p})$, which by \ref{Item1} yields $q(\Mon)\subseteq [a]_{E'}$. Therefore, $E=q(\Mon)^2$, proving \ref{Item4}. To prove \ref{Item5}, suppose that some $\mathbf q$-semiinterval $X^{\mathbf p}_a$ is determined by $F\in \mathcal E_q$ and $S^{\mathbf p}_a\subseteq X^{\mathbf p}_a$. Then $X^{\mathbf p}_a$ is a final part of $[a]_F$, so by \ref{Item4} $S^{\mathbf p}_a\subseteq [a]_F$ implies $F=\mathbf 1_q$, so $X^{\mathbf p}_a$ is not bounded. 
Therefore, no bounded $\mathbf q$-semiinterval $X^{\mathbf p}_a$ which contains $S^{\mathbf p}_a\subseteq X^{\mathbf p}_a$ is $\mathcal E_q$-determined. 

\smallskip 
Consider $\supin_D(S_x)$, the smallest initial part of $(D,<)$ that contains $S_x$. Clearly, $t\in \supin_D(S_x)$ is $0$-definable, say by formula $I(x,t)$, and for each $x\in D$: $I(x,\Mon)$ is an initial part of $(D,<)$, $x\in I(x,\Mon)$, and $S(x,\Mon)$ is a final part of $I(x,\Mon)$. In particular, we have:
\begin{enumerate}[(1)]
\setcounter{enumi}{5} 
\item\label{Item6} $I(x,p(\Mon))$ is an initial part of $p(\Mon)$ and $S^{\mathbf p}_x$ is a final part of $I(x,p(\Mon))$;
\item\label{Item7} $a\triangleleft^\mathbf q a'$ implies $I(a,\Mon)<a'$.
\end{enumerate}
To justify \ref{Item7}, assume $a\triangleleft^\mathbf q a'$. By \ref{Item3} we know $S^{\mathbf p}_a<a'$. Since, by \ref{Item6}, $S^{\mathbf p}_a$ is a final part of $I(a,p(\Mon))$, we conclude $I(a,p(\Mon))<a'$, which implies $a'\in I(a',p(\Mon))\smallsetminus I(a,p(\Mon))$. Since $I(a,\Mon)$ and $I(a',\Mon)$ are initial parts of $D$, and since $a'\models p$, we conclude $I(a,\Mon)<a'$. 

\smallskip
Consider the function $I:q(\Mon)\to \overline D^{\,in}$ given by $x\mapsto I(x,\Mon)$; clearly, $I$ is relatively $B$-definable.  To see that $I$ is non-constant, choose $a'$ with $a\triangleleft^\mathbf qa'$. Then $I(a,\Mon)<a'$ holds by \ref{Item7}, and $a'\in I(a',\Mon)$ by \ref{Item6}; these two together imply $I(a,\Mon)\subset I(a',\Mon)$, so $I$ is non-constant. This allows us to apply the upper monotonicity theorem \cite[Theorem 3.6]{MTwom}: There exists  $F_q\in\mathcal E_q\smallsetminus\{\mathbf 1_q\}$ such that:
\begin{enumerate}[(1)]
\setcounter{enumi}{7}
\item\label{Item8} $q(x)\cup q(y)\cup \{[x]_{F_q}<[y]_{F_q}\}\vdash I(x,\Mon)\subset I(y,\Mon).$
\end{enumerate}
By Corollary \ref{Fact_imaginaries_Epqdetermined_semiintervals}(a), there exists $F\in\mathcal E_p$ such that $F_{\strok q(\Mon)}=F_q$. Then, since both $[a]_F$ and $q(\Mon)$ are convex subsets of $p(\Mon)$, and since $[a]_{F_q}=[a]_F\cap q(\Mon)$ is a bounded subset of $q(\Mon)$, we deduce $[a]_{F}=[a]_{F_q}$. In particular, $[a]_F$ is a bounded subset of $q(\Mon)=\Pi(a,\Mon)$; here the equality follows by \ref{Item1}. The key observation is that $F_q\neq \mathbf 1_q$, by \ref{Item4}, implies $S^{\mathbf p}_a\not\subseteq [a]_{F_q}$. This implies $S^{\mathbf p}_a\not\subseteq [a]_{F}$, so $F\notin \mathcal E(\mathcal S^{\mathbf p})$. This proves: 
\begin{enumerate}[(1)]
\setcounter{enumi}{8}
\item\label{Item8'} $F\notin \mathcal E(\mathcal S^{\mathbf p})$ and $[x]_F$ is a bounded subset of $\Pi(x,\Mon)$ for all $x\models p$. 
\end{enumerate}
Next, we prove:
\begin{enumerate}[(1)]
\setcounter{enumi}{9}
\item\label{Item9}  \ $p(x)\cup p(y)\cup \{[x]_{F}<[y]_{F}\}\vdash I(x,\Mon)\subset I(y,\Mon).$
\end{enumerate}
First, note that, unlike in \ref{Item8}, no parameters appear in \ref{Item9}, so it suffices to prove it with $x$ replaced by $a$. Let $a'\models p$ be such that $[a]_{F}<[a']_{F}$; we need to show $I(a,\Mon)\subset I(a',\Mon)$. We have two cases to consider. The first is where $a'\models q$, in which case $I(a,\Mon)\subset I(a',\Mon)$ is a consequence of \ref{Item8}. In the second case, since $q(\Mon)$ is a convex subset of $p(\Mon)$, $a'\not\models q$ and $a<a'$ together imply $q(\Mon)<a'$, so $\Pi(a,\Mon)<a'$ holds by \ref{Item1}. Since $\Pi(x,y)$ (type-) defines an equivalence relation on $p(\Mon)$, and since the classes of this relation are convex (again by \ref{Item1}), $q(\Mon)<a'$ implies $\Pi(a,\Mon)<\Pi(a',\Mon)$; as $S^{\mathbf p}_a$ is a bounded subset of $\Pi(a,\Mon)$ and $S^{\mathbf p}_{a'}$ a bounded subset of $\Pi(a',\Mon)$, we derive $S^{\mathbf p}_{a}<S^{\mathbf p}_{a'}$. Finally, since $S^{\mathbf p}_{a}$ is a final part of $I(a,\Mon)$, $a'\models p$, and $S^{\mathbf p}_{a}<a'$, we conclude $I(a,p(\Mon))<a'$. Clearly, this implies $I(a,\Mon)\subset I(a',\Mon)$, which completes the proof of \ref{Item9}.

Consider $\hat F$, a $0$-definable convex extension of $F$ to $(D,<)$. Then \ref{Item9} implies:
\[ \models (\forall x,y\models p)([x]_{\hat F}<[y]_{\hat F}\rightarrow I(x,\Mon)\subset I(y,\Mon)).\]
By compactness, there is a formula $\theta(x)\in p$ that implies $x\in D$ such that the following holds:
\begin{enumerate}[(1)]
\setcounter{enumi}{10}
\item\label{Item11}  \ $\models (\forall x,y\in \theta(\Mon))([x]_{\hat F}<[y]_{\hat F}\rightarrow I(x,\Mon)\subset I(y,\Mon)).$
\end{enumerate}
To simplify the notation used below, we will assume that $\theta(x)$ is $x\in D$; alternatively, we can replace $D$ with $\theta(D)$. Note that this does not affect the validity of \ref{Item1}--\ref{Item11}. We now modify the formula $I(x,t)$. Choose an $L$-formula $J(x,t)$ that expresses $t\in \bigcup_{z\in [x]_{\hat F}}I(z,\Mon)$; clearly, $J(x,\Mon)$ is the smallest initial part of $D$ that contains all initial parts $I(z,\Mon)$ for $z\in [x]_{\hat F}$. Consider the function $J:D\to \overline{D}^{\,in}$, defined by $x\mapsto J(x,\Mon)$. Clearly, $I(x,\Mon)\subseteq J(x,\Mon)$ holds for all $x\in D$. From \ref{Item11} it is easy to deduce that $J$ is increasing. 

Define \ $R(x,t):=x\leqslant t\land J(x,t)$. Then $\mathcal R=(R(x,\Mon)\mid x\in D)$ is a family of semiintervals of $(D,<)$, and 
$R(x,\Mon)$ is a final part of $J(x,\Mon)$ for all $x\in D$; $\mathcal R$ is monotone as $J$ is so. 
Consider the $\mathbf p$-semiinterval $R^{\mathbf p}_a=R(a,p(\Mon))$.
Since, by \ref{Item6}, $I(x,p(\Mon))$ is an initial part of $p(\Mon)$ and $S^{\mathbf p}_x$ is a final part of $I(x,p(\Mon))$, and since $I(x,\Mon)\subseteq J(x,\Mon)$, we conclude $S^{\mathbf p}_a\subseteq R^{\mathbf p}_a$. Next, we {\em claim}: 
\begin{enumerate}[\hspace{10pt}(1)]
\setcounter{enumi}{11}
\item\label{Item12}  \  $R^{\mathbf p}_a$ is a bounded $\mathbf q$-semiinterval. 
\end{enumerate}
To prove it, choose $a',a''$ with $a\triangleleft^{\mathbf q}a'\triangleleft^{\mathbf q}a''$, and it suffices to show $R^{\mathbf p}_a< a''$.
By \ref{Item8'} we know that the class $[a]_F$ is a bounded subset of $\Pi(a,\Mon)=q(\Mon)$, hence $[a]_F\subseteq \mathcal D_q(a)$, so $a\triangleleft^{\mathbf q}a'$ implies $[a]_F<a'$, which yields $[a]_F<[a']_F$ and hence $[a]_{\hat F}<a'$. From \ref{Item11} we derive $I(t,\Mon)\subset I(a',\Mon)$ for all $t\in [a]_{\hat F}$, from which $J(a,\Mon)=\bigcup_{t\in[a]_{\hat F}}I(t,\Mon)\subseteq I(a',\Mon)$ follows and together with $I(a',\Mon)<a''$ from \ref{Item3} implies $J(a,\Mon)<a''$. Then $R^{\mathbf p}_a\subseteq J(a,\Mon)$ implies $R^{\mathbf p}_a<a''$, as needed.

\smallskip
Since $R^{\mathbf p}_a$ is a bounded $\mathbf q$-semiinterval and $S^{\mathbf p}_a\subseteq R^{\mathbf p}_a$, by \ref{Item5} we have that the $\mathbf q$-semiinterval $R^{\mathbf p}_a$ is not $\mathcal E_q$-determined; consequently, it is not $\mathcal E_p$-determined. By now, we have found an 0-definable extension $(D,<)$ of $(p(\Mon),<)$ and a monotone 0-definable family of semiintervals $\mathcal R=(R_x\mid x\in D)$ such that $\mathcal R^{\mathbf p}$ is not $\mathcal E_p$-determined and $S_x^{\mathbf p}\subseteq R^{\mathbf p}_x$ (for all $x\in p(\Mon)$). To complete the proof of the lemma, it remains to verify $\mathcal E(\mathcal S^{\mathbf p})=\mathcal E(\mathcal R^{\mathbf p})$. Let $E\in\mathcal E(R^{\mathbf p}_a)$. Then $S_a^{\mathbf p}\subseteq R^{\mathbf p}_a\subseteq [a]_E$ implies $E\in \mathcal E(S^{\mathbf p}_a)$; hence $\mathcal E(R^{\mathbf p}_a) \subseteq \mathcal E(S^{\mathbf p}_a)$. The reverse inclusion follows from $R^{\mathbf p}_a\subseteq q(\Mon)=\bigcap\{[a]_E\mid E\in \mathcal E (\mathcal S^{\mathbf p})\}$. This completes the proof.
\end{proof}

\begin{Proposition}\label{Prop_p_notssi_exists_shift}
Let $\mathbf p=(p,<)$ be a weakly o-minimal pair over $A$ and let $\mathcal S^{\mathbf p}=(S^\pp_x\mid x\in p(\Mon))$ be a family of $\mathbf p$-semiintervals that is not $\mathcal E_p$-determined. There exist an $A$-definable extension $(D,<)$ of $(p(\Mon),<)$ and a monotone $A$-definable family of semiintervals $\mathcal R=(R_x\mid x\in D)$ such that $S_x^{\mathbf p}\subseteq R^{\mathbf p}_x$, $\mathcal E(\mathcal S^{\mathbf p})=\mathcal E(\mathcal R^{\mathbf p})$, $R^{\mathbf p}_x$ is not $\mathcal E_p$-determined, and $R_x$ is an $\mathcal R$-shift for all $x\models p$.
\end{Proposition}
\begin{proof} 
Most of the proof is contained in Lemma \ref{Lemma_notEdetermined_implies_exists_p_monotone}, 
where an $A$-definable extension $(D,<)$ of $(p(\Mon),<)$ and a monotone $A$-definable family of semiintervals $\mathcal R=(R_x\mid x\in D)$ have been found such that: $\mathcal R^{\mathbf p}$ is not $\mathcal E_p$-determined, $\mathcal E(\mathcal S^{\mathbf p})=\mathcal E(\mathcal R^{\mathbf p})$, and $S_x^{\mathbf p}\subseteq R^{\mathbf p}_x$ (for all $x\in p(\Mon)$).
Since $\mathcal R^{\mathbf p}$ is not $\mathcal E_p$-determined, by Lemma \ref{Lemma two conditions for S_a E_p determined} we derive $(R^{\mathbf p}_a)^2\neq R^{\mathbf p}_a$ for all $a\models p$, so $R_x$ is an $\mathcal R$-shift by Lemma \ref{Lemma monotone family not Ep determined is a shift}.
\end{proof}

\begin{Theorem}\label{Thm_noshifts_iff_simplesemiintervals}
Let $T$ be a weakly quasi-o-minimal theory and let $<$ be a $0$-definable order on $\Mon$. 

(a) If $T$ satisfies (NS), then it also satisfies (SH).

(b) If $T$ is small or weakly o-minimal and satisfies (SH), then it also satisfies (NS).
\end{Theorem}
\begin{proof} 
(a) Assume that $p\in S_1(A)$ does not have simple semiintervals and let $\mathbf p=(p,<)$. By Proposition \ref{Prop_p_notssi_exists_shift}, there exists an $A$-definable extension $(D,<)$ of $(p(\Mon),<)$ and an $A$-definable monotone family $\mathcal R$ of semiintervals of $(D,<)$ such that $R_a$ is an  $\mathcal R$-shift for any $a\models p$. As $(D,<)$ is a definable extension of $(p(\Mon),<)$ we have $D\subseteq \Mon$, so by Lemma \ref{Lemma extending family of semiintervals}(b) $\mathcal R$ can be extended to a monotone $A$-definable family of semiintervals, $\mathcal S$, of $(\Mon,<)$ such that $S_a$ is an $\mathcal S$-shift.

(b) First, assume that $T$ is small. Assume that $\mathcal S=(S_x\mid x\in\Mon)$ is a definable family of semiintervals, $c_0\in\Mon$, and $S_{c_0}$ is an $\mathcal S$-shift. We need to prove that $T$ does not have simple semiintervals. 
The proof closely follows the first part of the proof of \cite[Theorem 1]{MTwqom1}. 
First, by \cite[Lemma 6.7]{MTwqom1} we can assume that $\mathcal S$ is monotone. Next, absorb in the language the parameter $c_0$ and a finite set of parameters needed to define the family $\mathcal S$; note that this does not affect the desired conclusion. Then redefine $\mathcal S$ by setting $S_x=\{x\}$ for all $x<c_0$, and modify the order $<$ by moving the interval $(-\infty, c_0)$ to the right end. Notice that the redefined family is still monotone and $0$-definable, that $c_0=\min \Mon$, and that $\mathcal S_{c_0}$ is still a shift. 
Denote $\mathcal C=\bigcup_{n\in\mathbb N} S_{c_0}^n$, and define
 \[\Pi(x)=\{\phi(x)\in L\mid \phi(\Mon)\text{ is convex and contains a final part of }\mathcal C\}.\] 
By \cite[Lemma 6.9(b,c)]{MTwqom1}, the set $\mathcal C\cup\Pi(\Mon)$ is an initial part of $(\Mon,<)$ and $\Pi(x)$ is the minimal interval type over $\emptyset$ consistent with $\mathcal C<x$. As $T$ is small, there is an isolated point, say $p$, in the space $S_\Pi=\{p\in S_1(T)\mid \Pi(x)\subseteq p(x)\}$; let $\theta(x)\in p$ isolate $p$, that is:
\begin{enumerate}[(1)]
    \item \ $\Pi(x)\cup\{\theta(x)\}\vdash p(x)$, \ or, \ $\Pi(\Mon)\cap \theta(\Mon)=p(\Mon)$. 
\end{enumerate} 
Let $a\models p$. Consider the set $D_a=\{x\in \Mon\mid S_a<x\land\theta(x)\}$; clearly, it is $a$-definable, $D_a\subseteq \theta(\Mon)$, and $D_a\cap \Pi(\Mon)\neq \emptyset$. By \cite[Lemma 6.12(b,c)]{MTwqom1}, there exists $b\in D_a$ such that the type $\tp(b/a)$ is isolated by $\phi(x,a)$ say, $\phi(\Mon,a)$ is a bounded subset of $\Pi(\Mon)$, and $b$ is an $\mathcal S$-minimal element\footnote{$\mathcal S$-minimality is defined in \cite[Definition 6.11]{MTwqom1}.} of $D_a$; the latter means that for some $d\in \Mon$ we have: 
\begin{enumerate}[(1)]
\setcounter{enumi}{1}
    \item  \ $b\in S_d$  \ \  and \ \ $d<D_a$. 
\end{enumerate} 
Since $\phi(\Mon,a)$ is a bounded subset of $\Pi(\Mon)$ and since $\phi(\Mon,a)\subseteq \theta(\Mon)$, by (1), we deduce that $\phi(\Mon,a)$ is a bounded subset of $p(\Mon)$ with $\phi(\Mon,a)\subseteq D_a$ and $S_a<\phi(\Mon,a)$.

Define $R_a=\bigcup_{t\in \phi(\Mon,a)}([a,t]\cap \theta(\Mon))$; clearly,  $R_a$ is an $a$-definable $\mathbf p$-semiinterval and $\phi(\Mon,a)$ is its final part. 
Thus, $\mathcal R^\pp=(R_x\mid x\models p)$ is a family of definable $\mathbf p$-semiintervals. We {\em claim} that $\mathcal R^\pp$ is not $\mathcal E_p$-determined; by Lemma \ref{Lemma two conditions for S_a E_p determined}, we need to prove $(\mathcal R^\pp)^2\neq \mathcal R^\pp$ or, equivalently, $R_a^2\neq R_a$. Pick $c\in \phi(\Mon,b)$; clearly, $c\in D_b$ implies $S_b<c$. Also, since $\phi(\Mon,b)$ is a final part of $R_b$, we have $c\in R_b$, which together with $b\in R_a$ implies $c\in R_a^2$. To finish the proof of the claim, it is enough to prove $c\notin R_a$. As $\phi(\Mon,a)$ is a final part of $R_a$, $b\in \phi(\Mon,a)$, and $b<c$, it suffices to show $c\notin\phi(\Mon,a)$; otherwise, $c\in\phi(\Mon,a)$ would imply $c\equiv b\,(a)$ because $\phi(x,a)$ isolates $\tp(b/a)$, so there would be some $e\in\Mon$ such that $ce\equiv bd\,(a)$, and then, by (2), we would have $c\in S_e$ and $e<D_a$. Note that $e<D_a$ and $b\in D_a$ together imply $e<b$, so by the monotonicity of $\mathcal S$ we obtain $\supin (S_e)\subseteq \supin (S_b)$, which together with $S_b<c$ implies $S_e<c$; this contradicts $c\in S_e$, proving the claim.

Therefore, the family $\mathcal R^\pp$ is not $\mathcal E_p$-determined, so $T$ does not have simple semiintervals. This completes the proof in the case where $T$ is small.

Now, assume that $T$ is weakly o-minimal and that $\mathcal P=(P_x\mid x\in\Mon)$ is a 0-definable family of semiintervals that contains a shift; again, we may assume that $\mathcal P$ is monotone. Consider the partial type \ $\Sigma(x)=\{P_x^n\subset P_x^{n+1}\mid n\in\mathbb N\}$. Note that $\Sigma(\Mon)$ is the set of all elements for which $P_x$ is a shift. The set $S_{\Sigma}=\{p\in S_1(T)\mid \Sigma(x)\subset p(x)\}$ is a closed subset of $S_1(T)$, which is linearly ordered by $<$ as $T$ is weakly o-minimal. Therefore, it contains a maximal element, say $q\in S_1(T)$. By the maximality, we have $P_x\subset q(\Mon)$ for all $x\models q$, so $P_x$ is a $\mathbf q$-semiinterval for which $P_x^2\neq P_x$ holds. By Lemma \ref{Lemma two conditions for S_a E_p determined}, $P_x$ is not $\mathcal E_q$-determined. This completes the proof of the theorem. 
\end{proof}

The next example shows that the converse of Theorem \ref{Thm_noshifts_iff_simplesemiintervals}(a) is not valid in general.  

\begin{Example}[A weakly quasi-o-minimal theory with simple semiintervals and a definable shift]\label{Example Z < many colors}
Consider the theory $T$ of the structure $\mathcal Z=(\mathbb Z,<,R_{q,k})_{\substack{q\text{ prime}\\0\leqslant k<q}}$, where $R_{q,k}= q\mathbb Z+k$. Since $\mathcal Z$ is a colored linear order, $T$ is weakly quasi-o-minimal. It is routine to check that after adding the successor function, $T$ eliminates quantifiers. For each sequence $\bar k= (k_q)_{q\text{ prime}}$, where $0\leqslant k_q<q$ for all $q$, we have a unique type $p_{\bar k}\in S_1(\emptyset)$ that contains $\{R_{q,k_q}(x)\mid q\text{ prime}\}$. These are all $1$-types over $\emptyset$ as for each prime $q$ each element is colored by exactly one of the colors $R_{q,k}$($0\leqslant k<q$). Note that for any $a,b\in\Mon$ at a finite distance, $\tp(a)=\tp(b)$ if and only if $a=b$, so, using elimination of quantifiers, it is routine to check that for every $p\in S_1(A)$ every bounded $\mathbf p$-semiinterval is a singleton; $T$ has simple semiintervals. On the other hand, the successor function produces a shift on $(\Mon,<)$. 
\end{Example}

\begin{Theorem}[$I(T,\aleph_0)<2^{\aleph_0}$]\label{Theorem convex few ctble has hssi}
Every convex weakly o-minimal type over a finite domain has hereditarily simple semiintervals.
\end{Theorem}
\begin{proof}
Let $p\in S(A)$ be a convex weakly o-minimal type, where $A$ is finite. Assuming that some extension of $p$, say $q\in S(B)$, does not have simple semiintervals, we will prove $I(T,\aleph_0)=2^{\aleph_0}$. By Lemma \ref{Lemma_notEdetermined_implies_exists_p_monotone} there is some $B$-definable extension, $(D,<)$, of $(q(\Mon),<)$ and a monotone $B$-definable family $\mathcal S=(S_x\mid x\in D)$ of semiintervals of $(D,<)$ such that $S^{\mathbf q}_a$ is not $\mathcal E_q$-determined for some (all) $a\models q$. Since $q$ is convex, by \cite[Proposition 3.4]{MTwqom1}, there is a formula $\theta(x)\in q$ such that $q(\Mon)$ is convex in $(\theta(D), <))$. Denote by $\mathcal R$ the induced family $\mathcal S^{\theta}$, such that $R_x=\theta(D)\cap S_x$ for all $x\in\theta(D)$. Note that the family $\mathcal R$ is monotone and that, by Lemma \ref{Lemma monotone family not Ep determined is a shift}, $R_x$ is a shift for all $x\models q$. 

Fix $a\models q$, and choose $b_1$ and $b_2$ to realize $q$ with $b_1<R^\omega_a<b_2$. Let $\mathcal R'=(R'_x\mid x\in [b_1,b_2]_{\theta(D)})$ be the family induced by $\mathcal R$ on the interval $[b_1,b_2]_{\theta(D)}$: $R_x'=R_x\cap [b_1,b_2]_{\theta(D)}$. 
Then $\mathcal R'$ is monotone and, by the convexity of $q(\Mon)$ within $(\theta(\Mon),<)$, we have $[b_1,b_2]_{\theta(D)}\subseteq p(\Mon)$.  By our choice of $b_1$ and $b_2$, we have $R_a^n\subseteq [b_1,b_2]_{\theta(D)}$, so $(R'_a)^n=R_a^n$ for all $n\in \mathbb N$; hence $R'_a$ is an $\mathcal R'$-shift. Therefore, we have a definable family of semiintervals of $[b_1,b_2]_{\theta(D)}$ that contains a shift. Let $A_0\supseteq A$ be a finite set of parameters needed to define the interval $[b_1,b_2]_{\theta(D)}$. Note that all complete types $r\in S_1(A_0)$ consistent with $x\in [b_1,b_2]_{\theta(D)}$ are weakly o-minimal, since  they extend $p$. In the terminology of \cite[Definition 4.8]{MTwom}, this means that the segment $[b_1,b_2]_{\theta(D)}$ is weakly quasi-o-minimal over $A_0$. Therefore, we have a definable set that is weakly quasi-o-minimal over a finite parameter set and a definable family of semiintervals that contains a shift; in this situation, \cite[Theorem 6.17]{MTwqom1} guarantees $I(T,\aleph_0)=2^{\aleph_0}$. This completes the proof.  
\end{proof}

Recall that \cite[Theorem 1]{MTwqom1} states that any weakly quasi-o-minimal theory $T$ that satisfies (SH) has $2^{\aleph_0}$ countable models. By Theorem \ref{Thm_noshifts_iff_simplesemiintervals}, conditions (SH) and (NS) are equivalent for small $T$, and, in particular, for theories with few countable models. 

\begin{Theorem}\label{Theorem few models implies simple semiintervals}
A weakly quasi-o-minimal theory with few countable models has simple semiintervals. 
\end{Theorem}


\section{Condition (R)}\label{Section_(R)}

In this section, we study weakly quasi-o-minimal theories that satisfy the following condition. 
\begin{enumerate}[\hspace{10pt}(1)]
\item[(R)] Every equivalence relation $E\in \mathcal E_p$ is relatively 0-definable for all $A\subseteq \Mon$ and all $p\in S_1(A)$.
\end{enumerate}
We prove that (R) is satisfied in all quasi-o-minimal theories and all weakly quasi-o-minimal theories with simple semiintervals that are, in addition, rosy or have a finite convexity rank. 
The most technically demanding result of this section is Proposition \ref{Prop_equivalents_of_R}, in which we prove that a weakly quasi-o-minimal theory with simple semiintervals $T$ satisfies (R) if and only if it is binary; this suffices to confirm Vaught's conjecture for $T$.

\begin{Proposition}\label{Prop_quasio_satisfy_R} 
Quasi-o-minimal theories satisfy condition (R).
\end{Proposition}
\begin{proof}
Suppose that $T$ is quasi-o-minimal with respect to $<$, and let $p_A\in S_1(A)$ and $E_A\in\mathcal E_p$, and we prove that $E_A$ is relatively 0-definable. We may assume $E_A\neq\mathrm{id}_{p_A(\Mon)}$ and $E_A\neq \mathbf 1_p$. 
Let $p_0=(p_A)_{\strok\emptyset}$, and fix $a\models p_A$.
Since $p_A(\Mon)$ is a convex subset of $p_0(\Mon)$ and $[a]_E$ is a bounded and convex subset of $p_A(\Mon)$, $[a]_E$ is a bounded and convex subset of $p_0(\Mon)$, and it is relatively $Aa$-definable within $p_0(\Mon)$. We prove that there are $\theta(x)\in p_0(x)$ and $b,c\in\Mon$ such that:
\setcounter{equation}{0}
\begin{equation}
[a]_{E_A}=(\theta(\Mon)\cap (b,c))\cap p_0(\Mon).
\end{equation}
As $T$ is quasi-o-minimal, the formula relatively defining $[a]_{E_A}$ within $p_0(\Mon)$ can be chosen in the form $\bigvee_{i=1}^n\theta_i(x)\land x\in J_i$, where $\theta_i(x)$ are $L$-formulae and $J_i$ are intervals in $\Mon$. Since $[a]_{E_A}\subset p_0(\Mon)$, we may assume that $\theta_i(x)\in p_0(x)$ for all $i=1,\dots,n$, and note that $\theta(x)\land\bigvee_{i=1}^nx\in J_i$ still relatively defines $[a]_{E_A}$ within $p_0(\Mon)$, where $\theta(x)=\bigwedge_{i=1}^n\theta_i(x)$. Also, we may assume that each $J_i$ meets $[a]_{E_A}$. The left end, say $b$, of some $J_i$ satisfies $b\leqslant[a]_{E_A}$, and the right end, say $c$, of some $J_i$ satisfies $[a]_{E_A}\leqslant c$. It is easy to see that $[a]_{E_A}$ is relatively defined by $\theta(x)\land x\in J$ within $p_0(\Mon)$, where $J$ is some interval with endpoints $b$ and $c$. 
Now, note that the class $[a]_{E_A}$ does not have a minimal (maximal) element; otherwise, every realization of $p_A$ would be minimal (maximal) in its $E_A$-class, so we would have $E_A=\mathrm{id}_{p_A(\Mon)}$, which is not the case. Thus $b,c\notin[a]_{E_A}$, so we may assume that $J=(b,c)$, and we have established (1). 

By Lemma \ref{Lemma O EKVIVALENCIJAMA} it suffices to prove that the class $[a]_{E_A}$ is $a$-invariant. Toward a contradiction, suppose that there is $A'\equiv A\,(a)$ such that $[a]_{E_A}\neq [a]_{E_{A'}}$ (where $E_{A'}$ is the conjugate of $E_A$). Then, after possibly reversing the order and/or exchanging the roles of $A$ and $A'$, we can assume $\supin_{p_0(\Mon)} ([a]_{E_{A'}})\subset\supin_{p_0(\Mon)} ([a]_{E_A})$. 
%
%
%
Denote: \ $q_0=\tp(c)$, $q_A=\tp(c/A)$,  and $C_A(a)=\{x\in p_0(\Mon)\mid [a]_{E_A}<x\}$.  From (1) we derive:
\begin{equation}
[a]_{E_A}<c \ \text{ and } \  c\leqslant C_A(a).
\end{equation}
Note that the set $C_A(a)\cap p_A(\Mon)$ is an initial part of $C_A(a)$ that is closed for $E_A$-classes. Consequently, since no $E_A$-class contains a minimal element, neither does $C_A(a)$, so $c\notin C_A(a)$. Then $[a]_{E_A}<c<C_A(a)$ implies $c\notin p_0(\Mon)$, so $q_0\neq p_0$.

Choose $a_0c_0\equiv ac\,(A)$ with $[a_0]_{E_A}<[a]_{E_A}$. Then $c_0\models q_A$, $[a_0]_{E_A}<c_0$, and there are no realizations of $p_A$ between $[a_0]_{E_A}$ and $c_0$. Since $q_A\neq p_A$ and since $[a_0]_{E_A}<[a]_{E_A}$, we conclude $c_0<[a]_{E_A}<c$. Next, choose $c'\models q_0$ so that $c'A'\equiv cA\,(a)$. 
Note that $\supin_{p_0(\Mon)} ([a]_{E_{A'}})\subset\supin_{p_0(\Mon)} ([a]_{E_A})$ implies that $[a]_{E_{A'}}<d$ for some $d\in [a]_{E_A}$. 
Then $d\in C_{A'}(a)$, so $c'<d$ and $c'\in \conv([a]_{E_A})$ (the convex hull of $[a]_{E_A}$ in $(\Mon,<)$). 
We conclude that $c'$ divides the set $\conv([a]_{E_A})$ into two parts, each containing realizations of $p_A(\Mon)$. Also, $c'\in \conv([a]_{E_A})$ combined with $c_0<[a]_{E_A}<c$ yields $c_0<c'<c$. Since $q_A(\Mon)$ is a convex subset of $q_0(\Mon)$, $c_0<c'<c$ together with $c_0,c\models q_A$ and $c'\models q_0$ implies $c'\models q_A$; hence $c\equiv c'\,(A)$. By our choice of $c$, $[a]_{E_A}<c<C_A(a)$ implies that $c$ does not divide the convex hull of any $E_A$-class; on the other hand, $c'$ divides $\conv([a]_{E_A})$. This contradicts $c\equiv c'\,(A)$ and completes the proof of the proposition. 
\end{proof}

\begin{Remark}\label{Remark_R_for binary} It is rather straightforward to verify that condition (R) is satisfied in all binary first-order theories.    
\end{Remark}

Now, we examine weakly quasi-o-minimal theories in which condition (R) fails. 

\begin{Lemma}\label{Lemma_non_R} Let $<$ be a 0-definable linear order on $\Mon$, let $p_A\in S_1(A)$ be weakly o-minimal, and let $E_A\in\mathcal E_{p_A}$. Suppose that $E_A$ is not relatively $0$-definable and $(p_A(\Mon)/E_A,<)$ is a dense linear order. 
Then for all $b\models p_A$ there exists a sequence $(A_n\mid n\in\omega)$ such that $A_0=A$, $\tp(A_{n})=\tp(A)$, and the class $[b]_{E_{A_n}}$ contains infinitely many $E_{A_{n+1}}$-classes for all $n\in \omega$ (where $E_{A_n}$ is the conjugate of $E_A$).
\end{Lemma}
\begin{proof} First, we {\it claim} that for all $b\models p$ there exists an automorphism $f\in \Aut(\Mon)$ such that the $E_A$-class $[b]_{E_A}$ contains infinitely many $f(E_A)$-classes, one of which is $[b]_{E_{f(A)}}$. 

To prove the claim, fix $a\models p_A$ and let $p=\tp(a)$. Since $E_A$ is not relatively 0-definable, by Lemma \ref{Lemma O EKVIVALENCIJAMA} the class $[a]_{E_A}$ is not $a$-invariant, so there is an $f\in \Aut_a(\Mon)$ such that $f([a]_{E_A})\neq [a]_{E_A}$. Then $E_{f(A)} \in \mathcal E_{p_{f(A)}}$ and $f([a]_{E_A})=[a]_{E_{f(A)}}$, where $p_{f(A)}$ denotes the conjugate of $p_A$. The classes $[a]_{E_A}$ and $[a]_{E_{f(A)}}$ are distinct convex subsets of $p(\Mon)$, so the initial or final parts of $(p(\Mon),<)$ generated by them are distinct. Without loss we assume $\supin ([a]_{E_{f(A)}})\subset \supin ([a]_{E_A})$; the other cases are dealt with in a similar way. 
Since $(p_A(\Mon)/E_A,<)$ is a dense infinite order, for each $c>[a]_{E_A}$ there are infinitely many $E_A$-classes that satisfy $[a]_{E_A}<[x]_{E_A}<c$. In particular, $\supin ([a]_{E_{f(A)}})\subset \supin ([a]_{E_A})$ implies that there are infinitely many $E_{f(A)}$-classes within $[a]_{E_{A}}$; choose $b$ in one of them. Clearly, the class $[b]_{E_A}$ ($=[a]_{E_A}$) contains infinitely many $E_{f(A)}$-classes, one of which is $[b]_{E_{f(A)}}$. Since this holds for one $b\models p_A$, it does for all. This proves the {\it claim}. 

To finish the proof of the lemma, fix $b\models p$ and define by recursion: $A_0=A$, $A_{n+1}=f_n(A_n)$, where $f_n$ is given by the claim applied to $b$ and to $A_n$ in place of $A$.  
\end{proof} 

It follows that a weakly quasi-o-minimal theory satisfies (R) if and only if there do not exist a complete 1-type $p_A$, $b\models p_A$, and a sequence $(A_n\mid n\in\omega)$ that satisfy the conclusion of the lemma. This equivalent condition, in the context of any linearly ordered structure that also has simple semiintervals and all 1-types convex, was studied by Baizhanov and Zambarnaya in \cite{Baizhanov2025}.

\begin{Remark}\label{Rematk_finconv_rosywom} Consider the class of countable weakly quasi-o-minimal theories $T$. We have already noted that the binary $T$ and the quasi-o-minimal $T$ satisfy condition (R). Using Lemma \ref{Lemma_non_R} we obtain two new subclasses.

\smallskip 
(a) $T$ has a finite convexity rank and few countable models.\\
Kulpeshov in \cite{kulpeshov1998} introduced the convexity rank of a unary formula $\phi(x)$ in a weakly o-minimal theory, denoted by $RC(\phi(x))$, as follows:
\begin{enumerate}[(i)]
    \item   $RC(\phi(x))) \geqslant  0$, if $\phi(x)$ is consistent;
    \item   $RC(\phi(x))\geqslant n+1$, if there exist a parametrically deﬁnable convex equivalence relation $E(x, y)$ on $\Mon$ and an inﬁnite set $\{ b_i\mid i\in\omega\}$ of representatives of different $E$-classes such that $RC(\phi(x)\land E(x,b_i))\geqslant n$ for all $i\in\omega$;
    \item $RC(\phi(x))=n$ if $n$ is maximal satisfying $RC(\phi(x))\geqslant n$; if no such $n$ exists, $RC(\phi(x))=\infty$.   
\end{enumerate}
$T$ has finite convexity rank if $RC(x=x)<\omega$. Since $T$ has few countable models, Theorem \ref{Theorem few models implies simple semiintervals} applies and guarantees that the order $(p(\Mon)/E,<)$ is dense and endless for all nonalgebraic types $p\in S_1(A)$ and all $E\in \mathcal E_p\smallsetminus\{\mathbf 1_p\}$. This allows us to apply Lemma \ref{Lemma_non_R}; since its conclusion contradicts $RC(x=x)<\omega$, $E$ is relatively 0-definable. Therefore, $T$ satisfies (R). 

\smallskip 
(b) $T$ is rosy and has few countable models. \\
Ealy and Onshuus in \cite{ealy2007characterizing}
introduced the equivalence rank $eq-rk_{\Delta}(\Pi(\bar x))$, where $\Pi(\bar x)$ is a partial type and $\Delta$ a finite set of $L$-formulae in an arbitrary first-order theory. The main clause of the definition is the following: 
\begin{enumerate}[$\bullet$]
\item $eq-rk_{\Delta}(\Pi(\bar x))\geqslant \alpha+1$ if and only if there is some equivalence relation on $\Mon^{|\bar x|}$,
$E(\bar x,\bar y)$, defined by $\phi(\bar x,\bar y,\bar c)$ with $\phi(\bar x,\bar y,\bar z)\in \Delta$ and $\bar c\in\Mon$, and there are $(\bar b_j\mid j\in \omega)$ representatives of different $E$-classes, such that $eq-rk_{\Delta}(\Pi(\bar x)\land E(\bar x,\bar b_j))\geqslant\alpha$ ($j\in\omega$).
\end{enumerate}
They prove that a first-order theory is rosy if and only if all equivalence ranks $eq-rk_{\Delta}(\Pi(\bar x))$ are finite (for all $\bar x, \Pi(\bar x)$, and finite $\Delta$). 
Arguing like in the part (a), we conclude that $T$ satisfies (R).
\end{Remark}

\begin{Lemma}\label{Lema_R_on_eq}
 If a weakly quasi-o-minimal theory $T$ satisfies condition (R), then it also satisfies:
\begin{enumerate}
\item[\hspace{10pt} (R)$^\mathrm{eq}$] Every relation $E\in \mathcal E_p$ is relatively $0$-definable for all $A\subseteq \Moneq$ and all $p\in S_1(A)$.
\end{enumerate}
\end{Lemma}
\begin{proof} 
Fix a 0-definable linear order $<$ on $\Mon$, $A\subseteq\Moneq$, a real type $p\in S_1(A)$, and $E\in\mathcal E_p$; denote $\mathbf p=(p,<)$.
If $E=\mathbf 1_p$, then the conclusion is obvious, so assume $E\neq\mathbf 1_p$. 
Choose $B\subseteq \Mon$ such that $A\subseteq \dcleq(B)$. Let $p_{AB}=(\mathbf p_r)_{\strok AB}$ and let $p_B=(\mathbf p_r)_{\strok B}$; clearly, the set $p_B(\Mon)=p_{AB}(\Mon)$ is a final part of $p(\Mon)$. Note that $E$ is relatively $B$-definable on $p(\Mon)$, and for all $b\models p_B$, $E\neq\mathbf 1_p$ implies that the class $[b]_E$ is a bounded, convex subset of $p_B(\Mon)$; it follows that the relation $E_B=E_{\strok p_B(\Mon)}$ is a relatively $B$-definable equivalence relation on $p_B(\Mon)$, $E_B\neq \mathbf 1_{p_B}$, and that $[b]_E=[b]_{E_B}$ holds for all $b\models p_B$. Then condition (R) applies to the type $p_B$ and the relation $E_B$: $E_B$ is relatively $0$-definable. 
Fix $b\models p_B$.
By Lemma \ref{Lemma O EKVIVALENCIJAMA} there is a $0$-definable equivalence relation $E_0$ on $p_{\strok \emptyset}(\Mon)$ such that $[b]_{E_0}=[b]_{E_B}$. Then $[b]_E=[b]_{E_0}$, so $[b]_E$ is $b$-invariant. By Lemma \ref{Lemma O EKVIVALENCIJAMA}, $E$ is relatively $0$-definable on $p(\Mon)$. 
\end{proof}

\begin{Proposition}\label{Prop_equivalents_of_R} 
The following conditions are equivalent for any weakly quasi-o-minimal theory $T$ with simple semiintervals.
\begin{enumerate}[\hspace{10pt} 1)]
    \item $T$ satisfies condition (R);
    \item For all $A\subset \Moneq$ every real sort type $p\in S_1(A)$ is trivial;
    \item $T$ is binary.
\end{enumerate}
\end{Proposition}
\begin{proof}
1)$\Rightarrow$2) Suppose that $T$ satisfies (R). Let $\mathbf p=(p,<)$ be a weakly o-minimal pair over $A\subset \Moneq$, where $p\in S_1(A)$. We need to prove that $p$ is trivial. Suppose not. Then, by \cite[Remark 5.8]{MTwqom1}, after possibly slightly enlarging the parameter set $A$ we can find a $\mathbf p_r$-triangle (here $\mathbf p=(p,<)$): 
\setcounter{equation}{0}   
\begin{equation}  a_0,a_1,a_2\models p, \ \ a_0\triangleleft^{\mathbf p} a_1 \triangleleft^{\mathbf p}a_2 \ \ \text{ and }  \ \ a_1\dep_{Aa_0} a_2.
\end{equation}
Denote $q=\tp(a_1/Aa_0)$ and choose a formula $\phi(x,y)\in \tp(a_1,a_2/Aa_0)$ witnessing $a_1\dep_{Aa_0} a_2$:  the set $\phi(a_1,\Mon)$ has an upper bound in $q(\Mon)$. Modify $\phi$ in an obvious way so that $\phi(a_1, q(\Mon))$ is a $\mathbf q$-semiinterval. Since $T$ has simple semiintervals, there exists an equivalence relation $E\in \mathcal E_q$ such that  $\phi(a_1, q(\Mon))=\{y\in [a_1]_E\mid a_1\leqslant y\}$. Then $a_2\in[a_1]_E$. Also, since $\phi(a_1,q(\Mon))$ has an upper bound in $q(\Mon)$, so does the class $[a_2]_E$; hence, $E\neq \mathbf 1_q$. 
By Lemma \ref{Lema_R_on_eq} the theory $T$ satisfies condition (R)$^\mathrm{eq}$,  so $E$ is relatively 0-definable on $q(\Mon)$. Then, by Lemma \ref{Lemma O EKVIVALENCIJAMA}, there exists a $0$-definable equivalence relation $E_0$ on the locus of $\tp(a_2)$ such that $[a_2]_E=[a_2]_{E_0}$. It follows that the class $[a_2]_{E_0}$ is a bounded subset of $q(\Mon)$, so it is bounded in $(p(\Mon),<)$. Since $E_0$ is relatively $0$-definable, $a_1,a_2\models p$, and $a_2\in[a_1]_{E_0}$, we infer $a_1\dep_A a_2$, which contradicts $a_1 \triangleleft^{\mathbf p} a_2$ from (1). Therefore, $p$ is trivial, as needed.  

\smallskip
2)$\Rightarrow$3) Assuming 2) we prove that for all $a,b,c\in\Mon$ and all $A\subseteq \Moneq$ the following holds: 
 \setcounter{equation}{0}
 \begin{equation}\tp_{xy}(ab/A)\cup \tp_{yz}(bc/A)\cup \tp_{xz}(ac/A)\vdash \tp_{xyz}(abc/A).
 \end{equation}
Clearly, this suffices to conclude that $T$ is binary. 

To simplify the notation, absorb the parameters $A$ into the language; this does not affect the  assumptions on $T$. 
Fix $a,b,c\in \Mon$ and denote \  $p=\tp(a)$, $q=\tp(b)$, and $r=\tp(c)$. The proof is split into two cases;  the first has  two subcases, while  the second case reduces to the first one. 

\smallskip 
 Case Ia.  \ \ $\{a,b,c\}$ is a pairwise independent set  and the types $\{p,q,r\}$ are pairwise $\nwor$.  

\smallskip
 Equip each of the types $p,q$ and $r$ with an order  so that the weakly o-minimal pairs $\mathcal F=\{\mathbf p, \mathbf q,\mathbf r\}$ are directly non-orthogonal. Since elements $a$, $b$, and $c$  are pairwise independent, they are $\triangleleft^{\mathcal F}$-comparable by \cite[Theorem 2.12]{MTwqom1}. Without loss of generality, assume $c\triangleleft^{\mathcal F} b\triangleleft^{\mathcal F} a$. 
By \cite[Proposition 5.13]{MTwqom1}, every $\triangleleft^{\mathcal F}$-increasing sequence  of realizations of trivial weakly o-minimal types from $\mathcal F$ is $\mathcal F$-independent, that is, \ \begin{equation}\mbox{$c\triangleleft^{\mathcal F} b\triangleleft^{\mathcal F} x\vdash cb\triangleleft^{\mathcal F}  x$; in particular, this holds for $x=a$.}
\end{equation}
Note that the information $b\triangleleft^{\mathcal F} a$ means $\tp_x(a/b)=(\mathbf p_r)_{\strok b}$ and that $cb\triangleleft^{\mathcal F}  a$ means 
$\tp_x(a/bc)=(\mathbf p_r)_{\strok bc}$. So, (2) implies \ $\tp_x(a/b)\vdash \tp_x(a/bc)$, which clearly implies (1).  

 \smallskip 
Case Ib.  \ \ $\{a,b,c\}$ is a pairwise independent set and the types $\{p,q,r\}$ are not pairwise $\nwor$.  

\smallskip
Without loss of generality, assume $p\wor q$ and denote $p_c=\tp(a/c),\ q_c=\tp(b/c)$. Since $p$ and $q$ are trivial types, and since by \cite[Theorem 5.10(c)]{MTwqom1} the weak orthogonality of trivial types transfers to their nonforking extensions, we have $p_c\wor q_c$ and hence \ $\tp_x(a/c)\cup \tp_y(b/c)\vdash \tp_{xy}(ab/c)$. Clearly, this implies (1). 
 
\smallskip 
 Case II. \ \ $\{a,b,c\}$ is not a pairwise independent set.

Let $D= m(a,b)\cup m(b,c)\cup m(c,a)$ and, without loss of generality, assume $D= m(a,b)$. 
Then by Proposition \ref{Lemma_triple_noshiftsconsequence}(b)   the elements $\{a,b,c\}$ are pairwise independent over $D$, so by Case I we have:
\begin{equation}
\tp_{xy}(ab/D)\cup \tp_{yz}(bc/D)\cup \tp_{xz}(ac/D)\vdash \tp_{xyz}(abc/D).
 \end{equation}
Fix $a,b$ as the parameters in (3). We get \ $\tp_z(c/Db)\cup \tp_z(c/Da) \vdash \tp_z(c/Dab)$   \ which together with
 $D\subseteq \dcleq(a)$ and $D\subseteq  \dcleq(b)$ gives $\tp_z(c/b)\cup \tp_z(c/a) \vdash \tp_z(c/ab)$; (1) follows.  

 \smallskip
 3)$\Rightarrow$1) We have already noted in Remark \ref{Remark_R_for binary} that any binary $T$ satisfies (R).
\end{proof}

\begin{Theorem}\label{Theorem_VC for R}
Vaught's conjecture holds for weakly quasi-o-minimal theories that satisfy (R).
\end{Theorem}
\begin{proof}
Suppose that $T$ is weakly quasi-o-minimal, satisfies (R), and  $I(T,\aleph_0)<2^{\aleph_0}$. By Theorem \ref{Theorem few models implies simple semiintervals}, $T$ has simple semiintervals, so Proposition \ref{Prop_equivalents_of_R} applies: $T$ is binary. Then Vaught's conjecture holds for $T$ by \cite[Theorem 1]{MT}.
\end{proof}

\setcounter{equation}{0}

\begin{Corollary}\label{Cor_few_ros_ finconvquasi bin_are_binary}
Vaught's conjecture holds for any weakly quasi-o-minimal theory that is, in addition, quasi-o-minimal, rosy, or has a finite convexity rank. 
\end{Corollary}
\begin{proof} Assume that  $I(T,\aleph_0)<2^{\aleph_0}$. 
Then, by Theorem \ref{Theorem few models implies simple semiintervals}, $T$ has simple semiintervals. In addition, $T$ satisfies (R). For $T$ quasi-o-minimal this follows by Proposition \ref{Prop_quasio_satisfy_R}, while for $T$ rosy or of a finite convexity rank, this follows from Remark \ref{Rematk_finconv_rosywom}. By Theorem \ref{Theorem_VC for R} Vaught's conjecture holds for $T$.
\end{proof}

\section{Convex types, definabilty and triviality}\label{Section 5} 

The main technical result of this section is the following theorem.

\begin{Theorem}\label{Theorem_convex} Suppose that $T$ is countable and $I(T,\aleph_0)<2^{\aleph_0}$. Then every weakly o-minimal type with hereditarily simple semiintervals is convex. 
\end{Theorem}

With this theorem and Theorem \ref{Theorem convex few ctble has hssi} in hand, we will be able to quickly complete the proof of Theorem \ref{Theorem main} in Subsection \ref{Subsection convex}. Then, in Subsection \ref{Subsection def triv}, we prove Proposition \ref{Proposition mixed kind is trivial}, which relates the definabilty of a convex weakly o-minimal type to its triviality and will be essentially used in the proof of Martin's conjecture in Section \ref{Section Martin}. In particular, Proposition \ref{Proposition mixed kind is trivial} implies that every convex definable weakly o-minimal type over a model (of any first-order $T$) is trivial.

\subsection{$p$-large sets} 

\begin{Definition}
Let $\mathbf p=(p,<)$ be a weakly o-minimal pair over $A$.
\begin{enumerate}[(a)]
    \item A parametrically definable set $D\subset \Mon$ is {\it $p$-large} if there are $a_1\triangleleft^{\mathbf p}a_2$ realizing $p$ with the segment $[a_1,a_2]$ of $(p(\Mon),<)$ completely contained in $D$; otherwise $D$ is $p$-small.
    \item A formula $\phi(x,\bar b)$ is $p$-large if and only if the set $\phi(\Mon,\bar b)$ is $p$-large.
    \item A type is $p$-large if and only if it consists of $p$-large formulae.
\end{enumerate}  
\end{Definition}

\begin{Remark}\label{Remark_p_large} Let $p\in S_1(A)$ be a weakly o-minimal type.
It is obvious that a superset of a $p$-large set is $p$-large,  as well as that any nonforking extension of $p$ is a $p$-large type. 
 
(a) $p$-large sets contain Morley sequences in $p$ of an arbitrary order-type; such a sequence can always be found within the interval $[a_1,a_2]$ that witnesses the $p$-largeness.    
 
(b) For every $p$-large set $D$ and every small set $C\subseteq p(\Mon)$ there exists a $C'\subset D$ such that $C\equiv C'\,(A)$: If $C$ is a small set of realizations of $p$, then $b_1<C<b_2$ holds for some independent pair of realizations of $p$. Moving the interval $[b_1,b_2]$ by an $A$-automorphism into $D$, moves $C$ to $C'$.     

(c) Let $(\Mon,<)\models DLO$, let $p$ be the unique type over $\emptyset$, and let $a,b\in\Mon$ satisfy $a<b$. Then $q\in S_1(a,b)$ determined by $a<x<b$ is a forking, $p$-large extension of $p$. 
\end{Remark}

Large extensions of weakly o-minimal types are characterized by the following lemma.

\begin{Lemma}\label{lemma large characterization}
Let $(p,<)$ be a weakly o-minimal pair over $A$ and let $q\in S(B)$ be an extension of $p$. Denote $\mathbf q=(q,<)$. 
The following conditions are equivalent:
\begin{enumerate}[\hspace{10pt}(1)]
\item $q$ is a $p$-large;
\item $a\triangleleft^{\mathbf q}b$ implies $a\triangleleft^{\mathbf p}b$ for all $a,b\models q$;
\item $\mathcal{D}_p(a)\subseteq\mathcal{D}_q(a)$ for all $a\models q$.
\end{enumerate}
\end{Lemma}
\begin{proof} 
(1)$\Rightarrow$(2) Suppose that $q$ is $p$-large.
Let $a,b\models q$ be such that $a\triangleleft^{\mathbf q}b$; we need to show $a\triangleleft^{\mathbf p}b$.
Since $q$ is $p$-large, there is some $c\models q$ such that $b\triangleleft^{\mathbf p}c$; we have $a\triangleleft^{\mathbf q}b\triangleleft^{\mathbf p}c$. Then $a\triangleleft^{\mathbf q}b<c$ implies $a\triangleleft^{\mathbf q}c$, so $b\equiv c\ (Ba)$ and there is a $d$ such that $db\equiv bc\ (Ba)$; then $a\triangleleft^{\mathbf q}d\triangleleft^{\mathbf p}b$. In particular, $a<d\triangleleft^{\mathbf p}b$ implies $a\triangleleft^{\mathbf p}b$, as needed.

(2)$\Rightarrow$(3) Assume (2) and let $a\models q$. We need to prove $\mathcal D_p(a)\subseteq\mathcal D_q(a)$. Otherwise, there would be some $b\in\mathcal D_p(a)\smallsetminus\mathcal D_q(a)$. Since $\mathcal D_p(a)\cap q(\Mon)$ and $\mathcal D_q(a)$ are convex subsets of $q(\Mon)$ that contain $a$ and since $\mathcal D_q(a)$ is a bounded subset of $q(\Mon)$, such a $b$ could be found within $q(\Mon)$.  
Without loss we may also assume that $\mathcal D_q(a)<b$, that is, $a\triangleleft^{\mathbf q}b$ (the other case, $b<\mathcal D_q(a)$, is similar). By (2), we obtain $a\triangleleft^{\mathbf p}b$, so $b\notin\mathcal D_p(a)$. Contradiction.

(3)$\Rightarrow$(1) Suppose that $\mathcal{D}_p(a)\subseteq\mathcal{D}_q(a)$ for all $a\models q$. Let $a_1,a_2\models q$ be such that $a_1\triangleleft^{\mathbf q}a_2$; we prove $a_1\triangleleft^{\mathbf{p}}a_2$. Clearly, $a_1\triangleleft^{\mathbf q}a_2$ implies $\mathcal{D}_q(a_1)<\mathcal{D}_q(a_2)$, so (3) implies $\mathcal{D}_p(a_1)< \mathcal{D}_p(a_2)$, that is, $a_1\triangleleft^{\mathbf{p}}a_2$. Now, $a_1,a_2\models q$, $a_1\triangleleft^{\mathbf{p}}a_2$, and the convexity of $q(\Mon)$ within $p(\Mon)$ together imply that $q$ is $p$-large, proving (1). 
\end{proof}

\begin{Lemma}\label{Lemma large ext of wom has equal D's}
Let $p\in S(A)$ be a trivial, weakly o-minimal type, and let $q\in S(B)$ be a $p$-large extension of $p$. Then $q$ is trivial and $\mathcal D_q(a)=\mathcal D_p(a)$ for all $a\models q$.  
\end{Lemma}
\begin{proof}
Let $\mathbf p=(p,<)$ be a weakly o-minimal pair over $A$ and let $\mathbf q=(q,<)$. First, we prove $\mathcal D_q(a)=\mathcal D_p(a)$ for all $a\models q$. Suppose not. Then, by Lemma \ref{lemma large characterization}, we have $\mathcal D_p(a)\subset\mathcal D_q(a)$. 
Let $a_0\models q$ be arbitrary and let $b_0\models q$ be such that $b_0\in\mathcal D_q(a_0)\smallsetminus\mathcal D_p(a_0)$; reversing the order if necessary, we can assume that
$a_0\triangleleft^{\mathbf p}b_0$, but $a_0\ntriangleleft^{\mathbf q}b_0$. Choose $a_1,b_1\models q$ such that $b_0\triangleleft^{\mathbf q}a_1\triangleleft^{\mathbf q}b_1$. Note that $a_0b_0\nequiv a_1b_1\ (B)$, so there is $\theta(x,y)\in\mathcal \tp(a_0,b_0/B)$ such that $\models\lnot\theta(a_1,b_1)$. We have:
\[a_0\triangleleft^{\mathbf p}b_0\triangleleft^{\mathbf q}a_1\triangleleft^{\mathbf q}b_1, \  \ a_0\not\triangleleft^{\mathbf q} b_0  \ \ \text{ and } \ \ \models\theta(a_0,b_0)\land\lnot\theta(a_1,b_1). \]
By recursion, for $n\geqslant 1$ define \ $a_{2n},b_{2n},a_{2n+1},b_{2n+1}\models q$ such that:
\[b_{2n-1}\triangleleft^{\mathbf q}a_{2n}\ \ \ \ \ \mbox{and}\ \ \ \ \ a_{2n}b_{2n}a_{2n+1}b_{2n+1}\equiv a_0b_0a_1b_1\, (B).\]
By Lemma \ref{lemma large characterization}, \ $a_0\triangleleft^{\mathbf p}b_0\triangleleft^{\mathbf p}a_1\triangleleft^{\mathbf p}\dots \triangleleft^{\mathbf p}a_n\triangleleft^{\mathbf p}b_n\triangleleft^{\mathbf p}\dots$. This is a Morley sequence in $\mathbf p_r$ over $A$, because every $\triangleleft^{\mathbf p}$-increasing sequence of realizations of a trivial type is (\cite[Lemma 5.1]{MTwqom1}). In particular, the sequence $(a_nb_n\mid n\in\omega)$ is indiscernible over $A$, so $\models\theta(a_{2n},b_{2n})\land \lnot\theta(a_{2n+1}, b_{2n+1})$ ($n\in\omega$) contradicts the fact that $p$ is NIP.
Therefore, $\mathcal D_q(a)=\mathcal D_p(a)$ for all $a\models q$.

Now, we prove that $q$ is trivial. Let $C=(c_i)_{i\in\omega}$ be a $\triangleleft^{\mathbf q}$-increasing sequence; we need to show that it is Morley in $\mathbf q_r$ over $B$. By Lemma \ref{lemma large characterization} we know that the sequence $C$ is $\triangleleft^{\mathbf p}$-increasing, so the triviality of $p$ implies that $C$ is Morley in $\mathbf p_r$ over $A$. 
By induction on $n$ we prove that $C_n=(c_i)_{i\leqslant n}$ is Morley in $\mathbf q_r$ over $B$. For $n=1$ this is obvious, so assume that $C_n$ is Morley for some $n\geqslant 1$. Let $p_n=(\mathbf p_r)_{\strok A\bar c_{<n}}$ and $q_n=(\mathbf q_r)_{\strok B\bar c_{<n}}$. Then $c_n,c_{n+1}\models q_n$ and, since $C$ is Morley in $\mathbf p_r$ over $A$ we have $c_n \triangleleft^{\mathbf p_n} c_{n+1}$ and hence $\mathcal D_{p_n}(c_n)<\mathcal D_{p_n}(c_{n+1})$. Clearly, this implies that $q_n$ is a $p_n$-large type. Note also that $p_n$ is trivial, as a nonforking extension of a trivial type, so by the above proved we have $\mathcal D_{p_n}(c_i)=\mathcal D_{q_n}(c_{i})$ for $i=n,n+1$. Then $\mathcal D_{q_n}(c_n)<\mathcal D_{q_n}(c_{n+1})$, that is $c_n \triangleleft^{\mathbf q_n} c_{n+1}$. This proves that $C_{n+1}$ is Morley in $\mathbf q$ over $B$ and completes the proof of the lemma.
\end{proof}

\subsection{Convex weakly o-minimal types}\label{Subsection convex}

\begin{Lemma}\label{Lemma_plargeconvex_implies_pconvex}
Every weakly o-minimal type $p$ that has a complete convex $p$-large extension is convex.
\end{Lemma}
\begin{proof}\setcounter{equation}{0}
Suppose that $(p,<)$ is a weakly o-minimal pair over $\emptyset$ and that $q\in S_1(A)$ is a convex $p$-large extension of $p$. Choose $(D,<)$, a $0$-definable extension of $(p(\Mon),<)$. Since $q$ is convex, we can apply \cite[Proposition 3.4]{MTwqom1}: There exists a formula $\theta_A(x)\in q$ such that $(\theta_A(\Mon),<)$ is a definable extension of $(p(\Mon),<)$ and $q(\Mon)$ is a convex subset of $\theta_A(\Mon)$; without loss, assume $\theta_A(\Mon)\subseteq D$.   
Since $q(\Mon)$ is a convex subset of $p(\Mon)$ we have:
$$q(x_1)\cup p(x)\cup q(x_2)\cup\{x_1<x<x_2\}\vdash q(x).$$
Then  \ 
$q(x_1)\cup p(x)\cup q(x_2)\cup\{x_1<x<x_2\}\vdash \theta_A(x)$, so by compactness there is a formula  $\phi_p(x)\in p$ that implies $x\in D$ such that:
\begin{equation}
q(x_1)\cup \{\phi_p(x)\}\cup q(x_2)\cup\{x_1<x<x_2\}\vdash \theta_A(x).
\end{equation} 
We will prove that $\phi_p(x)$ witnesses the convexity of $p$: $p(\Mon)$ is convex in $(\phi_p(\Mon),<)$. Assume that $a_1,a_2\in p(\Mon)$ and $\models \phi_p(a)\land a_1<a<a_2$; we need to prove $a\models p$. Since $q$ is $p$-large, by Remark \ref{Remark_p_large}(b) there are $b_1,b_2\in q(\Mon)$ with $\tp(a_1,a_2)=\tp(b_1,b_2)$. Choose $b\in \phi_p(\Mon)$ with $\tp(b,b_1,b_2)=\tp(a,a_1,a_2)$. Then $b_1,b_2\models q$, $b\in \phi_p(\Mon)$, and $b_1<b<b_2$, so by (1) we derive $\models\theta_A(b)$. 
Now we have: $b_1,b_2\models q$, $b_1<b<b_2$, and $b\in \theta_A(x)$. 
Since $\theta_A(x)$ witnesses the convexity of $q(\Mon)$, we conclude $b\in q(\Mon)$. In particular, $b\models p$ holds. By the choice of $b$ we have $\tp(b)=\tp(a)$, so $a\models p$. Therefore, $p$ is a convex type.
\end{proof}

Since isolated weakly o-minimal types are convex, we have an immediate corollary.

\begin{Corollary}\label{COr_isolated_plarge_is_convex}
If a weakly o-minimal type $p$ has an isolated, $p$-large extension, then $p$ is convex.   
\end{Corollary}
 
In the proof of Lemma \ref{Lemma_alatomic_plarge}, we will need the following fact.

\begin{Fact}[{\cite[Proposition 7.8(a)]{MTwom}}]\label{Fact_Morley_replace}
Let $\mathbf p=(p,<)$ be a weakly o-minimal pair over $A$. Suppose that $I=(a_j\mid j\in J)$ is a (possibly finite)  Morley sequence in $\mathbf p_r$ over $A$, $a\models p$, and $a\in\conv(I)$ (the $<$-convex hull of $I$).  Then we can remove at most one element from $I$ and insert $a$ in its place so that the sequence is still Morley over $A$.     
\end{Fact}

Following Newelski \cite{Newelski}, we say that a tuple $b$ is {\em almost atomic} over $A$ if there is a finite $A_0\subseteq A$ such that $\tp(b/A_1)$ is isolated for all finite $A_1$ with $A_0\subseteq A_1\subseteq A$;
a set $B$ is {\em almost atomic} over $A\subseteq B$ if every finite tuple $b\in B$ is such. 
If the theory is small, then there exist almost atomic models over any countable set. Moreover, if $T$ is small, $A$ is countable, and $\Pi(x)$ is a partial type over $A$ that is locally omissible (no consistent formula $\phi(x)$ implies $\Pi(x)$), then there is a countable $M\models T$ which is almost atomic over $A$ and omits $\Pi(x)$; the same conclusion holds for any countable family of locally omissible partial types.

\begin{Lemma}[$T$ small]\label{Lemma_alatomic_plarge} 
Let $\mathbf p=(p,<)$ be a weakly o-minimal pair over $\emptyset$ with hereditarily simple semiintervals. Suppose that $A=(a_i\mid i\in I)$ is a countable Morley sequence in $\mathbf p_r$, $M$ is countable and almost atomic over $A$, $b\in p(M)$  and $b\in\conv(A)$. Then $b\in \bigcup_{i\in I}\mathcal D_p(a_i)$ or the type $\tp(b/A)$ is $p$-large. 
\end{Lemma}
\begin{proof} Assume  $b\notin \bigcup_{i\in I}\mathcal D_p(a_i)$ and we will prove that the type $\tp(b/A)$ is $p$-large, that is, 
 every formula from $\tp(b/A)$ is $p$-large. In fact,  
it suffices to prove that every formula $\phi(x)$ that isolates $\tp(b/A_0)$ for some finite $A_0\subseteq A$ is $p$-large; then, since the type $\tp(b/A)$ is almost atomic over $A$, every formula of that type would be implied by such a $p$-large formula $\phi(x)$, and so would also be $p$-large. 

Suppose that $\phi(x)$ is over the finite parameter set $A_0\subset A$  with $\phi(x)\vdash \tp(b/A_0)$. In addition, we can assume that $A_0$ is large enough such that $b\in \conv (A_0)$. Then Fact \ref{Fact_Morley_replace} applies, so there exists $a_k\in A_0$ with $b\equiv a_k\,(A_0\smallsetminus \{a_k\})$;  denote $A_k=A_0\smallsetminus \{a_k\}$ and let $q=\tp(b/A_k)$. The assumption $b\notin \bigcup_{i\in I}\mathcal D_p(a_i)$ implies $b\notin \mathcal D_p(a_k)$, so $\mathcal D_p(a_k)<\mathcal D_p(b)$ or $\mathcal D_p(b)<\mathcal D_p (a_k )$ holds. From now on, we will assume $\mathcal D_p(a_k)<\mathcal D_p(b)$; the other case is handled similarly.

Now we want to display the parameter $a_k$  in the formula $\phi(x)$: Let $\phi(x,y)\in \tp(b,a_k/A_k)$. Define $$S_{a_k}:=\{x\in q(\Mon)\mid a_k\leqslant x\land (\exists t)(x\leqslant t\land \phi(t,a_k))\}.$$
Note that $S_{a_k}$ is a $(q,<)$-semiinterval containing $b$, so since $p$ has hereditarily simple semiintervals, it is $\mathcal E_q$-determined, say, by $E\in \mathcal E_q$; note that $E(a_k,b)$ is valid. 
Choose $c\in \Mon$ with $a_kb\equiv bc\,(A_k)$. Then $a_k<b<c$, $b\in [a_k]_E$ and $c\in[b]_E$ hold, so $c\in[a_k]_E$. Since $E$ determines $S_{a_k}$, $a_k<c$, and $c\in[a_k]_E$ realizes $q$, we conclude $c\in S_{a_k}$. 
By the definition of $S_{a_k}$, there exists $b'\in\phi(\Mon,a_k)$ with $c\leqslant b'$. Since the formula $\phi(x,a_k)$ isolates $\tp(b/A_0)$, we have $b\equiv b'(A_0)$. Now, $a_kb\equiv bc\,(A_k)$ and $\mathcal D_p(a_k)<b$  imply $\mathcal D_p(b)<c$, which together with $c\leqslant b'$ proves $\mathcal D_p(b)<b'$. Then $b\triangleleft^\pp b'$ and $b\equiv b'(A_0)$ together witness that the locus of $\tp(b/A_0)$ (which is isolated) is a $p$-large set. Therefore, assuming $b\notin \bigcup_{i\in I}\mathcal D_p(a_i)$ we proved that $\tp(b/A)$ is $p$-large; this completes the proof of the lemma. 
\end{proof}

\begin{proof}[Proof of Theorem \ref{Theorem_convex}] Assume that $T$ is countable and $I(T,\aleph_0)<2^{\aleph_0}$; in particular, $T$ is small. Let $\mathbf p=(p,<)$ be a weakly o-minimal pair over $\emptyset$ with hereditarily simple semiintervals.  Assuming that $p$ is not convex, we will reach a contradiction by constructing continuum many nonisomorphic countable models. 
Fix a countable, endless linear order $I$ and let $A=(a_i\mid i\in I)$ be a Morley sequence in $\mathbf p_r$. Since $I$ is endless, it is easy to see that $(\mathbf p_r)_{\strok A}=\bigcup_{i\in I}\tp(a_i/A_{<i})$ holds and implies that $(\mathbf p_r)_{\strok A}$ is a nonisolated type. Since $T$ is small and $A$ is countable, there exists a countable model $M_{\mathbb I}$ that is almost atomic over $A$ and omits $(\mathbf p_r)_{\strok A}$. We {\em claim}:

(1) \ $\tp(b/A)$ is $p$-small for all $b\in p(M_{\mathbb I})$.

Since $M_{\mathbb I}$ is almost atomic over $A$, there is a finite $A_0\subseteq A$ such that $\tp(b/A_0)$ is isolated, by $\theta(x)$ say. 
Since $p$ is not convex, by Corollary  \ref{COr_isolated_plarge_is_convex}, $p$ does not have an isolated $p$-large extension. Therefore, 
the formula $\theta(x)$ is $p$-small. Then any type containing $\theta(x)$ is $p$-small; in particular, the type $\tp(b/A)$ is $p$-small. We {\em claim}:

(2)  \   $\bigcup_{i\in I}(\mathcal D_p(a_i)\cap M_{\mathbb I})$ \ is a final part of  $(p(M_{\mathbb I}),<)$. 

Suppose   that  $b\in p(M_{\mathbb I})$ and  $b\nless \bigcup_{i\in I}\mathcal D_p(a_i)$; we need to prove $b\in \bigcup_{i\in I}\mathcal D_p(a_i)$.  First note that $b\in \conv(A)$ holds: otherwise, since the order $\mathbb I$ is endless, $\bigcup_{i\in I}\mathcal D_p(a_i)<b$  would imply   $b\models(\mathbf p_r)_{\strok A}$  and contradict the choice of $M_{\mathbb I}$. Therefore, $b\in\conv(A)$, so Lemma \ref{Lemma_alatomic_plarge} applies: $b\in \bigcup_{i\in I}\mathcal D_p(a_i)$ or the type $\tp(b/A)$ is $p$-large. The second option is ruled out by (1), so $b\in \bigcup_{i\in I}\mathcal D_p(a_i)$, as needed.

\smallskip
Now, consider $Inv_p(M_{\mathbb I})$, which is the order type of $(\{\mathcal D_p(a)\mid a\in p(M_{\mathbb I})\},<)$. By claim (2), $Inv_p(M_{\mathbb I})$ has a final part of the order-type of $\mathbb I$. 
Consider the class of countable linear orders and the relation: $\mathbb I\sim\mathbb J$ if and only if $\mathbb I$ and $\mathbb J$ have isomorphic final parts; $\sim$ is clearly an equivalence relation. It is not hard to see that there are continuum many $\sim$-classes (e.g.\ see \cite[Section 4.3]{Rast}).  Hence, when the orders $\mathbb I$ and $\mathbb J$ are from distinct $\sim$-classes,  the models $M_{\mathbb I}$ and $M_{\mathbb J}$ are non-isomorphic;  
 $I(T,\aleph_0)=2^{\aleph_0}$ follows. Contradiction. This shows that $p$ is convex, completing the proof of Theorem \ref{Theorem_convex}. 
\end{proof}

Recall \cite[Definition 3.7]{MTwqom1}: A complete first order theory $T$ is {\em almost weakly o-minimal} if every type $p\in S_1(T)$ is weakly o-minimal and convex.
As a corollary of Theorem \ref{Theorem_convex}, we obtain the following theorem, which is a restated Theorem \ref{Theorem main}.

\begin{Theorem}
Every weakly quasi-o-minimal theory with few countable models is almost weakly o-minimal and has simple semiintervals.   
\end{Theorem}
\begin{proof} 
Let $T$ be a countable, weakly quasi-o-minimal theory with few countable models. Then, by Corollary \ref{Theorem few models implies simple semiintervals}, $T$ has simple semiintervals, so all weakly o-minimal types $p\in S_1(T)$ have hereditarily simple semiintervals and hence are convex by Theorem \ref{Theorem_convex}; $T$ is almost weakly o-minimal.
\end{proof}

\subsection{Definability vs triviality of convex types}\label{Subsection def triv}

The notions of left-/right-definability of convex so-pairs were introduced in \cite{MT}; here, we need them only in the context of weakly o-minimal pairs.   

\begin{Definition}
A convex weakly o-minimal pair $\mathbf p=(p,<)$ is left-definable (right-definable) if the type $\mathbf p_l$ ($\mathbf p_r$) is definable.  
\end{Definition}

The definability of convex weakly o-minimal pairs is related to the definability of final parts of $(p(\Mon),<)$ and to nonisolation properties of nonforking extensions of $p$ as follows:

\begin{Remark}\label{remark left right definable} 
Suppose that $p\in S(A)$ is a weakly o-minimal type such that $p(\Mon)$ is a convex subset of $(D,<)$, an $A$-definable linear order. 
\begin{enumerate}[(a)]
    \item By Lemma 7.2 of \cite{MT}, the pair $\mathbf p=(p,<)$ is left-definable if and only if some initial part of $p(\Mon)$ is a definable set if and only if the set $p^-(\Mon)=\{x\in D\mid x<p(\Mon)\}$ is $A$-definable. Similarly for right-definable. 
    \item For all $B\supseteq A$, $\mathbf p=(p,<)$ is left-definable if and only if $((\mathbf p_l)_{\strok B},<)$ is left-definable, and $\mathbf p=(p,<)$ is right-definable if and only if $((\mathbf p_r)_{\strok B},<)$ is right-definable. 
    \item If $\mathbf p$ is not left-definable, then the type $\mathbf p_l$ is finitely satisfied in $p^-(\Mon)$. Indeed, if $\phi(x)\in\mathbf p_l$, then $\phi(\Mon)$ contains an initial part of $p(\Mon)$, so since no initial part of $p(\Mon)$ is definable, $\phi(x)$ is realized in $p^-(\Mon)$. Note that in this case, the type $p$, as well as the type $(\mathbf p_l)_{\strok B}$ (where $B\supseteq A$), is nonisolated. 
    \item If $p$ is simple and $\mathbf p$ is of kind (x,y) ($\mathrm{x,y}\in\{\mathrm{d,nd}\}$), as defined below, then $((\mathbf p_r)_{\strok Aa},<)$ is of kind (d,y) and $((\mathbf p_l)_{\strok B},<)$ is of kind (x,d) for all $a\models p$.
\end{enumerate} 
\end{Remark}

 There are four kinds of convex weakly o-minimal {\it pairs} depending on whether their ends are definable or not. We will label them by (d,d), (d,nd), (nd,d), and (nd,nd); for example, (d,nd) labels the pairs that are left-definable but not right-definable. There are three kinds of convex weakly o-minimal {\it types}.

 \begin{itemize}
     \item $p$ is of kind (d,d) if both global nonforking extensions of $p$ are definable; that is, if every (equivalently, some) weakly o-minimal pair $(p,<_p)$ over $A$ is of kind (d,d). Analogously for: $p$ is of kind (nd,nd). 
     \item $p$ is of {\it mixed kind} if exactly one of its global nonforking extensions is definable; in that case, every (equivalently, some) weakly o-minimal pair $(p,<_p)$ over $A$ is of kind (d,nd) or (nd,d). 
\end{itemize}

\begin{Remark}\label{Remark dd}
(a) A weakly o-minimal convex type $p$ is of kind (d,d) if and only if it is isolated:
By Remark \ref{remark left right definable}, $p$ is of kind (d,d) if and only if the locus $p(\Mon)$ has some initial part and some final part definable, which is easily seen to be equivalent to the definability of $p(\Mon)$. Finally, $p(\Mon)$ is definable if and only if $p$ is isolated. 

(b) Every definable, nonalgebraic, weakly o-minimal type over a model is of mixed kind. 

(c) In the literature, complete 1-types in (weakly) o-minimal theories of kind (nd,nd) are also called cuts or irrational types, while those of mixed kind are called noncuts, rational cuts or quasi-rational types.
\end{Remark}

\begin{Fact}\label{Fact direct no preserves the kind}\cite[Lemma 7.17]{MT}
Directly non-orthogonal convex weakly o-minimal pairs are of the same kind.
\end{Fact}

A consequence of this fact is that all convex weakly o-minimal types from the same $\nwor$-class are of the same kind. In different terminology, this was proved for complete 1-types in o-minimal theories by Marker in \cite{Marker1986omitting}, and in weakly o-minimal theories by Baizhanov in \cite{Baizhanov1999}.

In \cite[Proposition 5.18]{MTwqom1} we proved that all definable complete 1-types over a dense o-minimal structure are trivial. Now, we generalize this result.

\begin{Proposition}\label{Proposition mixed kind is trivial}
Every convex weakly o-minimal type of mixed kind is trivial. 
\end{Proposition} 
\begin{proof}
Toward a contradiction, let $p\in S(A)$ be a convex weakly o-minimal type of mixed kind that is not trivial. Without loss, assume that $\mathbf p=(p,<)$ is of kind (nd,d) and choose $(D,<)$, an $A$-definable order, such that $p(\Mon)$ is a final part of $D$. Note that the power $(\mathbf p_r)^n$ is definable over $A$ for all $n\in \mathbb N$.  
Since $p$ is nontrivial, for some $n\geqslant 3$ the partial type $\Pi_n=x_1\triangleleft^{\mathbf p}x_2\triangleleft^{\mathbf p}\ldots\triangleleft^{\mathbf p}x_n$ has at least two completions in $S_{x_1,\dots,x_n}(A)$. Moreover, assume that $A$ and $p$ have been chosen so that $n\geqslant 3$ is as minimal as possible.

First, we {\it claim} that $n\geqslant 4$, that is, there are no $p$-triangles. To prove it, let $(a_1,a_2,a_3)$ be a Morley sequence in $\mathbf p_r$ over $A$. 
By \cite[Lemma 5.9]{MTwqom1}, $p$-triangles do not exist if and only if $(a_3,a_2,a_1)$ is a Morley sequence in $\mathbf p_l$ over $A$.  
If $(a_3,a_2,a_1)$ were not Morley in $\mathbf p_l$, then we would have $a_1\not\models (\mathbf p_l)_{\strok Aa_2a_3}$, so there would be a formula $\phi(x,x_2,x_3)\in \tp(a_1,a_2,a_3/A)$ such that $\phi(\Mon,a_2,a_3)$ is a bounded subset of $p(\Mon)$; 
then the formula $(d_{\mathbf p_r^2} x_2,x_3)\,\phi(x,x_2,x_3)\in p(x)$ would isolate $p$, 
contradicting the fact that $\mathbf p_l$ is nondefinable. This proves the claim. The rest of the proof is divided into two cases.

Case I.  \ \ $p$ is simple. 

In this case, $x\in\mathcal D_p(y)$ is a relatively definable relation on $p(\Mon)^2$, so
$x_1\triangleleft^{\mathbf p}x_2\triangleleft^{\mathbf p}\ldots\triangleleft^{\mathbf p}x_m$ is a relatively definable relation on
$p(\Mon)^m$ for all $m\leqslant n$. Choose $\theta_m\in \Pi_m$ such that the type $\Pi_m(x_1,\dots,x_m)$ is equivalent to $p(x_1)\cup\{\theta_m(x_1,\dots,x_m)\}$. Let $(a_1,\dots,a_n)$ be a realization of $p(x_1)\cup\{\theta_n(x_1,\dots,x_n)\}$ such that $(a_n,\dots,a_1)$ is not a Morley sequence in $\mathbf p_l$ over $A$. By the minimality of $n$, both $(a_1,a_3,\dots,a_n)$ and $(a_2,a_3,\dots,a_n)$ realize $\Pi_{n-1}$, but $a_1$ does not realize $(\mathbf p_l)_{\strok Aa_2\dots a_n}$. 
The latter implies that there is a formula $\phi(x_1,\dots,x_n)\in \tp(a_1,\dots,a_n/A)$ which implies $\theta_{n-1}(x_1,x_3,\dots,x_n)\land \theta_{n-1}(x_2,x_3,\dots,x_n)$ and is such that $\phi(\Mon,a_2,\dots,a_n)$ is a bounded subset of $p(\Mon)$.  
Then the formula $(d_{\mathbf p_r} x_2)(\exists x_3,\dots,x_n)\,\phi(x,x_2,x_3,\dots,x_n)\in p(x)$ isolates $p(x)$. A contradiction.

Case II. \ \  p is not simple.

Let $(a_1,\dots,a_n)$ be a realization of $\Pi_n$ that is not a Morley sequence in $\mathbf p_r$ over $A$; by the minimality of $n$, every subsequence of length $n-1$ is Morley. Denote by $q(x)$ the type $a_1\triangleleft^{\mathbf p} x$ ($q=(\mathbf p_r)_{\strok Aa_1}$) and let $\mathbf q=(q,<)$. We will show that $q$ is of mixed kind and that $(a_2,\dots,a_n)$ witnesses the non-triviaity of $q$.
First, we see that $q$ is finitely satisfied in $\mathcal D_p(a_1)$: If $\psi(x,a_1)\in q$ were not satisfied in $\mathcal D_p(a_1)$, then $\psi(x,y)\land y<x$ would relatively define $y\triangleleft^{\mathbf p}x$ on $p(\Mon)^2$, contradicting the non-simplicity of $p$. Hence, $q$ is finitely satisfied in $\mathcal D_p(a_1)$, which is a final part of $(q^-(\Mon),<)$, so $\mathbf q_l$ is not definable. As $\mathbf p_r$ is definable and $\mathbf q_r=\mathbf p_r$, $\mathbf q$ is of kind (nd,d); $q$ is of mixed kind. Furthermore, notice that the assumed minimality of $n\geqslant 4$ implies that every 3-element subsequence of $(a_1,\dots,a_n)$ is Morley in $\mathbf p_r$ over $A$; in particular, $a_2\triangleleft^{\mathbf q}a_3\triangleleft^{\mathbf q}\dots\triangleleft^{\mathbf q}a_n$. Since $(a_2,\dots,a_n)$ is not a Morley sequence in $\mathbf q_r(=\mathbf p_r)$ over $Aa_1$, this contradicts the minimality property of $n$.   

In both cases, we reached a contradiction, so the proof of the proposition is complete.  
\end{proof}

\section{Martin's conjecture}\label{Section Martin}

Martin's conjecture is a strengthening of Vaught's conjecture. Assume that $T$ is small and let $L_1(T)$ be the smallest countable fragment of $L_{\omega_1,\omega}$ that contains all first-order formulae and formulae $\bigwedge_{\phi(x)\in p}\phi(x)$ for all $p\in S(T)$. For each countable $M\models T$, let $T_1(M)$ denote the $L_1(T)$-theory of $M$.

\smallskip
 {\bf Martin's Conjecture.} If $I(T,\aleph_0)<2^{\aleph_0}$ then $T_1(M)$ is $\aleph_0$-categorical for all countable $M\models T$.

\smallskip
Note that the quantifier rank of any $L_1(T)$-formula is below $\omega+\omega$, so  Martin's conjecture predicts that, assuming $I(T,\aleph_0)<2^{\aleph_0}$, the Scott ranks of countable models of $T$ are at most $\omega+\omega$; this is known to imply Vaught's conjecture. 
 
\begin{Definition}\label{Definition AAC} A complete first-order theory $T$ is almost $\aleph_0$-categorical if for all $n$ and all types $p_1(x_1),\ldots, p_n(x_n)\in S_1(T)$ the type $\bigcup_{1\leqslant i\leqslant n}p_i(x_i)$ has finitely many extensions in $S_n(T)$. 
\end{Definition}
Almost $\aleph_0$-categorical theories were introduced by Ikeda, Pillay, and Tsuboi in \cite{Ikeda}. The following holds in any almost $\aleph_0$-categorical theory:
\begin{itemize}
    \item Let $A\subseteq B$ be finite. Then every complete type over $A$ has finitely many extensions over $B$.
    \item Every complete type over a finite domain that extends an isolated type is also isolated.
    \item The almost $\aleph_0$-categoricity is preserved by naming finitely many parameters. 
\end{itemize}
These facts will be used silently in the proof of Theorem \ref{Theorem_AAC}, where we also need the following folklore fact and its consequence. 

\begin{Fact}[Any countable $T$]\label{Fact no wor sequence in few models}
If there exists a sequence of finite tuples $(a_n)_{n<\omega}$ such that:
\begin{itemize}
    \item[--] $\tp(a_n)$ is non-isolated for all $n<\omega$, and
    \item[--] $\tp(a_n)\wor\tp(\bar a_{<n})$ for all $n<\omega$,
\end{itemize}
then $I(\aleph_0,T)=2^{\aleph_0}$.
\end{Fact}
\begin{proof}
Suppose that a sequence $(a_n)_{n<\omega}$ satisfies the two properties. Recall that if $a,b,c$ are tuples such that $\tp(a)\wor\tp(b)$ and $\tp(ab)\wor\tp(c)$, then $\tp(a)\wor\tp(bc)$ (for if $a\equiv a'$, then $ab\equiv a'b$ as $\tp(a)\wor\tp(b)$, and let $c'$ be such that $a'bc\equiv abc'$; since $abc'\equiv abc$ as $\tp(ab)\wor\tp(c)$, we get $a'bc\equiv abc$, and hence $\tp(a)\wor\tp(bc)$). Now for all $n$, by induction on $m$, we see that $\tp(a_n)\wor\tp(\bar a_{<n}a_{n+1}\dots a_{n+m})$ by the previous fact and the second property of the sequence, whence $\tp(a_n)\wor\tp(\bar a_{\neq n})$.

Let $S\subseteq\omega$ be arbitrary and $\bar a_S=(a_n)_{n\in S}$. We will find a countable model $\mathcal M_S\supseteq\bar a_S$ omitting $\tp(a_n)$ if and only if $n\notin S$; this clearly implies $I(T,\aleph_0)=2^{\aleph_0}$. By Omitting Types Theorem, it is enough to show that for $n\notin S$ there is no consistent formula $\theta(x,\bar a_S)$ implying $\tp(a_n)$. Toward a contradiction, suppose that $\theta(x,\bar a_S)$ is a consistent formula such that $\theta(x,\bar a_S)\vdash\tp(a_n)$. Since $\tp(a_n)\wor\tp(\bar a_S)$ as $\tp(a_n)\wor\tp(\bar a_{\neq n})$, we see that the locus of $\tp(a_n)$ is definable by $\theta(x,\bar a_S)$, and since it is $0$-invariant, we conclude that it is $0$-definable; this contradicts the fact that $\tp(a_n)$ is non-isolated.
\end{proof}      

\begin{Corollary}\label{Corollary no strongly orthogonal sequence} Suppose that $T$ is countable with few countable models.
    \begin{enumerate}[(a)]
        \item There is a finite $C\subseteq\Mon$ such that for each nonisolated type $p\in S_1(T)$, $\tp(C)\nwor p$.
        \item For all $M\models T$ there is a finite $C\subseteq M$ such that for all non-isolated types $p\in S_1(T)$ that are realized in $M$, $\tp(C)\nwor p$ holds.
        \item There do not exist a finite set $A$ and a family $(p_n\mid n\in\omega)\subset S(A)$ such that each $p_n$ is nonisolated, weakly o-minimal, and trivial, such that $p_i\wor p_j$ holds for all $i\neq j$. 
    \end{enumerate}
\end{Corollary}
\begin{proof}
(a) By recursion, build a sequence of singletons $a_0,a_1,\dots$ as follows. Take any $a_0$ such that $\tp(a_0)$ is nonisolated. Suppose that $a_i$, $i<k$, are chosen. If possible, pick any $a_k$ such that $\tp(a_k)$ is nonisolated and $\tp(a_k)\wor\tp(\bar a_{<k})$. By Fact \ref{Fact no wor sequence in few models} this process ends in finitely many steps, say $n$. Then $C=\{a_i\mid i<n\}$ is as desired. The proof of (b) is similar. 

(c) First, we {\em claim} that for all $n$, all finite $A$, and all sequences of pairwise weakly orthogonal trivial weakly o-minimal types $(q_i\mid i<n)\subseteq S_n(A)$, the type $\bigcup_{i<n}q_i(x_i)$ has a unique completion in $S_{x_0\dots x_{n-1}}(A)$. 
By induction on $n$, suppose that this holds for some $n$. Let $a_0\models q_0$ and let $q_k'\in S(Aa_0)$ ($1\leqslant k\leqslant n$) denote the unique extension of $q_k$. Then, since by \cite[Theorem 5.10]{MTwqom1}, both the triviality of weakly minimal types and the orthogonality of trivial weakly o-minimal types are preserved in nonforking extensions, 
the sequence $(q_k'\mid 1\leqslant k\leqslant n)$ 
satisfies the induction hypothesis: $\bigcup_{1\leqslant i\leqslant n}q_i'(x_i)\vdash \tp(x_1\ldots x_n/Aa_0)$. 
Since $a_0\models q_0$ was arbitrary, $\bigcup_{i\leqslant n}q_i(x_i)\vdash \tp(x_0\ldots x_n/A)$ follows, proving the claim. 

Now, let $(p_i\mid i\in\omega)$ be nonisolated, weakly o-minimal, and trivial. For each $i\in\omega$, choose $a_i\models p_i$. By the claim, the sequence $(a_n)_{n<\omega}$ satisfies the assumptions of Fact \ref{Fact no wor sequence in few models}, so $I(T,\aleph_0)=2^{\aleph_0}$. Contradiction.  \end{proof}

 \subsection{Proof of Theorem \ref{Theorem_AAC}}

Throughout this subsection, we assume: 
\begin{itemize}
    \item $T$ is weakly quasi-o-minimal with respect to $<$;
    \item $I(T,\aleph_0)<2^{\aleph_0}$; in particular, $T$ is small;
    \item $T$ is almost $\aleph_0$-categorical.
\end{itemize}
As a consequence of Theorem \ref{Theorem main} we also have:
\begin{itemize}
    \item $T$ has simple semiintervals;
    \item $T$ is almost weakly 0-minimal: all types $p\in S_1(A)$ are convex for all $A$.
\end{itemize}

The strategy for proving Theorem \ref{Theorem_AAC} is the following. Let $M\models T$ be countable. We will find a finite set $C_M\subset M$, called a strong base for $M$, and $\mathcal U_M$, a set of global $0$-invariant 1-types, such that the following two conditions hold for all finite $D$ with $C_M\subseteq D\subset M$:
\begin{enumerate}[label=--]
\item $\mathfrak p_{\strok D}$ is trivial and realized in $M$ for all $\mathfrak p\in\mathcal U_M$; and 
\item if $a\in M$ and $\tp(a/D)$ is nonisolated, then $a\models \mathfrak p_{\strok D}$ for some $\mathfrak p\in\mathcal U_M$.
\end{enumerate}
We will prove that $\tp(C_M)$ and $\mathcal U_M$ can be reconstructed from the theory $T_1(M)$. Then for any countable $N$ with $M\equiv_{L_1(T)}N$, by back-and-forth, we easily construct an isomorphism between $M$ and $N$.

\begin{Lemma}\label{Lemma ALMOST isolation of extensionss}
Let $\mathbf p=(p,<)$ be a weakly o-minimal pair over $\emptyset$. Suppose that $C$ is finite and $\tp(C)\nwor p$. Let $S_p(C)=\{q\in S_1(C)\mid p\subseteq q\}$. 
    \begin{enumerate}[(a)]
       \item If $q\in S_p(C)$ is a forking extension of $p$, then $q$ is isolated.
       \item If $q=(\mathbf p_{l})_{\strok C}$, then the pair $\mathbf q=(q,<)$ is right-definable; $q$ is isolated if and only if $\mathbf p_l$ is definable. The dual assertion holds for $q=(\mathbf p_{r})_{\strok C}$.
    \end{enumerate}
\end{Lemma}
\begin{proof}Since $p$ is convex, without loss of generality, we may assume that $p(\Mon)$ is a convex subset of $\Mon$. Then the loci of the types from $S_p(C)$ are convex subsets of $\Mon$ that are linearly ordered by $<$. By almost $\aleph_0$-categoricity, the space $S_p(C)$ is finite, say $S_p=\{q_0,\dots,q_n\}$ where $q_0(\Mon)<q_2(\Mon)<\dots<q_n(\Mon)$. Note that $q_0=(\mathbf p_l)_{\strok C}$ and $q_n=(\mathbf p_r)_{\strok C}$; also note that $p\nwor \tp(C)$ implies $q_0\neq q_n$. We see that each $q_k(\Mon)$ is relatively definable within $p(\Mon)$: Fix $k\leqslant n$ and for each $i\neq k$ choose $\phi_i(x)\in q_k(x)\smallsetminus q_i(x)$. Then $q_k(\Mon)$ is relatively defined by $\sigma_k(x)=\bigwedge_{i\leqslant n,i\neq k}\phi_i(x)$. 

(a) Suppose that $q_k$ is a forking extension of $p$, that is, $0<k<n$. Choose $a_1\models q_0$ and $a_2\models q_n$. Then it is easy to see that $q_k(\Mon)$ is defined by $a_1<x<a_2\land \sigma_k(x)$. Since the set $q_k(\Mon)$ is definable and it is $C$-invariant, it is $C$-definable, so $q_k$ is isolated.   

(b) Consider $q=q_0=(\mathbf p_{l})_{\strok C}$ and let $\sigma(x)=\sigma_0(x)\land x<b$, where $b\models (\mathbf p_{r})_{\strok C}$. Then $q(\Mon)$ is a final part of $(\sigma(\Mon),<)$ so, by Remark \ref{remark left right definable}, $\mathbf q$ is right-definable, proving the first claim. To prove the second, note that $\mathbf q$ is of kind (nd,d) or (d,d) and that, by Remark \ref{Remark dd}, the latter holds if and only if $q$ is isolated. 
Since $q=(\mathbf p_{l})_{\strok C}$ implies $\mathbf p_l=\mathbf q_l$, $\mathbf p$ is left-definable if and only if $\mathbf q$ is left-definable. Therefore, $\mathbf q$ is of kind (d,d) if and only if 
$\mathbf p_l$ is definable. This completes the proof.   
\end{proof}

For the rest of this subsection, fix the following notation: 
\begin{itemize}
\item $\mathcal P$ is the family of all global, $0$-invariant, nondefinable 1-types.
    \item For each $\mathfrak p\in\mathcal P$ let $p=\mathfrak p_{\strok\emptyset}$, and let 
    $\lessdot$ be $<$ or $>$ such that $(p,\lessdot)_r=\mathfrak p$;
    \item For $\mathfrak p=(p,\lessdot)_r\in\mathcal P$ and finite $C$, denote  $\mathfrak p_C=(\mathfrak p_{\strok C},\lessdot)$.
    \item For a model $M\models T$ define:
        \begin{itemize}
            \item $\mathcal R_M=\{\mathfrak p\in\mathcal P\mid \mathfrak p_{\strok \emptyset}\text{ is realized in }M\}$;
            \item $\mathcal U_M=\{\mathfrak p\in\mathcal P\mid  \mathfrak p_{\strok C}\text{ is realized in }M \text{ for all finite } C\subseteq M\}$;
            \item $\mathcal B_M=\mathcal R_M\smallsetminus\mathcal U_M$.
        \end{itemize}
\end{itemize}

\begin{Definition}
    Let $M\models T$. A finite subset $C\subseteq M$ is called a {\em base of $M$} if $\tp(C)\nwor\mathfrak p_{\strok\emptyset}$ for each $\mathfrak p\in\mathcal R_M$; $C$ is a {\em strong base} if for every $a\in M$ and every finite $D\supseteq C$, if $\tp(a/D)$ is nonisolated then $\tp(a/D)=\mathfrak p_{\strok D}$ for some $\mathfrak p\in\mathcal U_M$.
\end{Definition}

\begin{Remark}\label{Remark ALMOST bases}
    \begin{enumerate}[(a)]
    \item For all $p\in S_1(T)$ and $\lessdot\in\{<,>\}$:  $\mathfrak p=(p,\lessdot)_r\in \mathcal P$ if and only if the pair $(p,\lessdot)$ is right-nondefinable.
    \item If $p\in S_1(T)$ is of kind (nd,nd) and $\mathbf p=(p,<)$, then $\mathbf p_r, \mathbf p_l\in \mathcal P$. 
    \item  $\mathfrak p_{\strok \emptyset}$ is not isolated for all $\mathfrak p\in \mathcal P$ since, by definition, $\mathfrak p$ is not definable.
    \item If $p\in S_1(T)$ is nonisolated, then $p=\mathfrak p_{\strok \emptyset}$ for some $\mathfrak p\in \mathcal P$.
    \item Every $M\models T$ has a base; this follows from Corollary \ref{Corollary no strongly orthogonal sequence}(b). 
    \item Clearly, if $C$ is a (strong) base of $M$ and $D\supseteq C$ is finite, then $D$ is also a (strong) base of $M$.
    \end{enumerate}
\end{Remark}

\begin{Lemma}\label{Lemma ALMOST sva priprema}
Suppose that $C$ is a base of a model $M\models T$.
    \begin{enumerate}[(a)]
        \item For all $a\in M$, if $\tp(a/C)$ is nonisolated, then there is $\mathfrak p\in\mathcal R_M$ such that $\tp(a/C)=\mathfrak p_{\strok C}$.
        \item For all $\mathfrak p\in\mathcal R_M$, the pair $\mathfrak p_C$ is of kind (d,nd). In particular, the type $\mathfrak p_{\strok C}$ is trivial and nonisolated.
        \item For all $\mathfrak p,\mathfrak q\in\mathcal R_M$ we have the following two cases:
\begin{enumerate}[(I)]
        \item $\mathfrak p\nwor\mathfrak q$. In this case, $\mathfrak p_{\strok C}\nwor\mathfrak q_{\strok C}$ and $\delta_C(\mathfrak p_C,\mathfrak q_C)$.
        
        \item $\mathfrak p\wor\mathfrak q$. In this case, $\mathfrak p_{\strok C}\wor\mathfrak q_{\strok C}$. 
\end{enumerate} 
        \item $\mathcal R_M$ has finitely many $\nwor$-classes. 
        \item If $\mathfrak p_1,\dots,\mathfrak p_n$ are representatives of all $\nwor$-classes in $\mathcal B_M$, then $C$ is a strong base if and only if each $(\mathfrak p_i)_{\strok C}$ is omitted in $M$.
        \item Strong bases of $M$ exist.
        \item $\mathcal P$ has finitely many $\nwor$-classes. 
    \end{enumerate}
\end{Lemma} 
\begin{proof}
(a) Let $a\in M$ be such that $\tp(a/C)$ is nonisolated. Then, by almost $\aleph_0$-categoricity, $p=\tp(a)$ is nonisolated. Since $C$ is a base, we have $p\nwor \tp(C)$, so Lemma \ref{Lemma ALMOST isolation of extensionss}(a) applies: $\tp(a/C)$ is a nonforking extension of $p=\tp(a)$. There are two possibilities: $\tp(a/C)=((p,<)_{l})_{\strok C}$ and $\tp(a/C)=((p,<)_r)_{\strok C}$. In the first case, by Lemma \ref{Lemma ALMOST isolation of extensionss}(b), the pair $(\tp(a/C),<)$ is right-definable so, since $\tp(a/C)$ is nonisolated, it is of kind (nd,d), in which case $\mathfrak p=(p,>)_r\in\mathcal R_M$ and $\tp(a/C)=\mathfrak p_{\strok C}$. Similarly, in the second case, $\mathfrak p=(\tp(a),<)_r\in\mathcal R_M$ and $\tp(a/C)=\mathfrak p_{\strok C}$.

(b) Let $\mathfrak p=(p,\lessdot)_r\in\mathcal R_M$. Then $\tp(C)\nwor p$. Assume that $\lessdot$ equals $<$; the proof when $\lessdot$ equals $>$ is similar. Note that $\mathfrak p$ is not definable as $\mathfrak p\in \mathcal P$, so $\mathfrak p=(\mathfrak p_C)_r$ implies that the pair $\mathfrak p_C$ is not right-definable. 
Since, by Lemma \ref{Lemma ALMOST isolation of extensionss}(b), $\mathfrak p_C$ is left-definable, it is of kind (d,nd).
Hence, $\mathfrak p_{\strok C}$ is of mixed kind, so it is trivial by Proposition \ref{Proposition mixed kind is trivial}. 

(c) First, assume $\mathfrak p\nwor \mathfrak q$. Then $\mathfrak p_{\strok C}\wor\mathfrak q_{\strok C}$ is impossible since by \cite[Theorem 5.10(c)]{MTwqom1}, $\wor$ of trivial types transfers to their nonforking extensions. Hence, $\mathfrak p_{\strok C}\nwor \mathfrak q_{\strok C}$, so one of $\delta_C(\mathfrak p_C,\mathfrak q_C)$ and $\delta_C(\mathfrak p_C,\mathfrak q_C^*)$ is true. By (b), the pairs $\mathfrak p_C$ and $\mathfrak q_C$ are of kind (d,nd). Clearly, $\mathfrak q^*_C$ is of kind (nd,d), so $\mathfrak p_C$ and $\mathfrak q_C^*$ have different kinds, and hence $\delta_C(\mathfrak p_C,\mathfrak q_C^*)$ is impossible by Fact \ref{Fact direct no preserves the kind}. We conclude $\delta_C(\mathfrak p_C,\mathfrak q_C)$.

Now, assume $\mathfrak p\wor \mathfrak q$. Toward a contradiction, assume $\mathfrak p_{\strok C}\nwor \mathfrak q_{\strok C}$. Then since $\mathfrak p_C$ and $\mathfrak q_C$ are of kind (d,nd), arguing as in the previous paragraph, we conclude $\delta_C(\mathfrak p_C,\mathfrak q_C)$. By \cite[Corollary 5.11(a)]{MTwqom1}, this, together with the triviality of $\mathfrak p_{\strok C}$ and $\mathfrak q_{\strok C}$, implies $\mathfrak p\nwor \mathfrak q$. Contradiction. 

(d) By way of contradiction, suppose that $(\mathfrak p_n\in\mathcal R_M\mid n<\omega)$, are pairwise weakly orthogonal. Without loss of generality, by passing to a subsequence and possibly reversing the order, we can assume that $\mathfrak p_n=(p_n,<)_r$, where $p_n\in S_1(C)$.  
By (b), each $p_n$ is nonisolated and trivial, and by (c), $(p_n\mid n\in\omega)$ are pairwise weakly orthogonal. This contradicts Corollary \ref{Corollary no strongly orthogonal sequence}(c).

(e) Suppose that $\mathfrak p_1,\dots,\mathfrak p_n$ are $\nwor$-representatives in $\mathcal B_M$. By (b), each $(\mathfrak p_i)_{\strok C}$ is nonisolated, so if $C$ is a strong base, it must be omitted in $M$ as $\mathfrak p_i\in\mathcal B_M$. 
For the reverse implication, suppose that each $(\mathfrak p_i)_{\strok C}$ is omitted in $M$ and we show that $C$ is a strong base of $M$. So suppose that $\tp(a/D)$ is nonisolated, where $a\in M$ and $D\supseteq C$ is finite. By (a) and the fact that $D$ is a base of $M$, $\tp(a/D)=\mathfrak p_{\strok D}$ for some $\mathfrak p\in\mathcal R_M$. 
It remains to prove $\mathfrak p\in \mathcal U_M$. By way of contradiction, suppose that $\mathfrak p\in \mathcal B_M$; then $\mathfrak p\nwor \mathfrak p_i$ for some $i$. By (b) and (c), both $\mathfrak p_{\strok D}$ and $(\mathfrak p_i)_{\strok D}$ are trivial, and $\delta_D(\mathfrak p_D,(\mathfrak p_i)_D)$. Since $D\triangleleft^{\mathfrak p_D}a$, by the density property \cite[Corollary 2.13(b)]{MTwqom1}, there is $b\in\Mon$ such that $D\triangleleft^{(\mathfrak p_i)_D}b \triangleleft^{\mathfrak p_D}a$. By (b), $(\mathfrak p_i)_D$ is of kind (d,nd) and by (c), $\tp(a/D)\nwor(\mathfrak p_i)_{\strok D}$, so by Lemma \ref{Lemma ALMOST isolation of extensionss} the only nonisolated extension of $(\mathfrak p_i)_{\strok D}$ over $Da$ is $(\mathfrak p_i)_{\strok Da}$. Together with $D\triangleleft^{(\mathfrak p_i)_D}b \triangleleft^{\mathfrak p}a$ this implies that $\tp(b/Da)$ must be isolated and therefore realized in $M$. Then $(\mathfrak p_i)_{\strok D}\subseteq \tp(b/Da)$ is realized in $M$. This contradicts the assumption that $(\mathfrak p_i)_{\strok C}$ is omitted in $M$.

(f) Let $\mathfrak p_1,\dots,\mathfrak p_n$ be a set of representatives of all $\nwor$-classes in $\mathcal B_M$. Choose a finite $D\subseteq C$ such that each $(\mathfrak p_i)_{\strok D}$ is omitted in $M$; this is possible by the definition of $\mathcal B_M$. By (e), $D$ is a strong base of $M$. 

(g) Let $N\models T$ be countable and saturated. Then $\mathcal P=\mathcal R_N$, so by (d), $\mathcal P$ has finitely many $\nwor$-classes.
\end{proof}

\begin{proof}[Proof of Theorem \ref{Theorem_AAC}]
    Suppose that $M,N\models T$ are countable models such that $M\equiv_{L_1(T)}N$. Then 
    $$\mathfrak p\in\mathfrak R_M \ \Leftrightarrow \  M\models (\exists x)\ \mathfrak p_{\strok\emptyset}(x) \ \Leftrightarrow \ N\models (\exists x)\ \mathfrak p_{\strok\emptyset}(x) \ \Leftrightarrow \ \mathfrak p\in\mathcal R_N$$ 
    proves $\mathcal R_M=\mathcal R_N$.  Next, we prove $\mathcal B_M=\mathcal B_N$. Let $\mathfrak p\in\mathcal B_M$. Then there is a finite $C\subseteq M$ such that $\mathfrak p_{\strok C}$ is omitted in $M$. Let $q(X)=\tp(C)$, and let $r(C,y)=\mathfrak p_{\strok C}(y)$; note that for $C'\models q$, $r(C',y)=\mathfrak p_{\strok C'}(y)$ by $0$-invariance of $\mathfrak p$. Since $M\models (\exists X)(q(X)\land\lnot(\exists y)\ r(X,y))$, $N\models(\exists X)(q(X)\land\lnot(\exists y)\ r(X,y))$ as well, so there is $D\subseteq N$ such that $D\models q$, and $r(D,y)=\mathfrak p_{\strok D}(y)$ is omitted in $N$. Thus, $\mathfrak p\in\mathcal B_N$, and $\mathcal B_M\subseteq\mathcal B_N$ follows. By a symmetric argument we get $\mathcal B_M=\mathcal B_N$. As a corollary, we conclude $\mathcal U_M=\mathcal U_N$.

Let $C_M\subseteq M$ be a strong base of $M$; it exists by Lemma \ref{Lemma ALMOST sva priprema}(f). We {\it claim} that there is $D_N$, a strong base for $N$, with $C_M\equiv D_N$. By Lemma \ref{Lemma ALMOST sva priprema}(d), $\mathcal R_M$ has finitely many $\nwor$-classes. Choose $\mathfrak p_1,\dots,\mathfrak p_n\in\mathcal B_M$ representing all $\nwor$-classes of $\mathcal B_M$. Since $C_M$ is a strong base, all $(\mathfrak p_i)_{\strok C_M}$ are omitted in $M$ by Lemma \ref{Lemma ALMOST sva priprema}(e), so we have:
\setcounter{equation}{0}
    \begin{equation} 
    M\models(\exists X)\left(q(X)\land\bigwedge_{i=1}^n\lnot(\exists y)\ r_i(X,y)\right),
    \end{equation}
where, as in the previous paragraph, $q(X)=\tp(C_M)$ and $r_i(C_M,y)=(\mathfrak p_i)_{\strok C_M}(y)$. The same $L_1(T)$-sentence holds in $N$, so there is some $D_N\subseteq N$ such that $D_N\models q$ and all $(\mathfrak p_i)_{\strok D_N}$ are omitted in $N$. We show that $D_N$ is a strong base for $N$: Since $C_M$ is a base of $M$, $\mathfrak p_{\strok \emptyset}\nwor \tp(C_M)$ holds for all $\mathfrak p\in \mathcal R_M$, so $\mathcal R_M=\mathcal R_N$ and $C_M\equiv D_N$ together imply that $D_N$ is a base of $N$. Furthermore, note that $\mathfrak p_1,\dots,\mathfrak p_n$ are representatives of $\nwor$-classes in $\mathfrak B_N$ as $\mathfrak B_M=\mathfrak B_N$ and, by (1) and $M\equiv_{L_1(T)}N$, each $(\mathfrak p_i)_{\strok D_N}$ is omitted in $N$. By Lemma \ref{Lemma ALMOST sva priprema}(e), this implies that $D_N$ is a strong base of $N$, proving the claim.

    We now prove $M\cong N$ by back-and-forth. Let $f_0$ be a partial (elementary) isomorphism mapping $C_M$ to $D_N$. Suppose that $f:C'\mapsto D'$ is a finite partial isomorphism that extends $f_0$. We prove that for all $a\in M$ we can extend $f$ to $C'a$. There are two cases to consider. 

    \smallskip\noindent {\em Case 1.} $\tp(a/C')$ is isolated, say, by  $\phi(x,C')$. Then $\phi(x,D')$ also isolates a type, so we can find $b\in N$ such that $\models\phi(b,D')$. Setting $f(a)=b$, we clearly extend $f$ to a finite partial isomorphism.

    \smallskip\noindent {\em Case 2.} $\tp(a/C')$ is nonisolated. Since $C_M$ is a strong base, $\tp(a/C')=\mathfrak p_{\strok C'}$ for some $\mathfrak p\in\mathcal U_M$. Note that $\mathfrak p\in\mathcal U_N$ as $\mathcal U_M=\mathcal U_N$, so, in particular, $\mathfrak p_{\strok D'}$ is realized in $N$ by, say, $b\in N$. By $0$-invariance of $\mathfrak p$, again by setting $f(a)=b$ we extend $f$ to a finite partial isomorphism.

\smallskip
    By a symmetric argument, for $b\in N$ we can find some $a\in M$ such that $f(a)=b$ extends $f$ to a finite partial isomorphism. Thus, the back-and-forth goes through, and the proof of Theorem \ref{Theorem_AAC} is finished. 
\end{proof}

\begin{proof}[{\it Proof of Theorem \ref{Theorem_bin_ros_quasio}.}]
Suppose that $T$ is a weakly quasi-o-minimal theory with an additional feature: binary, rosy, quasi-o-minimal, or has a finite convexity rank. Assume that $I(T,\aleph_0)<2^{\aleph_0}$. By Corollary \ref{Cor_few_ros_ finconvquasi bin_are_binary} $T$ is binary, so by \cite[Lemma 8.1] {MT} $T$ is almost $\aleph_0$-categorical. By Theorem \ref{Theorem_AAC} Martin's conjecture holds for $T$.
\end{proof}

\section{Concluding remarks}\label{Section 7}

Let $T$ be a weakly quasi-o-minimal theory. 
If $T$ is also binary, then all complete 1-types (over any domain) are trivial. It is interesting to know if the converse is true. The answer to this question is negative if we relax the assumption to: every complete 1-type over a finite subset of $\Mon$ is trivial, as \cite[Example 7.11]{MTwqom1} shows.
However, by Proposition \ref{Prop_equivalents_of_R}, the answer is positive for theories with simple semiintervals. 

\begin{Question}
If a weakly quasi-o-minimal theory has all complete 1-types (over any domain) trivial, must it be binary?
\end{Question}

For a weakly quasi-o-minimal theory $T$ with few countable models, we know that every type $p\in S_1(A)$ is convex and has simple semiintervals. For further understanding of weak non-orthogonality, weakly o-minimal types in $T^\mathrm{eq}$ may prove useful; for example, in explaining when a forking extension of $p$ is $\nwor$ to a nonforking extension of some $q\in S_1(A)$. So, it would be quite convenient that the answer to the following two general questions is affirmative even for types in $T^\mathrm{eq}$. 

\begin{Question}\label{Question_womtype has simple ssi}
Does every weakly o-minimal type over a finite domain in a first-order theory with few countable models have simple semiintervals?
\end{Question}

According to Theorem \ref{Theorem_convex}, a positive answer to this question would also answer the following question in the affirmative.

\begin{Question}\label{Question is wom type convex}
Is every weakly o-minimal type over a finite domain in a first-order theory with few countable models convex?    
\end{Question}

For the rest of this section, let $T$ be countable and weakly quasi-o-minimal with few countable models. Denote
\begin{itemize}
    \item[ ] $\mathcal P$ -- the collection of all global 1-types that are invariant over some finite parameter set but are not definable.
\end{itemize}
Analyzing the proof of Theorem \ref{Theorem_AAC} leads to the following two conjectures; we believe that their confirmation would lead to the confirmation of Martin's conjecture for $T$.  

\begin{Conjecture}
$\mathcal P$ has finitely many $\nwor$-classes.    
\end{Conjecture}

\begin{Conjecture}\label{C2}
For every $\mathfrak p\in\mathcal P$ there is a finite $A$ such that $\mathfrak p$ is $A$-invariant and $\mathfrak p_{\strok A}$ is trivial.
\end{Conjecture}
 
We state two equivalents of Conjecture \ref{C2}.  
Say that a weakly o-minimal type is eventually trivial (eventually simple) if some nonforking extension of $p$ over a finitely extended domain is trivial (simple).   

\smallskip 
  {\bf (2')} \  Every complete nonisolated 1-type over a finite domain is eventually trivial.  

\smallskip 
 {\bf  (2'')} \ Every complete nonisolated 1-type over a finite domain is eventually simple.  

\medskip
Here, (2') is easily seen to be equivalent to Conjecture \ref{C2}, and (2')$\Rightarrow$(2'') follows by   \cite[Theorem 5.22]{MTwqom1}, which states that every trivial weakly o-minimal type over a finite domain in a theory with few countable models is simple. For (2'')$\Rightarrow$(2')  assume that $p\in S_1(A)$ is nonisolated and eventually simple and let $\mathbf p=(p,<)$ be a weakly o-minimal pair over $A$. If $p$ is of mixed kind, then it is trivial by Proposition \ref{Proposition mixed kind is trivial} and we are done. So, assume that it is not of mixed kind; then it is of kind (nd,nd). 
Let $B\supseteq A$ be such that $(\mathbf p_r)_{\strok B}$ is simple, and let $a\models (\mathbf p_r)_{\strok B}$. Since $(\mathbf p_r)_{\strok B}$ is simple,  the pair $((\mathbf p_r)_{\strok Ba},<)$ is left-definable, so it is of kind (d,nd); $(\mathbf p_r)_{\strok B}$ is trivial by Proposition \ref{Proposition mixed kind is trivial}.

\printbibliography

@article{Alibek,
  title={Discrete order on a definable set and the number of models},
  author={Alibek, Aida A and Baizhanov, Bektur S and Zambarnaya, Tanya S},
  journal={Mathematical Journal},
  volume={14},
  number={3},
  pages={5--13},
  year={2014}
}

@article{Altayeva2020,
  title={Binarity of almost $\omega$-categorical quite o-minimal theories},
  author={Altayeva, Aizhan B and Kulpeshov, Beibut Sh},
  journal={Siberian Mathematical Journal},
  volume={61},
  number={3},
  pages={379--390},
  year={2020},
  publisher={Springer}
}

@incollection{Baizhanov1999,
  title={Orthogonality of one-types in weakly o-minimal theories},
  author={Baizhanov, Bektur S},
  booktitle={Algebra and Model Theory {II}},
  pages={5--28},
  year={1999},
  editor={A.G. Pinus and K.N. Ponomaryov},
  publisher={Novosibirsk State Technical University},
  adress={Novosibirsk}
}

@incollection{Baizhanov2006,
  title={On behaviour of 2-formulas in weakly o-minimal theories},
  author={Baizhanov, Bektur S and Kulpeshov, Beibut Sh},
  booktitle={Proceedings of the 9th Asian conference},
  pages={31--40},
  year={2006},
  publisher={World Scientific},
  doi = {10.1142/9789812772749_0003},
}

@article{Baizhanov2025,
title = {Constructing models of small ordered theories with maximal countable spectrum},
author={Baizhanov, Bektur S and Zambarnaya, Tanya S},
journal = {Annals of Pure and Applied Logic},
volume = {177},
number = {2},
pages = {103659},
year = {2026}
}

@article{ealy2007characterizing,
  title={Characterizing rosy theories},
  author={Ealy, Clifton and Onshuus, Alf},
  journal={The Journal of Symbolic Logic},
  volume={72},
  number={3},
  pages={919--940},
  year={2007},
  publisher={Cambridge University Press},
  doi={10.2178/jsl/1191333848}
}

@article{Herwig,
  title={On $\aleph_0$-categorical weakly o-minimal structures},
  author={Herwig, Bernhard and Macpherson, Dugald H and Martin, Gary A and Nurtazin, Abyz T and Truss, John K},
  journal={Annals of Pure and Applied Logic},
  volume={101},
  number={1},
  pages={65--93},
  year={1999},
  publisher={Elsevier},
  doi = {10.1016/S0168-0072(99)00029-9},
}

@article{Ikeda,
  title={On theories having three countable models},
  author={Ikeda, Koichiro and Tsuboi, Akito and Pillay, Anand},
  journal={Mathematical Logic Quarterly},
  volume={44},
  number={2},
  pages={161--166},
  year={1998},
  publisher={Wiley Online Library},
  doi={10.1002/malq.19980440203}
}

@article{kulpeshov1998,
  title={Weakly o-minimal structures and some of their properties},
  author={Kulpeshov, Beibut Sh},
  journal={The Journal of Symbolic Logic},
  volume={63},
  number={4},
  pages={1511--1528},
  year={1998},
  publisher={Cambridge University Press},
  doi={10.2307/2586664}
}

@article{Kulpeshov2007,
  title={Criterion for binarity of $\aleph_0$-categorical weakly o-minimal theories},
  author={Kulpeshov, Beibut Sh},
  journal={Annals of Pure and Applied Logic},
  volume={145},
  number={3},
  pages={354--367},
  year={2007},
  publisher={Elsevier},
  doi = {10.1016/j.apal.2006.10.004},
}

@article{Kulpeshov2020,
  title={Vaught’s conjecture for weakly $o$-minimal theories of finite convexity rank},
  author={Kulpeshov, Beibut Sh},
  journal={Izvestiya: Mathematics},
  volume={84},
  number={2},
  pages={324--347},
  year={2020},
  publisher={IOP Publishing},
  doi={10.1070/IM8894}
}

@article{Kulpeshov2021,
  title={A Criterion for Binarity of Almost $\omega$-Categorical Weakly o-Minimal Theories},
  author={Kulpeshov, Beibut Sh},
  journal={Siberian Mathematical Journal},
  volume={62},
  number={6},
  pages={1063--1075},
  year={2021},
  publisher={Springer},
  doi={10.1134/S0037446621060082}
}

@article{Mac,
  title={Weakly o-minimal structures and real closed fields},
  author={Macpherson, Dugald and Marker, David and Steinhorn, Charles},
  journal={Transactions of the American Mathematical Society},
  volume={352},
  number={12},
  pages={5435--5483},
  year={2000},
  doi={10.1090/S0002-9947-00-02633-7}
}

@article{Marker1986omitting,
  title={Omitting types in o-minimal theories},
  author={Marker, David},
  journal={The Journal of Symbolic Logic},
  volume={51},
  number={1},
  pages={63--74},
  year={1986},
  publisher={Cambridge University Press},
  doi={10.2307/2273943}
}

@article{MT,
  title={Stationarily ordered types and the number of countable models},
  author={Moconja, Slavko and Tanovi{\'c}, Predrag},
  journal={Annals of Pure and Applied Logic},
  volume={171},
  number={3},
  pages={102765},
  year={2020},
  publisher={Elsevier},
  doi = {10.1016/j.apal.2019.102765},
}

@article{MTwom,
  title={Weakly o-minimal types},
  author={Moconja, Slavko and Tanovi{\'c}, Predrag},
  journal={Annals of Pure and Applied Logic},
  volume = {176},
  number = {9},
  pages={103605},
  year={2025},
  doi = {https://doi.org/10.1016/j.apal.2025.103605},
}

@misc{MTwqom1,
      title={Countable models of weakly quasi-o-minimal theories {I}}, 
      author={Moconja, Slavko and Tanovi{\'c}, Predrag},
      eprint={2412.20589},
      archivePrefix={arXiv},
      url={https://arxiv.org/abs/2412.20589}, 
      year={2024}
}

@article{Newelski,
  title={Very simple theories without forking},
  author={Newelski, Ludomir},
  journal={Archive for Mathematical Logic},
  volume={42},
  number={6},
  pages={601--616},
  year={2003},
  doi={10.1007/s00153-003-0172-4}
}

@article{Rast,
  title={The Borel complexity of isomorphism for o-minimal theories},
  author={Rast, Richard and Sahota, Davender Singh},
  journal={The Journal of Symbolic Logic},
  volume={82},
  number={2},
  pages={453--473},
  year={2017},
  publisher={Cambridge University Press},
  doi={10.1017/jsl.2017.17}
}

\end{document}